\documentclass[12pt]{ociamthesis}

  \usepackage{amssymb}
  \usepackage{latexsym}
  \usepackage{amsfonts}
  \usepackage{amsthm}
  \usepackage{amsmath}
  \usepackage{amscd}
  \usepackage{enumerate}
  \usepackage{hyperref}
  \usepackage{xypic}
\usepackage[svgnames]{xcolor}
\usepackage{tikz} 
\usetikzlibrary{matrix,positioning,
    calc,
    arrows,
    decorations.markings,
    decorations.pathreplacing,
    backgrounds,
    patterns,
    shapes}
\usepackage{cite}
\usepackage[all,cmtip]{xy}
\usepackage{multirow}

\newcommand{\mc}{\mathcal}

\newcommand{\rr}{{\mathbb R}}

\newcommand{\nn}{{\mathbb N}}

\renewcommand{\c}{{\mathcal C}}

\renewcommand{\a}{\alpha}
\renewcommand{\b}{\beta}
\renewcommand{\d}{\delta}
\newcommand{\s}{\sigma}
\renewcommand{\S}{\Sigma}

\def\ba#1\ea{\begin{align*}#1\end{align*}}
\newcommand{\bcd}{\begin{CD}}
\newcommand{\ecd}{\end{CD}}

\def\bsm#1\esm{\left(\begin{smallmatrix}#1\end{smallmatrix}\right)}
\def\bm#1\em{\left(\begin{matrix}#1\end{matrix}\right)}

\newcommand{\lra}{\longrightarrow}

\newcommand{\ot}{\otimes}

\newcommand{\ob}{\mathrm{Ob}}

\newcommand{\mor}{\mathrm{Mor}}

\newcommand{\set}{\mathbf{Set}}

\newcommand{\rel}{\mathbf{Rel}}

\newcommand{\meas}{\mathbf{Meas}}
\newcommand{\stoch}{\mathbf{Stoch}}
\newcommand{\finstoch}{\mathbf{FinStoch}}
\newcommand{\cgst}{\mathbf{CGStoch}}

\newcommand{\pa}{\mathrm{pa}}

\renewcommand{\i}{^{-1}}
\newcommand{\id}{\mathrm{id}}

  \newtheorem{thm}{Theorem}[chapter]
  \newtheorem{lem}[thm]{Lemma}
  \newtheorem{prop}[thm]{Proposition}

  \theoremstyle{definition}
  \newtheorem{defn}[thm]{Definition}
  \newtheorem{defns}[thm]{Definitions}

  \theoremstyle{remark} 
  \newtheorem{ex}[thm]{Example}
  \newtheorem{exs}[thm]{Examples}
  
  \newtheorem{remark}[thm]{Remark}

\pgfdeclarelayer{edgelayer}
\pgfdeclarelayer{nodelayer}
\pgfsetlayers{edgelayer,nodelayer,main}

\tikzstyle{none}=[inner sep=0pt]
\tikzstyle{circ}=[circle,fill=White,draw=Black]
\tikzstyle{sq}=[rectangle,fill=White,draw=Black,inner sep=1pt,minimum width=15pt,minimum height=15pt]
\tikzstyle{dot}=[circle,fill=Black,draw=Black,inner sep=1pt]
\tikzstyle{tri}=[isosceles triangle, draw, shape border rotate=270, isosceles triangle apex angle=80, inner xsep=0pt, inner ysep=1pt]

\newcommand{\idline}{\begin{aligned}
\begin{tikzpicture}[scale=.3]
	\begin{pgfonlayer}{nodelayer}
		\node [style=none] (0) at (0, -0.5) {};
		\node [style=none] (1) at (0, 0.75) {};
	\end{pgfonlayer}
	\begin{pgfonlayer}{edgelayer}
		\draw (0.center) to (1.center);
	\end{pgfonlayer}
\end{tikzpicture}
\end{aligned}}

\newcommand{\comult}{\begin{aligned}
\begin{tikzpicture}[scale=.3]
	\begin{pgfonlayer}{nodelayer}
		\node [style=none] (0) at (-0.5, 0.5) {};
		\node [style=none] (1) at (0.5, 0.5) {};
		\node [style=none] (2) at (0, 0) {};
		\node [style=none] (3) at (0, -0.5) {};
		\node [style=none] (4) at (-0.5, 0.75) {};
		\node [style=none] (5) at (0.5, 0.75) {};
	\end{pgfonlayer}
	\begin{pgfonlayer}{edgelayer}
		\draw (4.center) to (0.center);
		\draw (5.center) to (1.center);
		\draw [in=0, out=-90, looseness=0.75] (1.center) to (2.center);
		\draw [in=180, out=-90, looseness=0.75] (0.center) to (2.center);
		\draw (2.center) to (3.center);
	\end{pgfonlayer}
\end{tikzpicture}
\end{aligned}}

\newcommand{\counit}{\begin{aligned}
\begin{tikzpicture}[scale=.3]
	\begin{pgfonlayer}{nodelayer}
		\node [style=dot] (0) at (0, 0.3) {};
		\node [style=none] (1) at (0, -0.5) {};
		\node [style=none] (2) at (0, 0.75) {};
	\end{pgfonlayer}
	\begin{pgfonlayer}{edgelayer}
		\draw (0) to (1.center);
	\end{pgfonlayer}
\end{tikzpicture}
\end{aligned}}

\newcommand{\swa}{\begin{aligned}
\begin{tikzpicture}[scale=.3]
	\begin{pgfonlayer}{nodelayer}
		\node [style=none] (1) at (-0.5, 1) {};
		\node [style=none] (2) at (0.5, 1) {};
		\node [style=none] (3) at (0.5, 0) {};
		\node [style=none] (4) at (-0.5, 0) {};
	\end{pgfonlayer}
	\begin{pgfonlayer}{edgelayer}
		\draw [in=90, out=-90, looseness=0.75] (1.center) to (3.center);
		\draw [in=90, out=-90, looseness=0.75] (2.center) to (4.center);
	\end{pgfonlayer}
\end{tikzpicture}
\end{aligned}}

\newcommand{\mult}{\begin{aligned}
\begin{tikzpicture}[scale=.3]
	\begin{pgfonlayer}{nodelayer}
		\node [style=none] (0) at (-0.5, -0.25) {};
		\node [style=none] (1) at (0.5, -0.25) {};
		\node [style=none] (2) at (0, 0.25) {};
		\node [style=none] (3) at (0, 0.75) {};
		\node [style=none] (4) at (-0.5, -0.5) {};
		\node [style=none] (5) at (0.5, -0.5) {};
	\end{pgfonlayer}
	\begin{pgfonlayer}{edgelayer}
		\draw (4.center) to (0.center);
		\draw (5.center) to (1.center);
		\draw [in=0, out=90, looseness=0.75] (1.center) to (2.center);
		\draw [in=180, out=90, looseness=0.75] (0.center) to (2.center);
		\draw (2.center) to (3.center);
	\end{pgfonlayer}
\end{tikzpicture}
\end{aligned}}

\newcommand{\unit}{\begin{aligned}
\begin{tikzpicture}[scale=.3]
	\begin{pgfonlayer}{nodelayer}
		\node [style=dot] (0) at (0, -0.5) {};
		\node [style=none] (1) at (0, 0.5) {};
	\end{pgfonlayer}
	\begin{pgfonlayer}{edgelayer}
		\draw (0) to (1.center);
	\end{pgfonlayer}
\end{tikzpicture}
\end{aligned}}

\title{Causal Theories:  \\ 
       A Categorical Perspective  \\
       on Bayesian Networks}   
\author{Brendan Fong}     
\college{Wolfson College}
\degree{Master of Science in Mathematics and the Foundations of Computer Science} 
\degreedate{Trinity 2012}    

\begin{document}
\baselineskip=18pt plus1pt
\setcounter{secnumdepth}{3} 
\setcounter{tocdepth}{3}

\maketitle                 
\begin{dedication}
\end{dedication}
 
\begin{acknowledgements}

It's been an amazing year, and I've had a good time learning and thinking about the contents of this essay. A number of people have had significant causal influence on this.

Foremost among these is my dissertation supervisor Jamie Vicary, who has been an excellent guide throughout, patient as I've jumped from idea to idea and with my vague questions, and yet careful to ensure I've stayed on track. We've had some great discussions too, and I thank him for them. John Baez got me started on this general topic, has responded enthusiastically and generously to probably too many questions, and, with the support of the Centre for Quantum Technologies, Singapore, let me come visit him to pester him with more. Bob Coecke has been a wonderful and generous general supervisor, always willing to talk and advise, and has provided many of the ideas that lurk in the background of those here. I thank both of them too. I also thank Rob Spekkens, Dusko Pavlovic, Prakash Panangaden, and Samson Abramsky for some interesting discussions and detailed responses to my queries.

For help of perhaps a less technical nature, but still very much appreciated, I thank Kati, Ross, and the rest of the MFoCS crowd, as well as Fiona, Jono, Daniel, and Shiori, for being happy to listen to me rant confusedly, eat cake, and generally making life better. And finally thank you to my brothers and my parents for their constant support and understanding, even as I've been far away.

\end{acknowledgements}  
\begin{abstract}

In this dissertation we develop a new formal graphical framework for causal reasoning. Starting with a review of monoidal categories and their associated graphical languages, we then revisit probability theory from a categorical perspective and introduce Bayesian networks, an existing structure for describing causal relationships. Motivated by these, we propose a new algebraic structure, which we term a causal theory. These take the form of a symmetric monoidal category, with the objects representing variables and morphisms ways of deducing information about one variable from another. A major advantage of reasoning with these structures is that the resulting graphical representations of morphisms match well with intuitions for flows of information between these variables. These categories can then be modelled in other categories, providing concrete interpretations for the variables and morphisms. In particular, we shall see that models in the category of measurable spaces and stochastic maps provide a slight generalisation of Bayesian networks, and naturally form a category themselves. We conclude with a discussion of this category, classifying the morphisms and discussing some basic universal constructions.

\end{abstract} 

\begin{romanpages} 
\tableofcontents
\end{romanpages}

\phantomsection
\addcontentsline{toc}{chapter}{Introduction}
\chapter*{Introduction}


From riding a bicycle to buying flowers for a friend, causal relationships form a basic and ubiquitous framework informing, at least informally, how we organise, reason about, and choose to interact with the world. It is perhaps surprising then that ideas of causality often are entirely absent from our formal scientific frameworks, whether they directly be models of the world, such as theories of physics, or methods for extracting information from data, as in the case of statistical techniques. It is the belief of the author that there remains much to be gained from formalising our intuitions regarding causality.

Indeed, taking the view that causal relationships are fundamental physical facts about the world, it is interesting to discover and discuss these in their own right. Even if one views causality only as a convenient way to organise information about dependencies between variables, however, it is hard to see how introducing such notions into formal theories, rather than simply ignoring these intuitions, will not benefit at least some of them. The artificial intelligence community gives a tangible example of this, with the widespread use of Bayesian networks indicating that causal relationships provide a far more efficient way to encode, update, and reason with information about random variables then simply working with the entire joint variable.

In what follows we lay out the beginnings of a formal framework for reasoning about causality. We shall do this by extending the aforementioned existing ideas for describing causal relationships between random variables through the use of category theory, and in particular the theory of monoidal categories.


\subsection*{Overview of the literature}

More precisely, in this dissertation we aim to bridge three distinct ideas. The first is the understanding of probability theory and probabilistic processes from a categorical perspective. For this we work with a category first defined by Lawvere a half-century ago in the unpublished manuscript \cite{L}, in which the objects are sets equipped with a $\s$-algebra and the morphisms specify a measure of the codomain for each element of the domain, subject to a regularity condition. These ideas were later developed, in 1982, in a short paper by Giry \cite{Gi}, and have been further explored by Doberkat \cite{Do}, Panangaden \cite{Pa}, and Wendt \cite{We}, among others, in recent years. 

The second idea is that of using graphical models to depict causal relationships. Termed Bayesian networks, these were first discussed predominantly in the machine learning community in the 1980s, with the seminal work coming in the influential book Pearl \cite{P2}. Since then Bayesian networks have been used extensively to discuss causality from both computational and philosophical perspectives, as can be seen in recent books Pearl \cite{P} and Williamson \cite{Wi}.

The third body of work, which will serve as a framework to unite the above two ideas, is the theory of monoidal categories and their graphical calculi. An introductory exposition of monoidal categories can be found in Mac Lane \cite{ML}, while the survey by Selinger \cite{Se} provides an excellent overview of the graphical ideas. Here our work is in particular influenced by that of the very recent paper by Coecke and Spekkens \cite{CS}, which uses monoidal categories to picture Bayesian inference, realising Bayesian inversion as a compact structure on an appropriate category. 

\subsection*{Outline}

Since they serve as the underlying framework for this thesis, we begin with a chapter reviewing the theory of monoidal categories, the last of the above ideas. We conclude this first chapter by discussing how the idea of a monoid can be generalised through a category we call the `theory of monoids', with monoids themselves being realised as monoidal functors from this category into the category $\set$ of sets and functions, while `generalised monoids' take the form of monoidal functors from the theory of monoids into other categories. It is in this sense the `causal theories' that we will define are theories. In Chapter 2 we turn our attention to reviewing the basic ideas of measure theoretic probability theory from a categorical viewpoint. Here we pay particular attention to the category, which will shall call $\stoch$, defined by Lawvere. Chapter 3 then provides some background on Bayesian networks, stating a few results about how they capture causal relationships between random variables through ideas of conditional independence. This motivates the definition of a causal theory, which we present in Chapter 4. Following our exploration of these categories and how to represent their morphisms graphically, we turn our attention to their models.  Although models in $\set$ and $\rel$ are interesting, we spend most of the time discussing models in Lawvere's category $\stoch$. This chapter is concluded by a discussion of confounding variables and Simpson's paradox, where we see some of the strengths of causal theories and in particular their graphical languages. In short, we show we can take the directed graph structure of a Bayesian network more seriously than just a suggestive depiction of dependencies. In the final chapter, Chapter 5, we discuss some properties of the category of stochastic causal models of a fixed causal theory. These are models in a certain full subcategory of $\stoch$ that omits various pathological eventuations, and have a close relationship with Bayesian networks.

\subsection*{New contributions}

The main contribution of this dissertation is the presentation of a new algebraic structure: causal theories. These are a type of symmetric monoidal category, and we will discuss how these capture the notion of causal relationships between variables; deterministic, possibilistic and probabilistic models of these categories; how their graphical calculi provide intuitive representations of reasoning and information flow; and some of the structure of the category of probabilistic models. In doing so we also move the discussion of Bayesian networks from the finite setting preferred by the computationally-focussed Bayesian network community to a more general setting capable of handling non-discrete probability spaces.

In particular, I claim all results of Chapters 4 and 5 as my own as well as, except for Proposition \ref{prop.detfunc}, the discussion of deterministic stochastic maps of Section 2.4.

\chapter{Preliminaries on Monoidal Categories}

Our aim is to explore representations of causal relationships between random variables from a categorical perspective. In this chapter we lay the foundations for this by introducing the basic language we will be working with---the language of monoidal categories---and, by way of example, informally discussing the notion of a `theory'.

We begin with a review of the relevant notions from category theory. Recall that a \emph{category} $\mc C$ consists of a collection $\ob \mc C$ of \emph{objects}, for each pair $A,B$ of objects a set $\mor(A,B)$ of \emph{morphisms}, and for each triple $A,B,C$ of objects a function, or \emph{composition rule}, $\mor(A,B) \times \mor(B,C) \to \mor(A,C)$, such that the composition rule is associative and obeys a unit law. We shall write $A \in \mc C$ if $A$ is an object of the category $\mc C$, and $f$ in $\mc C$ if $f$ is a morphism in the category $\mc C$. As we shall think of them, categories are the basic algebra structure capturing the idea of composable processes, with the objects of a category different systems of a given type, and the morphisms processes transforming one system into another. 

We further remind ourselves that a \emph{functor} is a map from one category to another preserving the composition rule, that under a mild size constraint the collection of categories itself forms a category with functors as morphisms, and that in this category products exist. Moreover, the set of functors between any two categories itself has the structure of a category in a standard, nontrivial way, and we call the morphisms in this category \emph{natural transformations}, with the invertible ones further called \emph{natural isomorphisms}. Two categories are \emph{equivalent} if there exists a functor in each direction between the two such that their compositions in both orders are naturally isomorphic to the identity functor.
\newpage

The reader seeking more detail is referred to Mac Lane \cite{ML}, in particular Chapters I and II. In general our terminology and notation for categories will follow the conventions set out there.

\section{Monoidal categories}

The key structure of interest to us in the following is that of a symmetric monoidal category. A monoidal category is a category with two notions of composition---ordinary categorical composition and the monoidal composition---, and symmetric monoidal categories may be thought of as the algebraic structure of processes that may occur simultaneously as well as sequentially. These categories are of special interest as they may be described precisely with a graphical notation possessing a logic that agrees well with natural topological intuitions. Among other things, this has been used to great effect in describing quantum protocols by Abramsky and Coecke \cite{AC}, and this work in part motivates that presented here.

\begin{defn}[Monoidal category]
A \emph{monoidal category} $(\c, \ot, I,\a,\rho,\lambda)$ consists of a category $\c$, together with a functor $\ot: \c \times \c \to \c$, a distinguished object $I \in \c$, for all objects $A,B,C \in \c$ isomorphisms $\a_{A,B,C}: (A \ot B) \ot C \to A \ot (B \ot C)$ in $\c$ natural in $A,B,C$, and for all objects $A \in \c$ isomorphisms $\rho_A: A \ot I  \to A$ and $\lambda_A: I \ot A \to A$ in $\c$ natural in $A$. To form a monoidal category, this data is subject to two equations: the \emph{pentagon equation}
\[
\xymatrix{
&\big((A \ot B) \ot C\big) \ot D \ar[dl]_{\a_{A,B,C}\ot\id_D} \ar[ddr]^{\a_{(A\ot B),C,D}} \\
\big(A \ot (B\ot C)\big) \ot D \ar[dd]_{\a_{A,(B\ot C),D}}\\
&&(A \ot B) \ot (C \ot D) \ar[ddl]^{\a_{A,B,(C\ot D)}}\\
A\ot\big((B \ot C)\ot D\big)\ar[dr]_{\id_A \ot \a_{B,C,D}} \\
&A \ot \big(B \ot (C \ot D)\big)
}
\]
and the \emph{triangle equation}
\[
\xymatrix{
(A\ot I)\ot B  \ar[rr]^{\a_{A,I,B}} \ar[ddr]_{\rho_{A}\ot \id_B} && A \ot (I \ot B) \ar[ddl]^{\id_A\ot \lambda_B}\\
\\
& A \ot B \\
}
\]
 
We call $\ot$ the \emph{monoidal product}, $I$ the \emph{monoidal unit}, the isomorphisms $\a, \a\i$ \emph{associators}, and the isomorphisms $\rho,\rho\i$ and $\lambda, \lambda\i$ \emph{right-} and \emph{left-unitors} respectively. Collectively, we call the associators and unitors the \emph{structure maps} of our monoidal category. We will often just write $\c$ for a monoidal category $(\c,\ot,I,\a,\rho,\lambda)$, leaving the remaining data implicit.
\end{defn}

The associators express the fact that the product objects $(A \ot B) \ot C$ and $A \ot (B \ot C)$ are in some sense the same---they are isomorphic via some canonical isomorphism---, while the unitors express the fact that $A$, $A \ot I$, and $I \ot A$ are the same. If these objects are in fact equal, and the structure maps are simply identity maps, then we say that our monoidal category is a \emph{strict monoidal category}. In this case then any two objects that can be related by structure maps are equal, and so we may write objects without parentheses and units without ambiguity. Although, importantly, this is not true in all cases, it is essentially true: loosely speaking, the triangle and pentagon equations in fact imply that any diagram of their general kind, expressing composites of structure maps between different ways of forming the monoidal product of some objects, commutes. This is known as Mac Lane's coherence theorem for monoidal categorise; see Mac Lane \cite[Corollary of Theorem VII.2.1]{ML} for a precise statement and proof.

In a monoidal category, the objects $A \ot B$ and $B \ot A$ need not in general be related in any way. In the cases we will interest ourselves, however, we will not want the order in which we write the objects in a tensor product to matter---all products consisting of a given collection of objects should be isomorphic, and isomorphic in a way we need not worry about the isomorphism itself. This additional requirement turns a monoidal category into a symmetric monoidal category.

\begin{defn}[Symmetric monoidal category] \label{def:smc}
A \emph{symmetric monoidal category} $(\c, \ot, I,\a,\rho,\lambda, \s)$ consists of a monoidal category $(\c, \ot, I,\a,\rho,\lambda)$ together with a collection of isomorphisms $\s_{A,B}: A \ot B \to B \ot A$ natural in $A$ and $B$ such that $\s_{B,A} \circ \s_{A,B} = \id_{A\ot B}$ and such that for all objects $A,B,C$ the hexagon
\[
\xymatrix{
& (A \ot B) \ot C \ar[dl]_{\a_{A,B,C}} \ar[dr]^{\s_{A,B} \ot \id_C} \\
A \ot (B \ot C) \ar[dd]_{\s_{A,B \ot C}} && (B \ot A) \ot C \ar[dd]^{\a_{B,A,C}} \\
\\
(B \ot C) \ot A \ar[dr]_{\a_{B,C,A}} && B \ot (A \ot C) \ar[dl]^{\id_B \ot \s_{A,C}}\\
& B \ot (C \ot A)
}
\]
commutes. 

We call the isomorphisms $\s_{-,-}$ \emph{swaps}. 
\end{defn}

As for monoidal categories, we have a coherence theorem for symmetric monoidal categories, stating in essence that all diagrams composed of identities, associators, unitors, and swaps commute. Details can again be found in Mac Lane \cite[Theorem XI.1.1]{ML}.

\begin{exs}[$\mathbf{FVect}$, $\mathbf{Mat}$]
An historically important example of a symmetric monoidal category is that of $\mathbf{FVect}_\rr$, the category of finite vector spaces over $\rr$ with linear maps as morphisms, tensor product as monoidal product. Here $\rr$ is the monoidal unit, and the structure maps are the obvious isomorphisms of tensor products of vector spaces. Note that this is not a strict symmetric monoidal category: it is not true for real vector spaces $U,V,W$ that we consider $(U \ot V) \ot W$ and $U \ot (V \ot W)$ as equal, but we do always have a canonical isomorphism between the two.

A related strict symmetric monoidal category is $\mathbf{Mat(\rr)}$, the category with objects natural numbers, morphisms from $m \in \mathbb{N}$ to $n \in \mathbb{N}$ given by $n \times m$ matrices over $\rr$, composition given by composition of matrices, monoidal product given by multiplication on objects and Kronecker product of matrices on morphisms.
\end{exs}

\begin{ex}[$\set$]
The category $\mathbf{Set}$ of sets and functions forms a symmetric monoidal category with the cartesian product $\times$. In this category any singleton set $\{\ast\}$ may be taken as the monoidal unit. Indeed, any category with finite products can be viewed as a symmetric monoidal category by taking the binary categorical product as the monoidal product, and the terminal object as the monoidal unit. The associators, unitors, and swaps are then specified by the unique isomorphisms given by the universal property of the product.
\end{ex}

\begin{ex}[$\rel$]
The category $\rel$ of sets and relations forms a symmetric monoidal category with cartesian product $\times$ and unit $\{\ast\}$. Here the monoidal product $r \ot s: X \times Z \to Y \times W$ of relations $r X \to Y$ and $s: Z \to W$ is the relation such that $(x,z) \in X \times Z$ is related to $(y,w) \in Y \times W$ if and only if $x$ is related to $y$ by $r$ and $z$ is related to $w$ by $s$.

Intuitively, the standard embedding of $\set$ into $\rel$, given by viewing functions as relations, is an embedding that respects the monoidal structure. To make this precise we need to talk about monoidal functors.
\end{ex}

When working with monoidal categories, it is often desirable to have functors between these categories preserve the monoidal structure, and to have natural transformations between these functors preserve the monoidal structure too. The same is true in the case of functors between symmetric monoidal categories. We thus introduce the notions of monoidal functors, symmetric monoidal functors, and monoidal natural transformations.

\begin{defn}[Monoidal functor]
Let $\c, \c'$ be monoidal categories. A \emph{monoidal functor} $(F, F_\ot, F_\ast): \c \to \c'$ from $\c$ to $\c'$ consists of a functor $F: \c \to \c'$, for all objects $A,B \in \c$ morphisms
\[
F_{\ot,A,B}: F(A) \ot F(B) \to F(A \ot B)
\]
in $\c'$ which are natural in $A$ and $B$, and for the units $I$ of $\c$ and $I'$ of $\c'$ a morphism $F_\ast: I' \to F(I)$ in $\c'$, such that for all $A,B,C \in \c$ the hexagon
\[
\xymatrix{
& (FA \ot FB) \ot FC \ar[dl]_{\a_{FA,FB,FC}} \ar[dr]^{F_{\ot,A,B} \ot \id_{FC}} \\
FA \ot (FB \ot FC) \ar[dd]_{\id_{FA} \ot F_{\ot,B,C}} && F(A \ot B) \ot FC \ar[dd]^{F_{\ot,A\ot B,C}} \\
\\
FA \ot F(B \ot C) \ar[dr]_{F_{\ot,A,B\ot C}} && F((A \ot B) \ot C) \ar[dl]^{F\a_{A,B,C}}\\
&F(A \ot (B \ot C))
}
\]
and the two squares
\[
\xymatrixcolsep{3pc}
\xymatrixrowsep{3pc}
\xymatrix{
F(A) \ot I' \ar[d]_{\id \ot F_\ast} \ar[r]^{\rho} & F(A) \\
F(A) \ot F(I) \ar[r]_{F_{\ot,A,I}} & F(A \ot I) \ar[u]_{F\rho} 
}
\qquad
\xymatrix{
I' \ot F(A) \ar[d]_{F_\ast \ot \id} \ar[r]^{\lambda} & F(A) \\
F(I) \ot F(A) \ar[r]_{F_{\ot,I,A}} & F(I \ot A) \ar[u]_{F\lambda} 
}
\]
commute.

We further say a monoidal functor is a \emph{strong monoidal functor} if the morphisms $F_{\ot,A,B}$ and $F_\ast$ are isomorphisms for all $A, B \in \c$. 
\end{defn}

\begin{defn}[Symmetric monoidal functor]
A symmetric monoidal functor \newline \mbox{$(F, F_\ot, F_\ast): \c \to \c'$} between symmetric monoidal categories $\c$ and $\c'$ is a monoidal functor $(F, F_\ot, F_\ast)$ such that
\[
\xymatrixcolsep{3pc}
\xymatrixrowsep{3pc}
\xymatrix{
FA \ot FB \ar[r]^{F_{\ot,A,B}} \ar[d]_{\s'_{FA,FB}} & F(A \ot B)\ar[d]^{F\s_{A,B}}\\
FB \ot FA \ar[r]_{F_{\ot,B,A}} & F(B \ot A)
}
\]
commutes for all $A,B \in \c$.
\end{defn}

\begin{defn}[Monoidal natural transformation]
A \emph{monoidal natural transformation} $\theta: F \Rightarrow G$ between two monoidal functors $F$ and $G$ is a natural transformation $\theta: F \Rightarrow G$ such that the triangle
\[
\begin{aligned}
\xymatrix{
FI \ar[rr]^{\theta_I}&& GI \\
& I' \ar[ul]^{F_\ast} \ar[ur]_{G_\ast}
} 
\end{aligned} \label{mnt:tri} \tag{MNT1}
\]
and square
\[
\begin{aligned}
\xymatrixcolsep{3pc}
\xymatrixrowsep{3pc}
\xymatrix{
FA \ot FB \ar[r]^{\theta_A \ot \theta_B} \ar[d]_{F_{\ot,A,B}} & GA \ot GB \ar[d]^{G_{\ot,A,B}}\\
F(A \ot B) \ar[r]_{\theta_{A\ot B}} & G(A \ot B)
}
\end{aligned} 
\label{mnt:sq} \tag{MNT2}
\]
commute for all objects $A,B$.
\end{defn}

\begin{ex}
$\mathbf{FVect_\rr}$ and $\mathbf{Mat(\rr)}$ are equivalent via strong monoidal functors. It is a corollary of the Mac Lane Coherence theorems that any monoidal category can be `strictified'---that is, for any monoidal category there exists a strict monoidal category equivalent to it via strong monoidal functors---and that a symmetric monoidal category can be strictified into a strict symmetric monoidal category. For more details see \cite[Theorem XI.3.1]{ML}.
\end{ex}

\section{Graphical calculi}

One of the draws of expressing concepts as symmetric monoidal categories is that the structure of these categories naturally lends itself to being expressed pictorially. These pictures, known as \emph{string diagrams}, represent the morphisms of a monoidal category, and have the benefit of hiding certain structural equalities and making use of our topological intuitions to suggest other important equalities. The aim of this section is merely to give the reader a basic working understanding of how to read and draw these diagrams; we leave the precise definition of a string diagram and proofs of their expressiveness to the survey \cite{Se} of Selinger.

String diagrams are drawn in two dimensions with, roughly speaking, one dimension representing the categorical composition and the other representing monoidal composition. We shall take the convention, common but far from universal, that we read composition up the page, leaving horizontal juxtaposition to represent the monoidal product of maps. Under this convention then, a string diagram consists of a graph with edges labelled by objects and vertices labelled by morphisms, which as a whole represents a morphism with domain the monoidal product of the edges at the lower end of the diagram, and codomain the monoidal product of the edges at the top.

The simplest example, consisting of just a single edge, represents the identity map:
\[
\id_X =
\begin{aligned}
\begin{tikzpicture}
	\begin{pgfonlayer}{nodelayer}
		\node [style=none] (0) at (0, 1) {};
		\node [style=none] (1) at (0, -1) {};
		\node [style=none] (2) at (0, 1.25) {$X$};
		\node [style=none] (3) at (0, -1.25) {$X$};
	\end{pgfonlayer}
	\begin{pgfonlayer}{edgelayer}
		\draw (0.center) to (1.center);
	\end{pgfonlayer}
\end{tikzpicture}
\end{aligned}
\]
More generally, we represent a morphism $f: X \to Y$ by drawing in sequence up the page an edge labelled by $X$, ending at a vertex labelled by $f$, which then gives rise to an edge labelled by $Y$:
\[
f =
\begin{aligned}
\begin{tikzpicture}
	\begin{pgfonlayer}{nodelayer}
		\node [style=none] (0) at (0, 1) {};
		\node [style=sq] (1) at (0, 0) {$f$};
		\node [style=none] (2) at (0, -1) {};
		\node [style=none] (3) at (0, 1.25) {$Y$};
		\node [style=none] (4) at (0, -1.25) {$X$};
	\end{pgfonlayer}
	\begin{pgfonlayer}{edgelayer}
		\draw (0.center) to (1.center);
		\draw (1.center) to (2.center);
	\end{pgfonlayer}
\end{tikzpicture}
\end{aligned}
\]
If $X = A \ot B \ot C$, and $Y = D \ot E$, we could also represent $f$ as:
\[
f =
\begin{aligned}
\begin{tikzpicture}
	\begin{pgfonlayer}{nodelayer}
		\node [style=none] (0) at (0, -0.25) {};
		\node [style=none] (1) at (0.5, -0.25) {};
		\node [style=none] (2) at (-0.5, -0.25) {};
		\node [style=none] (3) at (-.75, -0.25) {};
		\node [style=none] (4) at (-.75, 0.25) {};
		\node [style=none] (5) at (-0.3, 0.25) {};
		\node [style=none] (6) at (0.3, 0.25) {};
		\node [style=none] (7) at (.75, 0.25) {};
		\node [style=none] (8) at (.75, -0.25) {};
		\node [style=none] (9) at (0.5, -1) {};
		\node [style=none] (10) at (0, -1) {};
		\node [style=none] (11) at (-0.5, -1) {};
		\node [style=none] (12) at (-0.3, 1) {};
		\node [style=none] (13) at (0.3, 1) {};
		\node [style=none] (14) at (-0.5, -1.25) {$A$};
		\node [style=none] (15) at (0, -1.25) {$B$};
		\node [style=none] (16) at (0.5, -1.25) {$C$};
		\node [style=none] (17) at (-0.3, 1.25) {$D$};
		\node [style=none] (18) at (0.3, 1.25) {$E$};
		\node [style=none] (19) at (0, 0) {$f$};
	\end{pgfonlayer}
	\begin{pgfonlayer}{edgelayer}
		\draw (8.center) to (3.center);
		\draw (3.center) to (4.center);
		\draw (4.center) to (7.center);
		\draw (7.center) to (8.center);
		\draw (11.center) to (2.center);
		\draw (10.center) to (0.center);
		\draw (9.center) to (1.center);
		\draw (5.center) to (12.center);
		\draw (6.center) to (13.center);
	\end{pgfonlayer}
\end{tikzpicture}
\end{aligned}
\]

Given maps $f:X \to Y$ and $g: Y \to Z$, we represent their composite $g \circ f$ by placing a vertex $g$ on the $Y$-edge leaving the vertex $f$:
\[
g \circ f =
\begin{aligned}
\begin{tikzpicture}
	\begin{pgfonlayer}{nodelayer}
		\node [style=none] (0) at (0, 1) {};
		\node [style=sq] (1) at (0, -0.4) {$f$};
		\node [style=none] (2) at (0, -1) {};
		\node [style=none] (3) at (0, 1.25) {$Z$};
		\node [style=none] (4) at (0, -1.25) {$X$};
		\node [style=sq] (5) at (0, 0.4) {$g$};
	\end{pgfonlayer}
	\begin{pgfonlayer}{edgelayer}
		\draw (1) to (2.center);
		\draw (5) to (1);
		\draw (5) to (0.center);
	\end{pgfonlayer}
\end{tikzpicture}
\end{aligned}
\]
If the types of the maps are known, we lose no information if we omit the labels of edges that are connected to the vertices of the maps, as we have done for edge representing $Y$ in the above diagram. For the sake of cleanness and readability, we shall most often just label the `input' and `output' edges at the top and bottom of the diagram.

The monoidal product of two maps is given by their horizontal juxtaposition, with juxtaposition on the right representing monoidal product on the right, and on the left representing left monoidal product. As an example, given morphisms $f:X \to Y$ and $g: Z \to W$, we write their product $f \ot g: X \ot Z \to Y \ot W$ as:
\[
f \ot g =
\begin{aligned}
\begin{tikzpicture}
	\begin{pgfonlayer}{nodelayer}
		\node [style=none] (0) at (-0.5, 1) {};
		\node [style=none] (1) at (-0.5, -1) {};
		\node [style=none] (2) at (-0.5, 1.25) {$Y$};
		\node [style=none] (3) at (-0.5, -1.25) {$X$};
		\node [style=sq] (4) at (-0.5, 0) {$f$};
		\node [style=none] (5) at (0.25, -1.25) {$Z$};
		\node [style=none] (6) at (0.25, 1.25) {$W$};
		\node [style=none] (7) at (0.25, -1) {};
		\node [style=none] (8) at (0.25, 1) {};
		\node [style=sq] (9) at (0.25, 0) {$g$};
	\end{pgfonlayer}
	\begin{pgfonlayer}{edgelayer}
		\draw (0.center) to (4);
		\draw (4) to (1.center);
		\draw (8.center) to (9);
		\draw (9) to (7.center);
	\end{pgfonlayer}
\end{tikzpicture}
\end{aligned}
\]
The monoidal unit is an object with special properties in the category, and as a result the conventions for representing the unit diagrammatically are a little different: we don't draw it or its identity map
\[
\id_I = \begin{aligned}
\begin{tikzpicture}
	\begin{pgfonlayer}{nodelayer}
		\node [style=none] (0) at (0, 1) {};
		\node [style=none] (1) at (0, -1) {};
	\end{pgfonlayer}
\end{tikzpicture}
\end{aligned}
\]
This has the advantage of any diagram representing a morphism $f: A \to B$ also representing the `equivalent' morphism $f \ot \id_I: A \ot I \to B \ot I$, among other equivalent morphisms.

To read an arbitrary string diagram, it is often easiest to start at the lower edge and move up the diagram, reading off a morphism for every horizontal cross-section intersecting a vertex. The string diagram then represents the composite of these morphisms in the order that the morphisms were read, applying associators and unitors as needed for the map to be well-defined. For example, reading in this way the diagram
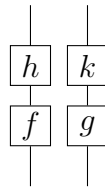
\[
\begin{aligned}
\begin{tikzpicture}
	\begin{pgfonlayer}{nodelayer}
		\node [style=none] (0) at (-0.5, 1.2) {};
		\node [style=none] (1) at (-0.5, -1.2) {};
		\node [style=sq] (4) at (-0.5, -.4) {$f$};
		\node [style=sq] (5) at (-0.5, .4) {$h$};
		\node [style=none] (7) at (0.25, -1.2) {};
		\node [style=none] (8) at (0.25, 1.2) {};
		\node [style=sq] (9) at (0.25, -.4) {$g$};
		\node [style=sq] (10) at (0.25, .4) {$k$};
	\end{pgfonlayer}
	\begin{pgfonlayer}{edgelayer}
		\draw (0.center) to (4);
		\draw (4) to (1.center);
		\draw (8.center) to (9);
		\draw (9) to (7.center);
	\end{pgfonlayer}
\end{tikzpicture}
\end{aligned}
\]
represents the map $(f \ot g) \circ (h \ot k)$. Note that it may also be read $(f \circ h) \ot (g \circ k)$, or even $(f \circ \id_X \circ h) \ot (g \circ k) \ot \id_I$ where $X$ is the codomain of $f$, but in any case all these different algebraic descriptions of the picture represent the same morphism. This is a key feature of string diagrams: many equalities of algebraic representations of morphisms become just the identity of diagrams. Furthermore, we need not be too careful about the precise geometry of the diagrams; the following topologically equivalent diagrams in fact also express equal morphisms:
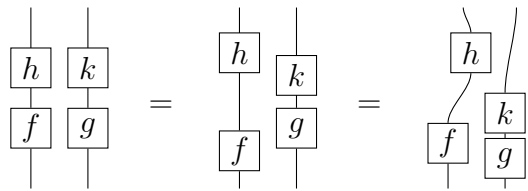
\[
\begin{aligned}
\begin{tikzpicture}
	\begin{pgfonlayer}{nodelayer}
		\node [style=none] (0) at (-0.5, 1.2) {};
		\node [style=none] (1) at (-0.5, -1.2) {};
		\node [style=sq] (4) at (-0.5, -.4) {$f$};
		\node [style=sq] (5) at (-0.5, .4) {$h$};
		\node [style=none] (7) at (0.25, -1.2) {};
		\node [style=none] (8) at (0.25, 1.2) {};
		\node [style=sq] (9) at (0.25, -.4) {$g$};
		\node [style=sq] (10) at (0.25, .4) {$k$};
	\end{pgfonlayer}
	\begin{pgfonlayer}{edgelayer}
		\draw (0.center) to (4);
		\draw (4) to (1.center);
		\draw (8.center) to (9);
		\draw (9) to (7.center);
	\end{pgfonlayer}
\end{tikzpicture}
\end{aligned}
\quad
=
\quad
\begin{aligned}
\begin{tikzpicture}
	\begin{pgfonlayer}{nodelayer}
		\node [style=none] (0) at (-0.5, 1.2) {};
		\node [style=none] (1) at (-0.5, -1.2) {};
		\node [style=sq] (4) at (-0.5, -.7) {$f$};
		\node [style=sq] (5) at (-0.5, .6) {$h$};
		\node [style=none] (7) at (0.25, -1.2) {};
		\node [style=none] (8) at (0.25, 1.2) {};
		\node [style=sq] (9) at (0.25, -.4) {$g$};
		\node [style=sq] (10) at (0.25, .3) {$k$};
	\end{pgfonlayer}
	\begin{pgfonlayer}{edgelayer}
		\draw (0.center) to (4);
		\draw (4) to (1.center);
		\draw (8.center) to (9);
		\draw (9) to (7.center);
	\end{pgfonlayer}
\end{tikzpicture}
\end{aligned}
\quad
=
\quad
\begin{aligned}
\begin{tikzpicture}
	\begin{pgfonlayer}{nodelayer}
		\node [style=none] (0) at (-0.3, 1.2) {};
		\node [style=none] (1) at (-0.5, -1.2) {};
		\node [style=sq] (4) at (-0.5, -.6) {$f$};
		\node [style=sq] (5) at (-0.2, .6) {$h$};
		\node [style=none] (7) at (0.25, -1.2) {};
		\node [style=none] (8) at (0.4, 1.2) {};
		\node [style=sq] (9) at (0.25, -.8) {$g$};
		\node [style=sq] (10) at (0.25, -.2) {$k$};
	\end{pgfonlayer}
	\begin{pgfonlayer}{edgelayer}
		\draw [in=90, out=-90, looseness=0.75]  (0.center) to (5);
		\draw [in=90, out=-90, looseness=0.75]  (5) to (4);
		\draw (4) to (1.center);
		\draw (7.center) to (9);
		\draw (9) to (10);		
		\draw  [in=90, out=-90, looseness=0.75] (8.center) to (10);
	\end{pgfonlayer}
\end{tikzpicture}
\end{aligned}
\]
This holds true in general.

\begin{thm}[Coherence of the graphical calculus for monoidal categories]
Two morphisms in a monoidal category are equal with their equality following from the axioms of monoidal categories if and only if their diagrams are equal up to planar deformation.
\end{thm}
\begin{proof}
Joyal-Street \cite[Theorem 1.2]{JS}.
\end{proof}

In a symmetric monoidal category, we usually omit the label for the swap, denoting it instead just by the intersection of two strings:
\[
\s_{A,B} =
\begin{aligned}
\begin{tikzpicture}
	\begin{pgfonlayer}{nodelayer}
		\node [style=none] (0) at (-0.5, 1.25) {};
		\node [style=none] (1) at (0.5, 1.25) {};
		\node [style=none] (2) at (-0.5, -0.25) {};
		\node [style=none] (3) at (0.5, -0.25) {};
		\node [style=none] (4) at (-0.5, 0.75) {};
		\node [style=none] (5) at (0.5, 0.75) {};
		\node [style=none] (6) at (0.5, 0.25) {};
		\node [style=none] (7) at (-0.5, 0.25) {};
		\node [style=none] (8) at (-0.5, -0.5) {$A$};
		\node [style=none] (9) at (0.5, -0.5) {$B$};
		\node [style=none] (10) at (-0.5, 1.5) {$B$};
		\node [style=none] (11) at (0.5, 1.5) {$A$};
		\node [style=none] (12) at (-0.75, 0.75) {};
		\node [style=none] (13) at (0.75, 0.75) {};
		\node [style=none] (14) at (0.75, 0.25) {};
		\node [style=none] (15) at (-0.75, 0.25) {};
		\node [style=none] (16) at (0, 0.5) {$\s_{A,B}$};
	\end{pgfonlayer}
	\begin{pgfonlayer}{edgelayer}
		\draw (0.center) to (4.center);
		\draw (6.center) to (3.center);
		\draw (1.center) to (5.center);
		\draw (7.center) to (2.center);
		\draw (14.center) to (15.center);
		\draw (15.center) to (12.center);
		\draw (12.center) to (13.center);
		\draw (13.center) to (14.center);
	\end{pgfonlayer}
\end{tikzpicture}
\end{aligned}
=
\begin{aligned}
\begin{tikzpicture}
	\begin{pgfonlayer}{nodelayer}
		\node [style=none] (0) at (-0.5, 1.25) {};
		\node [style=none] (1) at (0.5, 1.25) {};
		\node [style=none] (2) at (-0.5, -0.25) {};
		\node [style=none] (3) at (0.5, -0.25) {};
		\node [style=none] (4) at (-0.5, 1) {};
		\node [style=none] (5) at (0.5, 1) {};
		\node [style=none] (6) at (0.5, 0) {};
		\node [style=none] (7) at (-0.5, 0) {};
		\node [style=none] (8) at (-0.5, -0.5) {$A$};
		\node [style=none] (9) at (0.5, -0.5) {$B$};
		\node [style=none] (10) at (-0.5, 1.5) {$B$};
		\node [style=none] (11) at (0.5, 1.5) {$A$};
	\end{pgfonlayer}
	\begin{pgfonlayer}{edgelayer}
		\draw (0.center) to (4.center);
		\draw [in=90, out=-90, looseness=0.75] (4.center) to (6.center);
		\draw (6.center) to (3.center);
		\draw (1.center) to (5.center);
		\draw [in=90, out=-90, looseness=0.75] (5.center) to (7.center);
		\draw (7.center) to (2.center);
	\end{pgfonlayer}
\end{tikzpicture}
\end{aligned}
\]
We will also later take such an approach for other chosen maps, such as the multiplication and unit of a monoid.

The defining identities of the swap may then be written graphically as
\[
\begin{aligned}
\begin{tikzpicture}
	\begin{pgfonlayer}{nodelayer}
		\node [style=none] (0) at (-0.5, -1.25) {};
		\node [style=none] (1) at (0.5, -1.25) {};
		\node [style=none] (2) at (-0.5, 0) {};
		\node [style=none] (3) at (0.5, 0) {};
		\node [style=none] (4) at (0.5, -1) {};
		\node [style=none] (5) at (-0.5, -1) {};
		\node [style=none] (6) at (-0.5, -1.5) {$A$};
		\node [style=none] (7) at (0.5, -1.5) {$B$};
		\node [style=none] (8) at (0.5, 0) {};
		\node [style=none] (9) at (0.5, 1) {};
		\node [style=none] (10) at (-0.5, 1.5) {$A$};
		\node [style=none] (11) at (0.5, 1.5) {$B$};
		\node [style=none] (12) at (-0.5, 1.25) {};
		\node [style=none] (13) at (-0.5, 0) {};
		\node [style=none] (14) at (-0.5, 1) {};
		\node [style=none] (15) at (0.5, 1.25) {};
	\end{pgfonlayer}
	\begin{pgfonlayer}{edgelayer}
		\draw [in=90, out=-90, looseness=0.75] (2.center) to (4.center);
		\draw (4.center) to (1.center);
		\draw [in=90, out=-90, looseness=0.75] (3.center) to (5.center);
		\draw (5.center) to (0.center);
		\draw (12.center) to (14.center);
		\draw [in=90, out=-90, looseness=0.75] (14.center) to (8.center);
		\draw (15.center) to (9.center);
		\draw [in=90, out=-90, looseness=0.75] (9.center) to (13.center);
	\end{pgfonlayer}
\end{tikzpicture}
\end{aligned}
\quad
=
\quad
\begin{aligned}
\begin{tikzpicture}
	\begin{pgfonlayer}{nodelayer}
		\node [style=none] (0) at (-0.5, -1.25) {};
		\node [style=none] (1) at (0.5, -1.25) {};
		\node [style=none] (2) at (-0.5, -1.5) {$A$};
		\node [style=none] (3) at (0.5, -1.5) {$B$};
		\node [style=none] (4) at (-0.5, 1.5) {$A$};
		\node [style=none] (5) at (0.5, 1.5) {$B$};
		\node [style=none] (6) at (-0.5, 1.25) {};
		\node [style=none] (7) at (0.5, 1.25) {};
	\end{pgfonlayer}
	\begin{pgfonlayer}{edgelayer}
		\draw (6.center) to (0.center);
		\draw (7.center) to (1.center);
	\end{pgfonlayer}
\end{tikzpicture}
\end{aligned}\label{swapdiag1}\tag{Sym1}
\]
and
\[\label{swapdiag2}\tag{Sym2}
\begin{aligned}
\begin{tikzpicture}
	\begin{pgfonlayer}{nodelayer}
		\node [style=none] (0) at (-0.5, -0.25) {};
		\node [style=none] (1) at (0.5, -0.25) {};
		\node [style=none] (2) at (-0.5, 1.5) {};
		\node [style=none] (3) at (0.5, 0.5) {};
		\node [style=none] (4) at (-0.5, 0.5) {};
		\node [style=none] (5) at (-0.5, -0.5) {$A$};
		\node [style=none] (6) at (0.5, -0.5) {$B$};
		\node [style=none] (7) at (1.5, 0.5) {};
		\node [style=none] (8) at (0.5, 2.5) {$C$};
		\node [style=none] (9) at (1.5, 2.5) {$A$};
		\node [style=none] (10) at (0.5, 2.25) {};
		\node [style=none] (11) at (0.5, 1.5) {};
		\node [style=none] (12) at (1.5, 2.25) {};
		\node [style=none] (13) at (1.5, -0.25) {};
		\node [style=none] (14) at (-0.5, 2.25) {};
		\node [style=none] (15) at (-0.5, 2.5) {$B$};
		\node [style=none] (16) at (1.5, -0.5) {$C$};
		\node [style=none] (17) at (0, 1) {};
		\node [style=none] (18) at (1, 1) {};
		\node [style=none] (19) at (1.5, 1.5) {};
	\end{pgfonlayer}
	\begin{pgfonlayer}{edgelayer}
		\draw [in=90, out=-90, looseness=0.75] (2.center) to (3.center);
		\draw (3.center) to (1.center);
		\draw (4.center) to (0.center);
		\draw (10.center) to (11.center);
		\draw [in=90, out=-90, looseness=0.75] (11.center) to (7.center);
		\draw (14.center) to (2.center);
		\draw (7.center) to (13.center);
		\draw [in=180, out=90] (4.center) to (17.center);
		\draw (17.center) to (18.center);
		\draw [in=-90, out=0] (18.center) to (19.center);
		\draw (19.center) to (12.center);
	\end{pgfonlayer}
\end{tikzpicture}
\end{aligned}
\quad
=
\quad
\begin{aligned}
\begin{tikzpicture}
	\begin{pgfonlayer}{nodelayer}
		\node [style=none] (0) at (-0.5, -0.25) {};
		\node [style=none] (1) at (0.5, -0.25) {};
		\node [style=none] (2) at (-0.5, 1) {};
		\node [style=none] (3) at (0.5, 1) {};
		\node [style=none] (4) at (0.5, 0) {};
		\node [style=none] (5) at (-0.5, 0) {};
		\node [style=none] (6) at (-0.5, -0.5) {$A$};
		\node [style=none] (7) at (0.5, -0.5) {$B$};
		\node [style=none] (8) at (1.5, 1) {};
		\node [style=none] (9) at (1.5, 2) {};
		\node [style=none] (10) at (0.5, 2.5) {$C$};
		\node [style=none] (11) at (1.5, 2.5) {$A$};
		\node [style=none] (12) at (0.5, 2.25) {};
		\node [style=none] (13) at (0.5, 1) {};
		\node [style=none] (14) at (0.5, 2) {};
		\node [style=none] (15) at (1.5, 2.25) {};
		\node [style=none] (16) at (1.5, -0.25) {};
		\node [style=none] (17) at (-0.5, 2.25) {};
		\node [style=none] (18) at (-0.5, 2.5) {$B$};
		\node [style=none] (19) at (1.5, -0.5) {$C$};
	\end{pgfonlayer}
	\begin{pgfonlayer}{edgelayer}
		\draw [in=90, out=-90, looseness=0.75] (2.center) to (4.center);
		\draw (4.center) to (1.center);
		\draw [in=90, out=-90, looseness=0.75] (3.center) to (5.center);
		\draw (5.center) to (0.center);
		\draw (12.center) to (14.center);
		\draw [in=90, out=-90, looseness=0.75] (14.center) to (8.center);
		\draw (15.center) to (9.center);
		\draw [in=90, out=-90, looseness=0.75] (9.center) to (13.center);
		\draw (17.center) to (2.center);
		\draw (8.center) to (16.center);
	\end{pgfonlayer}
\end{tikzpicture}
\end{aligned}
\]
Including these identity into our collection of allowable transformations of diagrams gives coherence theorem for symmetric monoidal categories.

\begin{thm}[Coherence of the graphical calculus for symmetric monoidal categories]
Two morphisms in a symmetric monoidal category are equal with their equality following from the axioms of symmetric monoidal categories if and only if their diagrams are equal up to planar deformation and local applications of the identities \ref{swapdiag1} and \ref{swapdiag2}.
\end{thm}
\begin{proof}
Joyal-Street \cite[Theorem 2.3]{JS}.
\end{proof}
Just as two diagrams represent the same morphism in a monoidal category if they agree up to planar isotopy, this theorem may be regarded geometrically as stating that two diagrams represent the same morphism in a monoidal category if they agree up to isotopy in four dimensions.

These two theorems show that the graphical calculi go beyond visualisations of the morphisms, having the ability to provide bona-fide proofs of equalities of morphisms. As a general principle, one which we shall demonstrate in this dissertation, this fact combined the intuitiveness of manipulations and the encoding of certain equalities and structural isomorphisms make the string diagrams better than the conventional algebraic language for understanding monoidal categories.

\section{Example: the theory of monoids}
This section serves to both give examples of the constructions defined in this chapter and, more importantly, give a flavour of the spirit in which we will aim to use monoidal categories to discuss causality. 

Recall that a monoid is a set with an associative, unital binary operation. We shall classify these as strong monoidal functors from a category $\mathbf{Th(Mon)}$ into $\set$, and hence say that this category $\mathbf{Th(Mon)}$ describes the theory of monoids. The study of this category and its functorial images then gives new and interesting perspectives of the concept of a monoid, its generalisations, and relationships to other mathematical structures. In analogy to this, we will later define causal theories as monoidal categories that can be modelled within other categories through monoidal functors.

Define the category $\mathbf{Th(Mon)}$ as follows: fix some symbol $M$, and let the objects of $\mathbf{Th(Mon)}$ be any natural number of copies of this symbol. We shall write the objects $M^{\ot n}$, where $n \in \nn$ is the number of copies of $M$. Then the monoidal product on the objects of $\mathbf{Th(Mon)}$ is just addition of number of copies of $M$, with $I = M^{\ot 0}$ the monoidal unit. By definition this is a strict monoidal category, so we need not worry about the structure maps.

In addition to the identity morphism on each object, we also include morphisms $\mult: M \ot M \to M$ and $\unit: I \to M$ and all their composites and products, subject to the relations 
\[
\begin{aligned}
\begin{tikzpicture}
	\begin{pgfonlayer}{nodelayer}
		\node [style=none] (0) at (0.5, -0.5) {};
		\node [style=none] (1) at (1, 0.5) {};
		\node [style=none] (2) at (1.5, -0.25) {};
		\node [style=none] (3) at (1, 0.25) {};
		\node [style=none] (4) at (0.5, -0.25) {};
		\node [style=none] (5) at (1.5, -0.5) {};
		\node [style=none] (6) at (1.5, 1.5) {};
		\node [style=none] (7) at (2, 0.5) {};
		\node [style=none] (8) at (1.5, 1) {};
		\node [style=none] (9) at (1, 0.5) {};
		\node [style=none] (10) at (2, -0.5) {};
	\end{pgfonlayer}
	\begin{pgfonlayer}{edgelayer}
		\draw (0.center) to (4.center);
		\draw (5.center) to (2.center);
		\draw [in=0, out=90, looseness=0.75] (2.center) to (3.center);
		\draw [in=180, out=90, looseness=0.75] (4.center) to (3.center);
		\draw (3.center) to (1.center);
		\draw [in=0, out=90, looseness=0.75] (7.center) to (8.center);
		\draw [in=180, out=90, looseness=0.75] (9.center) to (8.center);
		\draw (8.center) to (6.center);
		\draw (7.center) to (10.center);
	\end{pgfonlayer}
\end{tikzpicture}
\end{aligned}
\quad
=
\quad
\begin{aligned}
\begin{tikzpicture}
	\begin{pgfonlayer}{nodelayer}
		\node [style=none] (0) at (0.5, -0.5) {};
		\node [style=none] (1) at (1, 0.5) {};
		\node [style=none] (2) at (1.5, -0.25) {};
		\node [style=none] (3) at (1, 0.25) {};
		\node [style=none] (4) at (0.5, -0.25) {};
		\node [style=none] (5) at (1.5, -0.5) {};
		\node [style=none] (6) at (0, -0.5) {};
		\node [style=none] (7) at (0.5, 1.5) {};
		\node [style=none] (8) at (1, 0.5) {};
		\node [style=none] (9) at (0.5, 1) {};
		\node [style=none] (10) at (0, 0.5) {};
	\end{pgfonlayer}
	\begin{pgfonlayer}{edgelayer}
		\draw (0.center) to (4.center);
		\draw (5.center) to (2.center);
		\draw [in=0, out=90, looseness=0.75] (2.center) to (3.center);
		\draw [in=180, out=90, looseness=0.75] (4.center) to (3.center);
		\draw (3.center) to (1.center);
		\draw (6.center) to (10.center);
		\draw [in=0, out=90, looseness=0.75] (8.center) to (9.center);
		\draw [in=180, out=90, looseness=0.75] (10.center) to (9.center);
		\draw (9.center) to (7.center);
	\end{pgfonlayer}
\end{tikzpicture}
\end{aligned} \tag{associativity}
\]
and
\[
\begin{aligned}
\begin{tikzpicture}
	\begin{pgfonlayer}{nodelayer}
		\node [style=none] (0) at (-0.5, -1.25) {};
		\node [style=none] (1) at (0, 0.75) {};
		\node [style=none] (2) at (0.5, -0.5) {};
		\node [style=none] (3) at (0, 0) {};
		\node [style=none] (4) at (-0.5, -0.5) {};
		\node [style=dot] (5) at (0.5, -0.75) {};
	\end{pgfonlayer}
	\begin{pgfonlayer}{edgelayer}
		\draw (0.center) to (4.center);
		\draw [in=-90, out=90] (5) to (2.center);
		\draw [in=0, out=90, looseness=0.75] (2.center) to (3.center);
		\draw [in=180, out=90, looseness=0.75] (4.center) to (3.center);
		\draw (3.center) to (1.center);
	\end{pgfonlayer}
\end{tikzpicture}
\end{aligned}
\quad
=
\quad
\begin{aligned}
\begin{tikzpicture}
	\begin{pgfonlayer}{nodelayer}
		\node [style=none] (0) at (0, 1) {};
		\node [style=none] (1) at (0, -1) {};
	\end{pgfonlayer}
	\begin{pgfonlayer}{edgelayer}
		\draw (1.center) to (0.center);
	\end{pgfonlayer}
\end{tikzpicture}
\end{aligned}
\quad
=
\quad
\begin{aligned}
\begin{tikzpicture}
	\begin{pgfonlayer}{nodelayer}
		\node [style=dot] (0) at (-0.5, -0.75) {};
		\node [style=none] (1) at (0, 0.75) {};
		\node [style=none] (2) at (0.5, -0.5) {};
		\node [style=none] (3) at (0, 0) {};
		\node [style=none] (4) at (-0.5, -0.5) {};
		\node [style=none] (5) at (0.5, -1.25) {};
	\end{pgfonlayer}
	\begin{pgfonlayer}{edgelayer}
		\draw (0) to (4.center);
		\draw [in=-90, out=90] (5.center) to (2.center);
		\draw [in=0, out=90, looseness=0.75] (2.center) to (3.center);
		\draw [in=180, out=90, looseness=0.75] (4.center) to (3.center);
		\draw (3.center) to (1.center);
	\end{pgfonlayer}
\end{tikzpicture}
\end{aligned} \tag{unitality}
\]
These equations correspond respectively to the associativity and unitality laws for the monoid.

Now given any monoid $(X,\cdot,1)$, we can define the strong monoidal functor $F:\mathbf{Th(Mon)} \to \set$ mapping $M^{\ot n}$ to the $n$-fold cartesian product $X^n$ of the set $X$, $m$ to the monoid multiplication function $X \times X \to X; (x,y) \mapsto x\cdot y$, and $e$ to the function $\{\ast\} \to X; \ast \mapsto 1$ with image the monoid unit. This is well-defined as the relations obeyed by $m$ and $e$ are precisely those required to ensure the monoid operation $\cdot$ is associative and unital. Furthermore, taking the canonical isomorphisms $X^{m} \times X^{n} \stackrel\sim\longrightarrow X^{m+n}$ given by the universal property of products, we see that $F$ is a strong monoidal functor.

Conversely, given any strong monoidal functor $(F,F_\ot,F_\ast):\mathbf{Th(Mon)} \to \set$, it is straightforward to show, using the naturality of $F_\ot$ and the diagrams obeyed by the definition of a strong monoidal functor, that the triple $(FM, Fm \circ F_{\ot,M,M},Fe \circ F_\ast(\ast))$ is a well-defined monoid. From here it also can be shown that these two constructions are inverses up to isomorphism, and so we have bijections
\[
\left\{ \begin{array}{c} \mbox{isomorphism classes of} \\ \mbox{monoids} \end{array} \vphantom{\begin{array}{c} . \\ .\\. \end{array}}\right\} \longleftrightarrow \left\{\begin{array}{c} \mbox{isomorphism classes of} \\ \mbox{strong monoidal functors} \\ \mathbf{Th(Mon)} \to \set\end{array} \right\}.
\]
This shows that the strong monoidal functors from $\mathbf{Th(Mon)}$ to $\set$ classify all monoids.

In fact, the category $\mathbf{Th(Mon)}$ classifies not only monoids themselves, but also the maps between them. Indeed, given monoids $X$, $X'$ and corresponding strong monoidal functors $F$, $F'$, we also have a bijection
\[
\left\{\begin{array}{c} \mbox{monoid homomorphisms} \\ X \to X'\end{array} \right\} \longleftrightarrow \left\{\begin{array}{c} \mbox{monoidal natural transformations} \\ F \Rightarrow F' \end{array} \right\}.
\]
This bijection sends a monoid homomorphism $\varphi: X \to X'$ to the monoidal natural transformation defined on $M \in \mathbf{Th(Mon)}$ by $\varphi: FM \to F'M$. The requirement that monoid homomorphisms preserve the identity corresponds to the triangle \ref{mnt:tri} that monoidal natural transformations must obey, with the requirement that monoid homomorphisms preserve the monoid multiplication corresponds to the square \ref{mnt:sq}. It is further possible to show that these bijections respect composition of monoid homomorphisms and monoidal natural transformations. This shows that the category of monoids is equivalent to the category of strong monoidal functors from $\mathbf{Th(Mon)}$ to $\set$. It is in this strong sense that $\mathbf{Th(Mon)}$ classifies monoids, and for this reason we call this category the theory of monoids.

One advantage of this perspective is that we may now talk of monoid objects in other monoidal categories, which are often interesting structures in their own right. This often gives insight into the relationships between known mathematical structures. For example, the category of monoid objects in the monoidal category $(\mathbf{Ab},\ot, \mathbb{Z})$ of abelian groups with tensor product as monoidal product and $\mathbb{Z}$ as the monoidal unit can be shown to be precisely the category of rings.

We will use this idea of defining generalised monoid-like objects in categories other than $\set$ in pursuing a categorical definition of a causal theory. In particular, we will be interested in commutative comonoid objects.

\begin{defn}[Commutative comonoid]
As for in defining $\mathbf{Th(Mon)}$, fix a symbol $M$, and define the symmetric monoidal category $\mathbf{Th(CComon)}$ to be the symmetric monoidal category with objects tensor powers of $M$ and morphisms generated by the swaps and the maps $\comult: M \to M \ot M$, $\counit: M \to I$, subject to the relations
\[
\begin{aligned}
\begin{tikzpicture}
	\begin{pgfonlayer}{nodelayer}
		\node [style=none] (0) at (-0.5, 0.25) {};
		\node [style=none] (1) at (0.5, 0.25) {};
		\node [style=none] (2) at (0, -0.75) {};
		\node [style=none] (3) at (0, -0.25) {};
		\node [style=none] (4) at (0, 1) {};
		\node [style=none] (5) at (-0.5, 0.25) {};
		\node [style=none] (6) at (-0.5, 0.5) {};
		\node [style=none] (7) at (0, 1.5) {};
		\node [style=none] (8) at (-1, 1) {};
		\node [style=none] (9) at (-1, 1.5) {};
		\node [style=none] (10) at (0.5, 1.5) {};
	\end{pgfonlayer}
	\begin{pgfonlayer}{edgelayer}
		\draw [in=0, out=-90, looseness=0.75] (1.center) to (3.center);
		\draw [in=180, out=-90, looseness=0.75] (0.center) to (3.center);
		\draw (3.center) to (2.center);
		\draw (9.center) to (8.center);
		\draw (7.center) to (4.center);
		\draw [in=0, out=-90, looseness=0.75] (4.center) to (6.center);
		\draw [in=180, out=-90, looseness=0.75] (8.center) to (6.center);
		\draw (6.center) to (5.center);
		\draw (1.center) to (10.center);
	\end{pgfonlayer}
\end{tikzpicture}
\end{aligned}
\quad
=
\quad
\begin{aligned}
\begin{tikzpicture}
	\begin{pgfonlayer}{nodelayer}
		\node [style=none] (0) at (-0.5, 1.5) {};
		\node [style=none] (1) at (-0.5, 0.25) {};
		\node [style=none] (2) at (0.5, 0.25) {};
		\node [style=none] (3) at (0, -0.75) {};
		\node [style=none] (4) at (0, -0.25) {};
		\node [style=none] (5) at (1, 1) {};
		\node [style=none] (6) at (0.5, 0.25) {};
		\node [style=none] (7) at (0.5, 0.5) {};
		\node [style=none] (8) at (1, 1.5) {};
		\node [style=none] (9) at (0, 1) {};
		\node [style=none] (10) at (0, 1.5) {};
	\end{pgfonlayer}
	\begin{pgfonlayer}{edgelayer}
		\draw (0.center) to (1.center);
		\draw [in=0, out=-90, looseness=0.75] (2.center) to (4.center);
		\draw [in=180, out=-90, looseness=0.75] (1.center) to (4.center);
		\draw (4.center) to (3.center);
		\draw (10.center) to (9.center);
		\draw (8.center) to (5.center);
		\draw [in=0, out=-90, looseness=0.75] (5.center) to (7.center);
		\draw [in=180, out=-90, looseness=0.75] (9.center) to (7.center);
		\draw (7.center) to (6.center);
	\end{pgfonlayer}
\end{tikzpicture}
\end{aligned}
\tag{coassociativity}
\]
\[
\begin{aligned}
\begin{tikzpicture}
	\begin{pgfonlayer}{nodelayer}
		\node [style=none] (0) at (-0.5, 1) {};
		\node [style=none] (1) at (-0.5, 0.25) {};
		\node [style=none] (2) at (0.5, 0.25) {};
		\node [style=dot] (3) at (0.5, 0.5) {};
		\node [style=none] (4) at (0, -1) {};
		\node [style=none] (5) at (0, -0.25) {};
	\end{pgfonlayer}
	\begin{pgfonlayer}{edgelayer}
		\draw (0.center) to (1.center);
		\draw (3) to (2.center);
		\draw [in=0, out=-90, looseness=0.75] (2.center) to (5.center);
		\draw [in=180, out=-90, looseness=0.75] (1.center) to (5.center);
		\draw (5.center) to (4.center);
	\end{pgfonlayer}
\end{tikzpicture}
\end{aligned}
\quad
=
\quad
\begin{aligned}
\begin{tikzpicture}
	\begin{pgfonlayer}{nodelayer}
		\node [style=none] (0) at (0, 1) {};
		\node [style=none] (1) at (0, -1) {};
	\end{pgfonlayer}
	\begin{pgfonlayer}{edgelayer}
		\draw (1.center) to (0.center);
	\end{pgfonlayer}
\end{tikzpicture}
\end{aligned}
\quad
=
\quad
\begin{aligned}
\begin{tikzpicture}
	\begin{pgfonlayer}{nodelayer}
		\node [style=dot] (0) at (-0.5, 0.5) {};
		\node [style=none] (1) at (-0.5, 0.25) {};
		\node [style=none] (2) at (0.5, 0.25) {};
		\node [style=none] (3) at (0.5, 1) {};
		\node [style=none] (4) at (0, -1) {};
		\node [style=none] (5) at (0, -0.25) {};
	\end{pgfonlayer}
	\begin{pgfonlayer}{edgelayer}
		\draw (0) to (1.center);
		\draw (3.center) to (2.center);
		\draw [in=0, out=-90, looseness=0.75] (2.center) to (5.center);
		\draw [in=180, out=-90, looseness=0.75] (1.center) to (5.center);
		\draw (5.center) to (4.center);
	\end{pgfonlayer}
\end{tikzpicture}
\end{aligned}
\tag{counitality}
\]
and
\[
\begin{aligned}
\begin{tikzpicture}
	\begin{pgfonlayer}{nodelayer}
		\node [style=none] (0) at (0.5, 1) {};
		\node [style=none] (1) at (-0.5, 0.25) {};
		\node [style=none] (2) at (0.5, 0.25) {};
		\node [style=none] (3) at (-0.5, 1) {};
		\node [style=none] (4) at (0, -1) {};
		\node [style=none] (5) at (0, -0.25) {};
	\end{pgfonlayer}
	\begin{pgfonlayer}{edgelayer}
		\draw [in=90, out=-90, looseness=0.75] (0.center) to (1.center);
		\draw [in=90, out=-90, looseness=0.75] (3.center) to (2.center);
		\draw [in=0, out=-90, looseness=0.75] (2.center) to (5.center);
		\draw [in=180, out=-90, looseness=0.75] (1.center) to (5.center);
		\draw (5.center) to (4.center);
	\end{pgfonlayer}
\end{tikzpicture}
\end{aligned}
\quad
=
\quad
\begin{aligned}
\begin{tikzpicture}
	\begin{pgfonlayer}{nodelayer}
		\node [style=none] (0) at (-0.5, 1) {};
		\node [style=none] (1) at (-0.5, 0.25) {};
		\node [style=none] (2) at (0.5, 0.25) {};
		\node [style=none] (3) at (0.5, 1) {};
		\node [style=none] (4) at (0, -1) {};
		\node [style=none] (5) at (0, -0.25) {};
	\end{pgfonlayer}
	\begin{pgfonlayer}{edgelayer}
		\draw (0.center) to (1.center);
		\draw (3.center) to (2.center);
		\draw [in=0, out=-90, looseness=0.75] (2.center) to (5.center);
		\draw [in=180, out=-90, looseness=0.75] (1.center) to (5.center);
		\draw (5.center) to (4.center);
	\end{pgfonlayer}
\end{tikzpicture}
\end{aligned}
\tag{commutativity}
\]

A \emph{commutative comonoid} in a symmetric monoidal category $\c$ is a strong symmetric monoidal functor $\mathbf{Th(CComon)} \to \c$. Abusing our terminology slightly, we will often just say that the image of $M$ under this functor is a commutative comoniod.
\end{defn}

\chapter{Categorical Probability Theory}

Probability theory concerns itself with \emph{random variables}: properties of a system that may take one of a number of possible outcomes, together with a likelihood for each possible observation regarding the property. We call the property itself the \emph{variable}, and together the likelihoods for the observations form a \emph{probability assignment} for the variable. As we will mainly concern ourselves with relationships between random variables, we have particular interest in rules that specify a probability assignment on one variable given a probability assignment on another---this can be seen as the latter variable having some causal influence on the former. 

In this chapter we will develop the standard tools to talk about all these things, but with emphasis on a categorical perspective. These categorical ideas originate with Lawvere \cite{L}, and were extended by Giry in \cite{Gi}. We caution that the terminology we have used for the basic concepts in probability is slightly nonstandard, but predominantly follows that of Pearl \cite{P} and the Bayesian networks community. Although it will not affect the mathematics, we will implicitly take a frequentist view of probability to complement our physical interpretation of causality.

\section{The category of measurable spaces}

The idea of a variable is captured by measurable spaces. These consist of a set $X$, thought of as the set of `outcomes' of the variable, and a collection $\S$ of subsets of $X$ obeying certain closure properties, which represent possible observations about $X$ and which we call the measurable sets of $X$. We then talk of probability assignments on these measurable spaces via a function $P: \S \to [0,1]$ satisfying some consistency properties. While the collection of measurable sets is often taken to be the power set $\mathcal P(X)$ when $X$ is finite, for larger sets some restrictions are usually necessary if one wants to assign interesting collections of probabilities to the space.

Given a measurable set $A$, we think of the number $P(A)$ as the chance that the outcome of the random variable with outcomes represented by $X$ will lie in the subset $A$. As an example, the process of rolling a six-sided die can be described by the measurable space with set of outcomes $X = \{1,2,3,4,5,6\}$ and measurable subsets $\S = \mathcal P(X)$ all subsets of $X$. The statement that the die is fair is then the statement that the probability associated to any subset $A \subseteq X$ is $P(A) = \tfrac16 |A|$. 

We formalise this in the following standard way; more details can be found in \cite{Ash} or \cite{SS}, or indeed any introductory text to probability theory.

\begin{defns}[$\s$-algebra, measurable space]
Given a set $X$, a \emph{$\s$-algebra} $\S$ on $X$ is a set $\S$ of subsets of $X$ that contains the empty set and is closed under both countable union and complementation in $X$. We call a pair $(X,\S)$ consisting of a set $X$ and a $\s$-algebra $\S$ on $X$ a \emph{measurable space}. 
\end{defns}

On occasion we will just write $X$ for the measurable space $(X,\S)$, leaving the $\s$-algebra implicit. In these cases we will write $\S_X$ to mean the $\s$-algebra on $X$.

\begin{ex}[Discrete and indiscrete measurable spaces]
Let $X$ be a set. The power set $\mathcal P(X)$ of $X$ forms a $\s$-algebra, and we call $(X,P(X))$ a \emph{discrete} measurable space. At the other extreme, distinct whenever $X$ has more than one element, is the $\s$-algebra $\{\varnothing, X\}$. In this case we call $(X,\{\varnothing, X\})$ an \emph{indiscrete} measurable space.
\end{ex}

Even beyond the two of the above example, it is not hard to find $\s$-algebras: we may construct one from any collection of subsets. Indeed, we say that the $\s$-algebra $\S(\mathcal G)$ generated by a collection $\mathcal G= \{G_i\}_{i \in I}$ of subsets of a set $X$ is the intersection of all $\s$-algebras on $X$ containing $\mathcal G$. An explicit construction can be given by taking all countable intersections of the sets in $\mathcal G$ and their complements, and then taking all countable unions of the resulting sets. We say that a measurable space is \emph{countably generated} if there exists a countable generating set for it.

\begin{ex}[Borel measurable spaces]
Many frequently used examples of measurable spaces come from topological spaces. The \emph{Borel $\s$-algebra} $\mathcal B_X$ of a topological space $X$ is the $\s$-algebra generated by the collection of open subsets of the space. 
\end{ex}

\begin{ex}[Product measurable spaces] \label{ex:prodsga}
Given measurable spaces $(X,\S_X)$, $(Y,\S_Y)$, we write $\S_X \ot \S_Y$ for the $\s$-algebra on $X \times Y$ generated by the collection subsets $\{A \times B \subseteq X \times Y \mid A \in \S_X, B \in \S_Y\}$. We call this the \emph{product $\s$-algebra} of $\S_X$ and $\S_Y$, and call the resulting measurable space $(X \times Y, \S_X \ot \S_Y)$ the \emph{product measurable space} of $(X,\S_X)$ and $(Y,\S_Y)$. Note that as $(A \times B) \cap (A' \times B') = (A \cap A') \times (B \cap B')$ and $(A \times B)^c = (A \times B^c) \cup (A^c \times B)$, we may write
\[
\S_X \ot \S_Y = \left\{ \bigcup_{i \in I} (A_i \times B_i) \,\bigg| \, A_i \in \S_X, B_i \in \S_Y\right\}.
\] 
\end{ex}

The product measurable space is in fact a categorical product in the category of measurable spaces. To understand this, we first must specify the notion of morphism corresponding to measurable spaces. Just as continuous functions reflect the open sets of a topology, the important notion of map for measurable sets is that of functions that reflect measurable sets. 

\begin{defn}[Measurable function]
A function $f:X \to Y$ between measure spaces $(X,\S_X)$ and $(Y,\S_Y)$ is called \emph{measurable} if for each $A \in \S_Y$, $f\i(A) \in \S_X$.
\end{defn}

We write $\meas$ for the category of measurable spaces and measurable functions. It is easily checked that this indeed forms a category with composition simply composition of functions. 

It is also not difficult to check that the product measurable space $(X \times Y, \S_X \ot \S_Y)$ is the product of the measurable spaces $(X,\S_X)$ and $(Y, \S_Y)$ in this category. As the projection maps $\pi_X: X \times Y \to X$ and $\pi_Y: X \times Y \to Y$ of the set product are measurable maps, it is enough to show that for any measurable space $(Z,\S_Z)$ and pair of measurable functions $f: (Z,\S_Z) \to (X,\S_X)$, and $g: (Z,\S_Z) \to (X, \S_X)$ the unique function $\langle f,g \rangle: Z \to X \times Y$ given by the product in $\set$ is a measurable function. Since for all countable collections $\{A_i \times B_i\}_{i \in I}$ of subsets of $X \times Y$ we have
\[
\langle f,g \rangle\i\bigg(\bigcup_{i \in I}(A_i \times B_i)\bigg) = \bigcup_{i \in I} \langle f,g \rangle\i(A_i \times B_i) = \bigcup_{i \in I}(f\i(A_i) \cap g\i(B_i)),
\]
this is indeed true. 

Note also that any one point set $1= \{\ast\}$ with its only possible $\s$-algebra $\{\varnothing, 1\}$ is a terminal object in $\meas$. We thus may immediately view $\meas$ as a symmetric monoidal category, with the symmetric monoidal structure given by the fact that $\meas$ has finite products. The swaps $\s_{X,Y}: X \times Y \to Y \times X;(x,y) \mapsto (y,x)$ of $\meas$ are the same as those of the symmetric monoidal category $\set$. We shall by default consider $\meas$ as a symmetric monoidal category in this way. 

Similarly, we may also show that the full subcategories $\mathbf{FinMeas}$ and $\mathbf{CGMeas}$ with objects finite measurable spaces and countably generated measurable spaces respectively are also a symmetric monoidal category with monoidal product the categorical product.

\section{Measures and integration}

The reason we deal with measurable spaces is that these form the basic structure required for an object to carry some idea of a probability distribution. More precisely, we deal with measurable spaces because they can be endowed with probability measures. 

\begin{defns}[Measure, measure space]
Given a measurable space $(X,\S)$, a \emph{measure} $\mu$ on $(X,\S)$ is a function $\mu: \S \to \rr_{\ge 0}\cup\{\infty\}$ such that: 
\begin{enumerate}[(i)] 
\item the empty set $\varnothing$ has measure $\mu(\varnothing) = 0$; and 
\item if $\{A_i\}_{i \in I}$ is a countable collection of disjoint measurable sets then $\mu(\cup_{i \in I}A_i) = \sum_{i \in I}\mu(A_i)$. 
\end{enumerate}
Any such triple $(X,\S,\mu)$ is then known as a \emph{measure space}. When $\mu(X) =1$, we further call $\mu$ a \emph{probability measure}, and $(X,\S,\mu)$ a \emph{probability space}.
\end{defns}

We will have to pay close attention to the properties of the collections, in fact $\s$-ideals, of sets of measure zero of probability spaces in the following. These represent possible observations of our random variable that nonetheless are `never' observed, giving us very little information about their causal consequences. Very often we will pronounce functions equal `almost everywhere' if they agree but for a set of a measure zero. More generally, we say a property with respect to a measure space is true \emph{almost everywhere} or for \emph{almost all} values if it holds except on a set of measure zero. We also say that a measure space is of \emph{full support} if its only subset of measure zero is the empty set $\varnothing$. Such spaces are necessarily countable measure spaces.

\begin{ex}[Finite and countable measurable spaces] \label{ex:fms}
We shall say that a measurable space $(X, \S)$ is a \emph{finite measurable space} if $\S$ is a finite set. In this case there exists a finite generating set $\{A_1, \dots, A_n\}$ for $\S$ consisting of pairwise disjoint subsets of $X$, and measures $\mu$ on $(X,\S)$ are in one-to-one correspondence with to functions $m: \{A_1, \dots, A_n\} \to \rr_{\ge 0}$, with $\mu(A) = \sum_{A_i \subseteq A} m(A_i)$ for all measurable subsets $A$ of $X$. Measures may thus also be thought of as vectors with non-negative entries in $\rr^{n}$, with probability measures those vectors whose entries also sum to 1. We may similarly define \emph{countable measurable spaces}, and note that measures on these spaces are in one-to-one correspondence with functions $m: \mathbb{N} \to \rr_{\ge 0}$.

Writing $n$ for some chosen set with $n \in \mathbb{N}$ elements, note that this suggests each finite measurable space is in some sense `isomorphic' to $n= (n,\mathcal P(n))$ for some $n$. Although this is not true in $\meas$, we will work towards constructing a category in which this is true.
\end{ex}

We give two more useful examples of measures.

\begin{ex}[Borel measures, Lebesgue measure]
A \emph{Borel measure} is a measure on a Borel measurable space. An important collection of examples of these are the Lebesgue measures on $(\rr^n,\mathcal B_{\rr^n})$. These may be characterised as the unique Borel measure on $\rr^n$ such that the measure of each closed $n$-dimensional cube is given by its $n$-dimensional volume. See any basic text on measure theory, such as \cite[Chapter 1]{SS}, for more details.

When speaking of $\rr$ as a measure space, we will mean $\rr$ with its Borel $\s$-algebra and Lebesgue measure. In particular, when referring to a real-valued measurable function, we shall take the codomain as having this structure.
\end{ex}

\begin{ex}[Product measures]
Given measure spaces $(X,\S_X,\mu)$ and $(Y,\S_Y,\nu)$, we may define the \emph{product measure} $\mu \times \nu$ on the product measurable space $(X \times Y, \S_X \ot \S_Y)$ as the unique measure on this space such that for all $A \in \S_X$ and $B \in \S_Y$,
\[
\mu\times\nu (A \times B) = \mu(A)\nu(B).
\]
A proof of the existence and uniqueness of such a measure can be found in \cite[Theorem 6.1.5]{SS}.
\end{ex}

One way in which measures interact with measurable functions is that measures may be `pushed forward' from the domain to the codomain of a measurable map. 

\begin{defn}[Push-forward measure]
Let $(X,\S_X,\mu)$ measure space, $(Y,\S_Y)$ measurable space, and $f:X \to Y$ be a measurable function. We then define the \emph{push-forward measure $\mu_f$ of $\mu$ along $f$} to be the map $\S_Y \to \rr$ given by
\[
\mu_f(B) = \mu(f\i(B)).
\]
Note that $\mu_f(Y) = \mu(f\i(Y)) = \mu(X)$, so the push-forward of a probability measure is again a probability measure. 
\end{defn}

As causality concerns the relationships between random variables, we shall be particularly interested in measures on product spaces, so-called \emph{joint measures}. An important example of a push-forward measure is that of the \emph{marginals} of a joint measure. These are the push-forward measures of a joint measure along the projections of the product space: given a joint measure space $(X \times Y, \S_X \ot \S_Y, \mu)$ with projections $\pi_X: X \times Y \to X$ and $\pi_Y: X \times Y \to Y$, we define the marginal $\mu_X$ of $\mu$ on $X$ to be the push forward measure of $\mu$ along $\pi_X$, and similarly for $\mu_Y$. We also say that we have \emph{marginalised over $Y$} when constructing the marginal $\mu_X$ from the measure $\mu$. Note that the marginals of a joint probability measure are again probability measures.

Observe that for each point $x$ in its domain, a measurable function $f:(X,\S_X) \to (Y, \S_Y)$ induces a `point measure' 
\[
\d_{f,x}(B) = \begin{cases} 1 & \textrm{if } f(x) \in B, \\ 0 & \textrm{if } f(x) \notin B, \end{cases}
\]
on its codomain $(Y,\S_Y)$. From this point of view, the push-forward measure of some measure $\mu$ on $(X,\S_X)$ along $f$ can be seen as taking the `$\mu$-weighted average' or `expected value' of these induced point measures on $(Y,\S_Y)$. More precisely, the push-forward measure may be defined as the integral of these point measures with respect to $\mu$.\footnote{A complementary perspective views the integral in terms of push-forwards, but only once we have define the standard notion of multiplying functions with measures to produce a new measure. Indeed, given a bounded real-valued measurable function $f$ and a measure $\mu$ on a measurable space $(X,\S)$, this new measure $f\mu$ is equal to $\int_A f \,d\mu$ on each $A \in \S$, and this allows us to see the integral $\int_A f \,d\mu$ as the value, on the set $\ast$, of the push-forward measure of the measure $\chi_Af\mu$ along the unique map $X \to \ast$ to the terminal object.}

For the sake of completeness, we quickly review the definition of the integral for bounded real-valued measurable functions nonzero on a set of finite measure, but the reader is referred to Ash \cite[\textsection1.5]{Ash} or Stein and Shakarchi \cite[Chapter 2]{SS} for full detail. 

We first define the integral of simple functions. Let $(X,\S,\mu)$ be a measure space, and let $A$ be subset of $X$. We write $\chi_A: X \to \rr$ for the characteristic function 
\[
\chi(x) = \begin{cases} 1 & \textrm{if } x \in A; \\ 0 & \textrm{if } x \notin A, \end{cases}
\]
and call a weighted sum $\varphi = \sum_{k=1}^N c_k \chi_{A_k}$ of characteristic functions of measurable sets a \emph{simple function}. The integral $\int_A \varphi \,d\mu$ of a simple function $\varphi$ over the measurable set $A$ with respect to $\mu$ is defined to be
\[
\int_A \varphi \,d\mu = \sum_{k=1}^N c_k \mu(A_k \cap A)
\]
when this sum is finite. Note that this implies that the integral over $X$ of the characteristic function of a measurable set $A$ is just $\mu(A)$.

Let now $f$ be a bounded real-valued measurable function such that the set $\{x \in X \mid f(x) \ne 0\}$ is of finite measure. It can be shown there then exists a uniformly bounded sequence $\{\varphi_n\}_{n\in \mathbb N}$ of simple functions supported on the support of $f$ and converging to $f$ for almost all $x$. Using this sequence, we define the integral $\int_Af\,d\mu$ of $f$ over $A$ with respect to $\mu$ to be
\[
\int_A f \, d\mu = \lim_{n \to \infty} \int_A \varphi_n \, d\mu.
\]
By our assumptions, this limit always exists, is finite, and is independent of the sequence $\{\varphi_n\}_{n\in \mathbb N}$. Where we do not write the domain of integration $A$, we mean that the integral is taken over the entire domain $X$ of $f$.

We will not discuss the technicalities of the integral further, but instead note that in the case of Lebesgue measure the notion of integration agrees with that of Riemann integration, and for finite measure spaces it can be viewed as analogous to matrix multiplication---this will be explained fully in the following section. Our examples will be limited to these cases.

More generally, this idea of averaging measures will play a crucial role in how we reason about consequences of causal relationships. As an illustration, suppose that we have measurable spaces $C$ and $R$, representing say cloud cover and rain on a given day respectively, and for each value of cloud cover---that is, each measurable set in $C$---we are given the probability of rain. We will assume this forms a real-valued measurable function $f$ on $C$. If we are further given a measure $\mu$ on $C$ representing how cloudy a day is likely to be, we can `average' over this measure to give a probability of rain on that day. This averaging process is given by the integral of $f$ with respect to $\mu$. 

Implicitly here we are talking about conditional probabilities---for each outcome of the space $C$ we get a measure on $R$. This idea will form our main idea of map between measurable spaces.

\section{The category of stochastic maps}

Measurable functions describe a deterministic relationship between two variables: if one fixes an outcome of the domain variable, a measurable function specifies a unique corresponding outcome for the codomain. When describing a more stochastic world, such as that given by a Markov chain, such certainty is often out of reach. In these cases stochastic maps---variously also called stochastic kernels, Markov kernels, conditional probabilities, or probabilistic mappings---may often be useful instead. These are more general, mapping outcomes of the domain to probability measures on, instead of points of, the codomain. 

\begin{defn}[Stochastic map]
Let $(X,\S_X)$ and $(Y,\S_Y)$ be measurable spaces. A \emph{stochastic map} $k: (X, \S_X) \to (Y, \S_Y)$ is a function
\ba
k(x,B): X \times \S_Y \lra [0,1]
\ea
such that 
\begin{enumerate}[(i)]
\item for each $x \in X$ the function $k_x := k(x,-): \S_Y \to [0,1]$ is a probability measure on $Y$; and 
\item for each measurable set $B \subseteq Y$ the function $k_B:= k(-,B): X \to [0,1]$ is measurable.
\end{enumerate} 
The \emph{composite of stochastic maps}
\[
\ell \circ k: (X,\S_X) \stackrel{k}\longrightarrow (Y,\S_Y) \stackrel{\ell}\longrightarrow (Z,\S_Z)
\]
is defined by the integral
\[
\ell \circ k(x,C) = \int_Y \ell(-,C) \,dk_x,
\]
where $x \in X$ and $C \in \S_Z$. That this is a well-defined stochastic map follows immediately from the basic properties of the integral.
\end{defn}

Note that these definitions are those suggested by our discussion at the close of the previous section: put more succinctly, a stochastic map is a measure-valued function (subject to a measurability requirement), and the composite $\ell \circ k(x,C)$ of stochastic maps $\ell$ and $k$ is given by integrating the measures $\ell_y$ on the codomain $(Z,\S_Z)$ with respect to the measure $k_x$ on the intermediate variable $(Y, \S_Y)$. 

We give a few examples.

\begin{ex}[Probability measures as stochastic maps] \label{ex:pmasm}
Observe that a stochastic map $k: 1 \to (X,\S)$ is simply a probability measure on $(X,\S)$. 
\end{ex}

\begin{ex}[Deterministic stochastic maps] \label{ex:dsm}
In the previous section we discussed how a measurable function  $f: (X,\S_X) \to (Y,\S_Y)$ induces `point measures' on its codomain. We can now interpret these as defining the stochastic map $\d_f: (X,\S_X) \to (Y,\S_Y)$ given by
\ba
\d_f: X\times \S_Y &\longrightarrow [0,1]; \\
(x,B) &\longmapsto \d_{f,x}(B)= \begin{cases} 1 & \textrm{if } f(x) \in B; \\ 0 & \textrm{if } f(x) \notin B. \end{cases}
\ea
We call this the \emph{deterministic stochastic map induced by $f$}. More generally, we call any stochastic map taking values in only in the set $\{0,1\}$ a \emph{deterministic stochastic map}.

Observe that given measurable functions $f: (X,\S_X) \to (Y,\S_Y)$ and $g: (Y,\S_Y) \to (Z, \S_Z)$, the composite of their induced maps is given by
\[
\d_g \circ \d_f(x,C) = \int_Y \d_g(-,C)\, d\d_{f,x} = \d_{f,x}(g\i(C))= \begin{cases} 1 & \textrm{if } g\circ f(x) \in C; \\ 0 & \textrm{if } g\circ f(x) \notin C, \end{cases}
\]
where $x \in X$ and $C \in \S_Z$. Thus $\d_g \circ \d_f = \d_{g \circ f}$. 

More generally, for a stochastic map $k$ and measurable function $f$ of the types required for composition to be well-defined, we have $k \circ \d_f(x,B) = k(f(x),B)$, and $\d_f \circ k(x,B) = k(x,f\i(B))$.
\end{ex}

\begin{ex}[Stochastic matrices] \label{ex:fms2}
Let $X$ and $Y$ be finite measurable spaces of cardinality $n$ and $m$ respectively. Note that if $x,x' \in X$ are such that $x$ lies in a measurable set $A$ if and only if $x'$ does, then for any stochastic map $k: X \to Y$ the measurability of $k_B$ for each $B \in \S_Y$ implies the measures $k_x$ and $k_{x'}$ must be equal. Thus, with reference to Example \ref{ex:fms}, we may assume without loss of generality $X$ and $Y$ are discrete. Then, observing that all maps with discrete domain are measurable and recalling that probability distributions on a finite discrete measurable space $m$ may be considered as vectors in $\rr^m$ with non-negative entries that sum to one, we see that stochastic maps $k: X \to Y$ may be considered as $m \times n$ matrices $K$ with non-negative entries and columns summing to one. Indeed, the correspondence is given by having the $yx$th entry $K_{y,x}$ of $K$ equal to $k(x,\{y\})$ for all $x \in X$ and $y \in Y$. We call such matrices---matrices with entries in $[0,1]$ and columns summing to 1---\emph{stochastic matrices}.

Let also $Z$ be a discrete finite measurable space, and let $\ell: Y \to Z$ be a stochastic map, with corresponding stochastic matrix $L$. Then for all $x \in X$ and $z \in Z$, we have
\[
\ell \circ k(x,\{z\}) = \int_Y \ell(-,\{z\}) \,dk_x = \sum_{y \in Y} \ell(y,\{z\})k(x,\{y\}),
\]
and writing this in matrix notation then gives
\[
\ell \circ k(x,\{z\}) = \sum_{y \in Y} L_{z,y}K_{y,x} = (LK)_{z,x}.
\]
Thus our representation of finite stochastic maps as stochastic matrices respects composition. This hints at an equivalence of categories.
\end{ex}

We are now in a position to define the main category of interest: let the category of stochastic maps, denoted $\stoch$, be the category with objects measurable spaces and morphisms stochastic maps. It is straightforward to show this is a well-defined category. In particular, the associativity of the composition rule follows directly from the monotone convergence theorem \cite[Theorem 1]{Gi}, and for each object $(X,\S)$ of $\stoch$ the delta function $\d: (X,\S) \to (X,\S)$ defined by
\[
\d(x,A) = \begin{cases} 1 & \textrm{if } x \in A; \\ 0 & \textrm{if } x \notin A, \end{cases}
\]
---that is, the deterministic stochastic map induced by the identity function on $X$---is the identity map.

Viewed with this new category, Example \ref{ex:dsm} defines a functor $\d: \meas \to \stoch$. In fact, we may further endow $\stoch$ with a symmetric monoidal structure such that this is a symmetric monoidal functor. For this we take the product of two objects to be their product measurable space, and the product 
\[
k \ot \ell: (X \times Z, \S_X \ot \S_Z) \to (Y\times W, \S_Y \ot \S_W)
\]
of two stochastic maps $k: (X,\S_X) \to (Y,\S_Y)$ and $\ell: (Z,\S_Z) \to (W,\S_W)$ to be the unique stochastic map extending
\[
k \ot \ell \big((x,z),B \times D\big) = k(x,B)\ell(z,D),
\]
where $x \in X$, $z \in Z$, $B \in \S_Y$ and $D \in \S_W$. This assigns to each pair $(x,z)$ the product measure of $k_x$ and $\ell_z$ on $Y \times W$, and indeed results in a well-defined functor $\ot: \stoch \times \stoch \to \stoch$. Using as structural maps the induced deterministic stochastic maps of the corresponding structural maps in $\meas$ then gives $\stoch$ the promised symmetric monoidal structure.

\begin{remark}
Observe that any indiscrete $\s$-algebra $(\ast,\{\varnothing,\ast\})$ is a terminal object in $\stoch$: from any other measurable space $(X,\S)$ there only exists the map 
\[
t(x,B) = \begin{cases} 1 & \textrm{if } B = \ast; \\ 0 & \textrm{if } B = \varnothing. \end{cases}
\] 
Example \ref{ex:pmasm} thus shows that the points of an object of $\stoch$ are precisely the probability measures on that space.
\end{remark}

While $\stoch$ has a straightforward definition and interpretation, the generality of the concept of a $\s$-algebra means that $\stoch$ admits a few pathological examples that indicate it includes more than what we want to capture. For this reason, and for the clarity that simpler cases can bring, we will mostly work with two full subcategories of $\stoch$. The first is $\finstoch$, the category of finite measurable spaces and stochastic maps. Building on Example \ref{ex:fms2}, and as promised in Example \ref{ex:fms}, this is monoidally equivalent to the skeletal symmetric monoidal category $\mathbf{SMat}$ with objects natural numbers and morphisms stochastic matrices. As categories of vector spaces are well studied, this characterisation gives much insight into the structure of $\finstoch$.

The main disadvantage of $\finstoch$ is that many random variables are not finite. One category admitting infinite measure spaces---and used by Giry \cite{Gi}, Panangaden \cite{Pa}, and Doberkat \cite{Do}, among others---is the category of standard Borel spaces,\footnote{A measurable space is a standard Borel space if it is the Borel measurable space of some Polish space. A topological space is a Polish space if it is the underlying topological space of some complete separable metric space. This category then has objects standard Borel spaces and morphisms stochastic maps between them.} which can be skeletalised as the countable measurable spaces and the unit interval with its Borel $\s$-algebra. We will favour the less frequently used but slightly more general category $\cgst$, the full subcategory of $\stoch$ obtained by restricting the objects to the countably generated measurable spaces. This setting is general enough to handle almost all examples of probability spaces that arise in applications, but has a few nice properties that $\stoch$ does not. In the next section we see one of them: the deterministic stochastic maps here are precisely those that arise from measurable functions.

\section{Deterministic stochastic maps}

Recall that the deterministic stochastic maps are those that take only the values $0$ and $1$. These will play a crucial role in maps between collections of causally related random variables. The key reason for this is that these maps show much more respect for the structure of the measurable spaces than general stochastic maps. For example, for a stochastic map to be an isomorphism in $\stoch$, it must be deterministic.

\begin{prop}
Let $k: (X, \S_X) \to (Y,\S_Y)$ be an isomorphism in $\stoch$. Then $k$ is deterministic.
\end{prop}
\begin{proof}
Our argument rests on the fact that if $f: X \to [0,1]$ is a measurable function on a probability space $(X,\S,\mu)$ such that $\int f\,d\mu =1$, then $\mu(f\i\{1\}) = 1$. 

Write $h$ for the inverse stochastic map to $k$, and fix $B \in \S_Y$. We begin by defining $A =k_B\i\{1\}$, where we remind the reader that $k_B$ is the measurable function $k(-,B): X \to [0,1]$. Note that we then have $B \subseteq h_A\i\{1\}$, since for any $y \in B$ that $h$ is the inverse to $k$ gives $\int k(-,B)\,dh_y =1$, so by the above fact $h_y(A) = 1$, and hence $y \in h_A\i\{1\}$. 

It is enough to show that for any $x \in X$, $k(x,B) = 0$ or $1$. If $x \in A$ we are done: by definition then $k(x,B) =1$. Suppose otherwise. Then, again as $h$ and $k$ are inverses, $\int h(-,A) \,dk_x = 0$. But 
\[
\int h(-,A)\,dk_x \ge k_x\big(h_A\i\{1\}\big) \ge k_x(B).
\]
Thus $k(x,B) =0$, as required.
\end{proof}

In the previous section, we showed that every measurable function induces a deterministic stochastic map. One of the reasons that we prefer to work with countably generated measurable spaces is that in $\cgst$ the converse is also true.

\begin{prop} \label{prop.detfunc}
Let $(Y, \S_Y)$ be a measurable space with $\S_Y$ countably generated. Then a stochastic map $k:X \to Y$ is deterministic if and only if there exists a measurable function $f:X \to Y$ with $k = \d_f$.
\end{prop}
\begin{proof}
A proof can be found in \cite[Proposition 2.1]{CuS}, but we outline a version here to demonstrate the use of the countable generating set, and point out that we assume the axiom of choice.

We have seen that measurable functions induce deterministic maps. For the converse, let $\mathcal G$ be a countable generating set for $\S_Y$. Now for each $x \in X$ let $B_x =\bigcap_{\{B \in \mathcal G\,\mid\, k(x,B) =1\}} B$. This is a measurable set as $\mathcal G$ is countable, and has $k_x$-measure 1 as its complement may be written as a countable union of sets of $k_x$-measure zero. Choosing then for each $x$ some $y \in B_x$, we define $f$ such that $f(x) =y$. It is then easily checked that $k = \d_f$, and $f$ is measurable as each $k_B$ is.
\end{proof}

\begin{remark}
On the other hand, one need not look too hard for a deterministic stochastic map that is not induced by a measurable function when dealing with non-countably generated measurable spaces.\footnote{For fun, we note that if we further add the requirement that every subset of our codomain be measurable, then we \emph{do} need to look quite hard. We say that a cardinal is a measurable cardinal if there exists a countably-additive two-valued measure on its power set such that it has measure 1 and each point has measure 0. If we are looking for such measures, then our set has to be a strongly inaccessible cardinal \cite{Ulam}. These are truly huge; in some models of set theory they're too huge to exist!} Indeed, take any uncountable set $X$, and endow it with the $\s$-algebra generated by the points of $X$. This means that set is measurable if and only if it or its complement is countable. It is then easily checked that assigning countable sets measure 0 and uncountable measurable sets measure 1 defines a measure. This gives a deterministic stochastic map from the terminal object to $X$ not induced by any measurable function.
\end{remark}

\begin{remark}
Note that the measurable function specifying a deterministic stochastic map need not be unique, so we should not view the deterministic stochastic maps as merely the collection of measurable maps lying inside $\cgst$. As an example of this, consider the one point measurable space $(\ast,\{\varnothing, \ast\})$ and any other indiscrete measurable space $(X,\S)$. Then all of the $|X|$ functions $f: \ast \to X$ are measurable, and all induce the deterministic stochastic map $\d_f(\ast,\varnothing) = 0$; $\d_f(\ast,X) =1$. In this way $\stoch$ captures the intuition that every indiscrete measuarable space is the same. 

In particular, non-bijective measurable endofunctions can induce the identity stochastic map, so measurable spaces may be isomorphic in $\stoch$ even if they are not isomorphic in $\meas$. This lets $\finstoch$ admit the skeletalisation $\mathbf{SMat}$, even while the classification of isomorphic objects in $\mathbf{FinMeas}$ is not nearly so neat.

More abstractly, this shows that although our symmetric monoidal functor $\d: \meas \to \stoch$ is injective on objects, it is not faithful, and so we can not view $\meas$ as a subcategory of $\stoch$.
\end{remark}

Although inducing deterministic stochastic maps from measurable functions is in general a many-to-one process, we may always take quotients of our measurable spaces so it becomes one-to-one. We briefly explore this idea in order to further our understanding of deterministic stochastic maps. 

Call two outcomes of a measurable space \emph{distinguishable} if there exists a measurable subset containing one but not the other, and \emph{indistinguishable} otherwise. Indistinguishability gives an equivalence relation on the outcomes of a measure space. We may take a quotient by this equivalence relation, and use the quotient map to induce a $\s$-algebra on the quotient set, defining a set in the quotient to be measurable if its preimage is. In this quotient space all outcomes are distinguishable; we call this an \emph{empirical measurable space}. The quotient map in fact induces an isomorphism in $\stoch$.

We call a deterministic monomorphism in $\stoch$ an \emph{embedding} of measurable spaces. These are deterministic stochastic maps induced by injective measurable functions on the empiricisations, and so may be thought of as maps that realise a measurable space as isomorphic to a sub-measurable space of another. We call an epimorphism in $\stoch$ a \emph{coarse graining} of measurable spaces. These are deterministic stochastic maps induced by surjective measurable functions on the emipiricisations, and so may be thought of as maps that remove the ability to distinguish between some outcomes of the domain.

The following proposition then gives a precise understanding of deterministic stochastic maps in $\cgst$. 

\begin{prop} \label{pr.dsms}
In $\cgst$, every deterministic stochastic map may be factored as a coarse graining followed by an embedding. 
\end{prop}
\begin{proof}
We may without loss of generality assume spaces are empirical. Then we may treat the deterministic stochastic maps as functions, and we know that each function factors into a surjection followed by an injection.
\end{proof}

\section{Aside: the Giry monad}

To shed further light on the close relationship between $\meas$ and $\stoch$, we mention a few results that first stated in \cite{L}, and proved in \cite{Gi}. The main observation is that $\stoch$ forms a relation-like version of $\meas$. More precisely, we observe that just as $\rel$ is the Kleisli category for the power set monad on $\set$, $\stoch$ is the Kleisli category for Giry monad on $\meas$. 

Recall that a monad on a category $\c$ consists of a functor $T: \c \to \c$ and natural transformations $\eta: 1_\c \Rightarrow T$ and $\mu: T^2 \Rightarrow T$ such that for all objects $X \in \c$ the diagrams
\[
\xymatrixcolsep{3pc}
\xymatrixrowsep{3pc}
\begin{aligned}
\xymatrix{
T(T(T(X))) \ar[d]_{\mu_{T(X)}} \ar[r]^{T(\mu_X)} & T(T(X)) \ar[d]^{\mu_X} \\
T(T(X)) \ar[r]_{\mu_X} & T(X) 
}
\end{aligned}
\qquad
\textrm{and}
\qquad
\begin{aligned}
\xymatrix{
T(X) \ar[d]_{T(\eta_X)} \ar[r]^{\eta_{T(X)}} \ar@{=}[rd] & T(T(X)) \ar[d]^{\mu_X} \\
T(T(X)) \ar[r]_{\mu_X} & T(X)
}
\end{aligned}
\]
commute. Also recall that the Kleisli $\c_T$ category of such a monad on $\c$ is the category with objects that of $\c$, for all $X, Y \in \c$ homsets $\hom_\c(X, TY)$, and composition of $f^\ast: X_T \to Y_T$, $g^\ast: Y_T \to Z_T$, defined by $f: X \to TY$, $g: Y \to TZ$, given by $g^\ast \circ_T f^\ast = (\mu \circ Tg \circ f)^\ast$.

As mentioned above, it can be checked that the functor mapping a set to its power set can be viewed as a monad on $\set$, and the Kleisli category for this monad is isomorphic to $\rel$. In the case of $\meas$ and $\stoch$, we define the functor of the Giry monad $\mathcal P: \meas \to \meas$ to be the functor taking a measurable space $(X,\S)$ to the set of all probability measures on $(X,\S)$ with the smallest $\s$-algebra such that the evaluation maps 
\ba
\mathcal PX \times \S &\longrightarrow [0,1]; \\
(\mu,B) &\longmapsto \mu(B).
\ea
are measurable.\footnote{We earlier saw hints that a stochastic map may be viewed as a measure-valued measurable function. We now see the precise meaning of this statement: a stochastic map is defined by a measurable function $X \to \mathcal PX$.} The associated natural transformations of the monad are that $1 \to \mathcal P$ sending a point to its point measure, and that $\mathcal P^2 \to \mathcal P$ sending a measure on the set of measures to its integral. It can then be shown that this forms a well-defined monad, with Kleisli category $\stoch$.

As $\stoch$ is the Kleisli category for $\mathcal P$, $\mathcal P$ can be factored through $\stoch$, and in fact through the functor $\d: \meas \to \stoch$. This is done by defining the functor $\varepsilon: \stoch \to \meas$ sending a measurable space $X$ to $\mathcal PX$ and a stochastic map $k:X \to Y$ to the measurable function $\varepsilon k: \mathcal PX \to \mathcal PY$ defined by
\ba
\varepsilon k: \mathcal PX &\longrightarrow \mathcal PY; \\
\mu & \longmapsto \left(B \in \S_Y \mapsto \textstyle\int_X k(x,B) \,d\mu\right).
\ea
We then have an adjunction $\d \dashv \varepsilon$
\[
\xymatrix{
\meas \ar@/^/[rr]^{\d} & & \stoch \ar@/^/[ll]^{\varepsilon}
}
\]
with composite $\gamma \circ \d = \mathcal P$.

Finally, note that if $X$ is finite or countably generated then $\mathcal PX$ is finite or countably countably generated respectively too, so we may also view $\finstoch$ and $\cgst$ as Kleisli categories of monads on $\mathbf{FinMeas}$ and $\mathbf{CGMeas}$ respectively.

\chapter{Bayesian Networks} \label{ch:bn}

In the first chapter we discussed a formalism for representing processes, while in the second we introduced a way to think of these processes as probabilistic. In this short third chapter we now add to this some language for describing selected probabilistic processes as causal. 

As in the case of probability, although the intuition for the concept is clear, any attempt to make precise what is meant by causality throws up a number of philosophical questions. We shall not delve into these here, but instead say that we will naively view a causal relationship as an asymmetric one between two variables, in which the varying of one---the cause---necessarily induces variations in the other---the effect. In particular, we think of a causal relationship as implying a physical, objective, mechanism through which this occurs.

The structure we have chosen, Bayesian networks, has roots in graphical models in statistics, and was first proposed as a language for causality by Pearl in \cite{P3}, with special interest in applications to machine learning and artificial intelligence. Since then Bayesian networks have played a significant role in discussions of causality from both a computational and a philosophical perspective. This chapter in particular relies on expositions by Pearl \cite{P} and Williamson \cite{Wi}.

\section{Conditionals and independence}

Much of the difficulty in the discussion of causality arises from the fact that causal relationships can never be directly observed. We instead must reconstruct such relationships from hints in independencies between random variables. The key point is that if $A$ causes $B$, then $A$ and $B$ cannot be independent.

\begin{defn}[Independence]
Let $(X,\S_X)$ and $(Y,\S_Y)$ be measurable spaces, and let $P$ be a joint probability measure on the product space $(X \times Y, \S_X \ot \S_Y)$. We say that $X$ and $Y$ are \emph{independent with respect to $P$} if $P$ is equal to the product measure of $P_X$ and $P_Y$, and \emph{dependent with respect to $P$} otherwise. 
\end{defn} 

Independent joint distributions can also be characterised as those that are of the form
\[
P(A\times B) = \int_A c(-,B)\,dP_X,
\]
for all $A \in \S_X$ and $B \in \S_Y$, where $c: X \to Y$ is a stochastic map that factors through the terminal object of $\stoch$. Indeed, to find such a $c$ corresponding to any independent joint distribution, we may just take the stochastic map defined by
\[
c(x,B) = P_Y(B)
\]
for all $x \in X$. This general idea gives a recipe for deconstructing, or \emph{factorising}, a joint probability measure $P$ into a marginal $P_X$ on one factor and a stochastic map $P_{Y|X}: X \to Y$ from that factor to the product of the others. We call this stochastic map a conditional for the joint measure.

\begin{defn}[Conditional]
Let $(X,\S_X)$ and $(Y,\S_Y)$ be measurable spaces, and let $P$ be a joint probability measure on the product space $(X \times Y, \S_X \ot \S_Y)$. Then we say that a stochastic map $P_{Y|X}: (X, \S_X) \to (Y, \S_Y)$ is a \emph{conditional for $P$ with respect to $X$} if for all $A \in \S_X$ and $B \in \S_Y$ we have
\[
P(A \times B) = \int_A P_{Y|X}(-,B) \, dP_X.
\]
Note that the above integral consequently defines a measure on $X \times Y$ equal to $P$.
\end{defn}

Considering the marginals as stochastic maps $1 \to (X,\S_X)$, this also implies that
\[
P_Y = P_{Y|X}\circ P_X.
\]
This says that $P_{Y|X}$ is a stochastic map from $X$ to $Y$ that maps the marginal $P_X$ to the marginal $P_Y$. As there are many joint measures with marginals $P_X$ and $P_Y$, however, this is not a sufficient condition for $P_{Y|X}$ to be the conditional for $P$ with respect to $X$.

While this is so, given a joint probability measure with marginals again probability measures, under mild constraints it is always true that there exists a conditional for it, and that this conditional is `almost' unique. This is made precise by the following proposition. Recall that given a measurable space $(X,\S)$, we call a measure $\mu$ on $X$ \emph{perfect} if for any measurable function $f: X \to \rr$ there exists a Borel measurable set $E \subseteq f(X)$ such that $\mu(f\i(E)) = \mu(X)$. This proposition represents another reason why we will occasionally restrict our attention to $\cgst$.

\begin{prop}[Existence of regular conditionals] \label{prop.disint}
Let $(X,\S_X)$ and $(Y,\S_Y)$ be countably generated measurable spaces, and let $\mu$ be a measure on the product space such that the marginal $\mu_X$ is perfect. Then there exists a stochastic map $k: (X,\S_X) \to (Y,\S_Y)$ such that for all $A \in \S_X$ and $B \in \S_Y$ we have
\[
\mu(A \times B) = \int_Ak(-,B) \, d\mu_X.
\]
Furthermore, this stochastic map is unique in the sense that if $k'$ is another stochastic map with these properties, then $k$ and $k'$ are equal $\mu_X$ almost everywhere.
\end{prop}
\begin{proof}
Existence is proved in Faden \cite[Theorem 6]{Fa}. Uniqueness is not difficult to show, but can be found in Vakhania, Tarieladze\cite[Proposition 3.2]{VT}.
\end{proof}

The existence of conditionals gives rise to a more general notion of independence, aptly named conditional independence.

\begin{defn}[Conditional independence]
Let $(X,\S_X)$, $(Y,\S_Y)$, $(Z,\S_Z)$ be measurable spaces, and let $P$ be a joint probability measure on the product space $X\times Y \times Z$. We say that \emph{$X$ and $Y$ are conditionally independent given $Z$ (with respect to $P$)} if 
\begin{enumerate}[(i)]
\item a conditional $P_{XY|Z}: (Z,\S_Z) \to (X\times Y,\S_X \ot \S_Y)$ exists; and 
\item for each $z \in Z$, $X$ and $Y$ are independent with respect to the probability measure $P_{XY|Z}(z,-)$.
\end{enumerate}
\end{defn}

This notion gives us far more resolution in investigations of how variables can depend on each other, and hence in finding causal relationships. For example, the variables representing the amount of rain on a given day in London and in Beijing are dependent---on a winter day it on average rains more than a summer one in both cities---, but we can tell they are not causally related because they are conditionally independent given the season. A key feature of Bayesian networks is that it allows us to translate facts about causal relationships into facts about conditional independence, and vice versa.

The following lemma helps with further conceptualising conditional independence. In particular, conditions (iii) and (iv) say that if $X$ and $Y$ are conditionally independent given $Z$, then upon knowing the outcome of $Z$, the outcome of $X$ gives no information about the outcome of $Y$, and the outcome of $Y$ gives no information about the outcome of $X$.

\begin{lem}[Countable conditional independence]
Let $X,Y,Z$ be countable discrete measurable spaces with a joint probability measure $P$ such that the marginals on $Z$, $XZ$, $YZ$ each have full support. The following are equivalent:
\begin{enumerate}[(i)]
\item $X$ and $Y$ are conditionally independent given $Z$.
\item $P_{X|Z}(z,\{x\})P_{Y|Z}(z,\{y\}) = P_{XY|Z}(z,\{(x,y)\})$ for all $x \in X$, $y \in Y$, $z \in Z$.
\item $P_{X|YZ}(y,z,\{x\}) = P_{X|Z}(z,\{x\})$ for all $x \in X$, $y \in Y$, $z \in Z$.
\item $P_{Y|XZ}(x,z,\{y\}) = P_{Y|Z}(z,\{y\})$ for all $x \in X$, $y \in Y$, $z \in Z$.
\end{enumerate}
\end{lem}
\begin{proof}
The equivalence of (i) and (ii) is just the definition of conditional independence, noting that in the discrete case the conditionals are uniquely determined by their values on individual outcomes. The equivalence of (ii), (iii), and (iv) follow from elementary facts in probability theory; a proof can be found in \cite{Lau}.
\end{proof}

\section{Bayesian networks}

We introduce Bayesian networks with an example. 

\begin{ex} \label{ex.bn}
Suppose that we wish to add a causal interpretation to a joint probability measure on the binary random variables $A=\{a, \neg a\}$, $B = \{b, \neg b \}$, and $C= \{c, \neg c\}$ representing the propositions that, upon being presented with a food:
\begin{itemize}
\item[A:] you like, or \emph{appreciate}, the food.
\item[B:] the food is nutritionally \emph{beneficial}.
\item[C:] you \emph{choose} to eat the food.
\end{itemize}
\noindent Let these random variables have joint probability measure given by the table
\begin{center}
\begin{tabular}{c|l}
$x$ & $P(x)$ \\ \hline
$a,b,c$ &Ê0.24 \\
$a,b,\neg c$ & 0\\
$a,\neg b,c$ & 0.18 \\
$a,\neg b, \neg c$ & 0.18 \\
$\neg a, b,c$ &Ê0.06 \\
$\neg a, b,\neg c$ & 0.10 \\
$\neg a,\neg b,c$ & 0 \\
$\neg a,\neg b, \neg c$ & 0.24
\end{tabular}
\end{center}
Intuitively, the causal relationships between our variables are obvious: liking a food influences whether you choose to eat it, and so does understanding it has health benefits, but otherwise there are no causal relationships between the variables---liking a food does not cause it to be more (or less) healthy. We shall represent these causal relationships by the directed graph
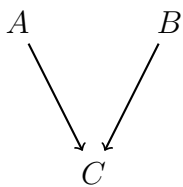
\begin{center}
\begin{tikzpicture}
  \node (A) at (-1,1) {$A$};
  \node (B) at (1,1)  {$B$};
  \node (C) at (0,-1)  {$C$};
  \foreach \from/\to in {A/C,B/C}
    \draw[->,thick] (\from) -- (\to);
\end{tikzpicture}
\end{center}
where we have drawn an arrow from one variable to another to indicate that that variable has causal influence on the other. The above joint probability measure and graph comprise what we will later define as a Bayesian network. 

Note that we could not have chosen just any directed graph with vertices $A$, $B$, and $C$, as assertions about causal relationships have consequences that must be reflected in the joint probability measure. For example, as in the above graph neither $A$ or $B$ cause of each other, nor have a common cause, we expect that $A$ and $B$ are independent with respect to the marginal $P_{AB}$. This is true. Writing probability measures as $n \times 1$ stochastic matrices with respect to the bases $\{x, \neg x\}$ for the binary variables and $\{ab,a\neg b, \neg a b, \neg a \neg b\}$ for the variable $AB$, we have
\[
P_A = \begin{pmatrix} 0.6 \\ 0.4 \end{pmatrix}; \qquad P_B = \begin{pmatrix} 0.4 \\ 0.6 \end{pmatrix}
\]
and hence
\[
P_A \ot P_B = P_{AB} = \begin{pmatrix} 0.24 \\ 0.36  \\ 0.16 \\ 0.24 \end{pmatrix}
\]
Furthermore, the above graph suggests that the probability measure on $C$ can be written as a function of the outcomes of both its causes $A$ and $B$. We thus expect that the measure has factorisation
\[
P(x_a,x_b,x_c) = \int_{(x_a,x_b)}\int_{x_c} dP_{C|AB} dP_{AB} = \int_{x_a}\int_{x_b}\int_{x_c} dP_{C|AB} dP_B dP_A,
\]
where $x_a \in A$, $x_b \in B$, $x_c \in C$. As these variables are finite, we might also write this requirement as
\[
P(x_a,x_b,x_c) = P_{C|AB}(x_a,x_b,\{x_c\}) P_B(\{x_b\}) P_A(\{x_a\}).
\]
Again, this is also true, with
\[
P_{C|AB} = \begin{pmatrix} 1 & 0.5 & 0.375 & 0 \\ 0 & 0.5 & 0.625 & 1 \end{pmatrix}.
\]
Motivated by this, we will later define a compatibility requirement in terms of the existence of a certain factorisation. Note that although in general a probability measure will have many factorisations; the directed graph specifies a factorisation that we attach greater---causal---significance to.

An advantage of expressing the causal relationships as a graph is that we may read from it other, acausal, dependencies. For example, while $A$ and $B$ are independent, the above graph suggests that if we know something about their common consequence $C$, this should induce some dependence between them. Indeed we find this does occur. Observe that 
\[
P_{AB|C}(c,-) = \begin{pmatrix} 0.5 \\ 0.375  \\ 0.125 \\ 0 \end{pmatrix},
\]
which indicates that of the foods you choose to eat, the foods you like are more likely to be unhealthy than those you dislike. 

More than this, however, marking certain relationships as causal affects our understanding of how a joint probability measure should be interpreted; we will see an example of this in the next chapter.
\end{ex}

To make these ideas precise we introduce some definitions. Recall that a \emph{directed graph} $G = (V,A,s,t)$ consists of a finite set $V$ of vertices, a finite set $A$ of arrows, and source and target maps $s,t:A \to V$ such that no two arrows have the same source and target---precisely, such that for all $a, a' \in A$ either $s(a) \ne s(a')$ or $t(a) \ne t(a')$. An arrow $a \in A$ is said to be an \emph{arrow from $u$ to $v$} if $s(a)=u$ and $t(a)=v$, while a sequence of vertices $v_1, \dots, v_k \in V$ is said to form a \emph{path from $v_1$ to $v_k$} if for all $i =1, \dots, k-1$ there exists an arrow $a_i \in A$ from $v_i$ to $v_{i+1}$. A path is also called a \emph{cycle} if in addition $v_1 = v_k$. A directed graph is \emph{acyclic} if it contains no cycles. 


As demonstrated in the above example, directed acyclic graphs provide a depiction of causal relationships between variables; the direction represents the asymmetry of the causal relationship, while cycles are disallowed as variables cannot have causal influence on themselves. When we think of the set of vertices of a directed acyclic graph as the set of random variables of a system, we will also call the graph a \emph{causal structure}.

Given a directed graph, we use the terminology of kinship to talk of the relationships between vertices, saying that a vertex $u$ is a \emph{parent} of a vertex $v$ if there is an arrow from $u$ to $v$, $u$ is an \emph{ancestor} of $v$ if there is a path from $u$ to $v$, $u$ is a \emph{child} of $v$ if there is an arrow from $v$ to $u$, and $u$ is a \emph{descendent} of $v$ if there is a path from $u$ to $v$. We will in particular talk of the parents of a vertex frequently, and so introduce the notation
\[
\pa(v) = \{u \in V \mid \textrm{there exists } a \in A \textrm{ such that } s(a) = u, t(a) = v\}
\]
for the set of parents of a vertex. When dealing with graphs as causal structures, we will also use the names \emph{direct causes}, \emph{causes}, \emph{direct effects}, and \emph{effects} to mean parents, ancestors, children, and descendants respectively.

We say that an ordering $\{v_1, \dots , v_n\}$ of the set $V$ is an \emph{ancestral ordering} if $v_i$ is an ancestor of $v_j$ only when $i < j$.

\begin{defn}[Bayesian network]
Let $G = (V,A,s,t)$ be a directed acyclic graph, for each $v \in V$ let $X_v$ be a measurable space, and let $P$ be a joint probability measure on $\prod_{v \in V} X_v$. We say that the causal structure $G$ and the joint probability measure $P$ are \emph{compatible} if there exists an ancestral ordering of the elements $V$ such that there exist conditionals such that
\[
P(A_1 \times A_2 \times \dots \times A_n) =  \int_{A_1} \int_{A_2} \dots \int_{A_n} P_{X_n|\pa(X_n)} \dots P_{X_2|\pa(X_2)} P_{X_1}.
\]
A \emph{Bayesian network} $(G,P)$ is a pair consisting of a compatible joint probability measure and causal structure.
\end{defn}

A better understanding of this compatibility requirement can be gained from the examining following theorem.

\begin{thm}[Equivalent compatibility conditions]
Let $G$ be a causal structure, let $\{(X_v,\S_v) \mid v \in V\}$ be a collection of finite measurable spaces indexed by the vertices of $G$, and let $P$ be a joint probability measure on their product space. Then the following are equivalent:
\begin{enumerate}[(i)]
\item $P$ is compatible with $G$.
\item $P(x_1, \dots, x_n) = \prod_{i=1}^n P_{X_i|\pa(X_i)}(\pa(x_j),\{x_i\})$, where $\pa(x_j)$ is the tuple consisting of $x_j$ such that $X_j \in \pa(X_i)$.
\item $P$ obeys the \emph{ordered Markov condition} with respect to $G$: given any ancestral ordering of the variables, each variable is independent of its remaining preceding variables conditional on its parents.
\item $P$ obeys the \emph{arental Markov condition} with respect to $G$: each variable is independent of its nondescendents conditional on its parents.
\end{enumerate}
\end{thm}
\begin{proof}
This follows from Corollaries 3 and 4 to Theorem 3.9 in Pearl \cite{P2}.
\end{proof}

As an illustration of the relevance of causal structure, we note that conditional independence relations between variables of a Bayesian network can be read from the causal structure using a straightforward criterion. Call a sequence of vertices $v_1, \dots, v_k \in V$ an \emph{undirected path from $v_1$ to $v_k$} if for all $i =1, \dots, k-1$ there exists an arrow $a_i \in A$ from $v_i$ to $v_{i+1}$ or from $v_{i+1}$ to $v_i$. An undirected paths $v_1,v_2,v_3$ of three vertices then take the form of a
\begin{enumerate}[(i)]
\item \emph{chain}: $v_1 \to v_2 \to v_3$ or $v_1 \leftarrow v_2 \leftarrow v_3$, 
\item \emph{fork}: $v_1 \leftarrow v_2 \to v_3$; or
\item \emph{collider}: $v_1 \to v_2 \leftarrow v_3$.
\end{enumerate} 
An undirected path from a vertex $u$ to a vertex $v$ is said to be \emph{d-separated} by a set of nodes $S$ if either the path contains a chain or a fork such that the centre vertex is in $S$, or if the path contains a collider such that the neither the centre vertex nor any of its descendants are in $S$. A set $S$ then \emph{d-separates} a set $U$ from a set $T$ if every path from a vertex in $U$ to a vertex in $T$ is \emph{d}-separated by $S$.

The main idea of this definition is that causal influence possibly creates dependence between two random variables if one variable is a cause of the other, the variables have a common cause, or the variables have a common consequence and the outcome of this consequence is known. In this last case, knowledge of the consequence `unseparates' the two variables along the path through the known common consequence. On the other hand, any information gained through having a common cause is rendered moot if we have knowledge about a variable through which the causal influence is mediated. These ideas are captured by the following theorem.

\begin{thm}
Let $(G,P)$ be a Bayesian network. If sets $U$ and $T$ of vertices of $G$ are \emph{d}-separated by a third set $S$, then with respect to $P$, the product random variables $\prod_{u \in U} X_u$ and $\prod_{t \in T} X_t$ are conditionally independent given $\prod_{s \in S} X_s$.
\end{thm}
\begin{proof}
See Verma and Pearl \cite{VP}.
\end{proof}

An example of this is the way the variables $A$ and $B$ are independent in Example \ref{ex.bn}, but dependent conditional on $C$.

\chapter{Causal Theories}

We now tie the elements of the last three chapters together to propose and develop a novel algebraic structure: a causal theory. After introducing these structures, we discuss their models in various categories and how such models might be interpreted, and then look at a possibly confusing situation that causal theories and their associated graphical language help make lucid.

\section{The category associated to a causal structure}

We wish to fashion a category that captures methods of reasoning with causal relationships. In this category, we will want our objects to represent the variables of a situation, while the morphisms should represent the ways one can deduce knowledge about one variable from another. Furthermore, as we will want to deal with more than one variable at a time, and the outcomes their joint variable may take, this category will be monoidal. 

As we may only reason about causal relationships once we have some causal relationships to reason with, we start by fixing a set of symbols for our variables and the causal relationships between them. Let $G = (V,A,s,t)$ be a directed acyclic graph. From this we construct a strict symmetric monoidal category $\c_G$ in the following way. 

For the objects of $\c_G$ we take the set $\mathbb N^V$ of functions from $V$ to the natural numbers. These may be considered collections of elements of the set of variables $V$, allowing multiplicities, and we shall often just write these as strings $w$ of elements of $V$. Here the order of the symbols in the string is irrelevant, and we write $\varnothing$ for empty string, which corresponds to the zero map of $\mathbb N^V$. We view these objects as the \emph{variables} of the causal theory, and we further call the objects which are collections consisting of just one instance of a single element of $V$ the \emph{atomic variables} of the causal theory. 

There are two distinct classes of generating morphisms for $\c_G$. The first class is the collection of \emph{comonoid maps}: for each atomic variable $v \in V$, we include morphisms $\comult_v: v \to vv$ and $\counit_v: v \to \varnothing$. These represent the ideas of duplicating some information about $v$, or forgetting some. The maps of the second class are called the \emph{causal mechanisms}. These consist of, for each atomic variable $v \in V$, a morphism $[v|pa(v)]: pa(v) \to v$, where $pa(v)$ is the string consisting of the parents of $v$ in any order, and represent the ways we may use information about a collection of variables to infer facts about another. We then use these morphisms as generators for a strict symmetric monoidal category, taking all products and well-defined compositions, subject only to the constraint that for each $v \in V$ the pair $(\comult_v, \counit_v)$ forms a comonoid. As the swaps are identity maps, these comonoids are immediately commutative.

We call this category $\c_G$ the \emph{causal theory} of the causal structure $G$. Morphisms of $\c_G$ represent ways to reason about the outcome of the codomain variable given some knowledge about the domain variable.

As the causal mechanisms are labelled with their domain and codomain, there is usually no need to label the strings when representing morphisms of $\c$ with string diagrams. We also often do not differentiate between the comonoid maps with labels, as the context makes which comonoid map we are applying. The order in which we write the string representing the set $\pa(v)$ corresponds to the order of the input strings. 

\begin{ex}
The causal theory of the causal structure 
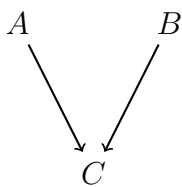
\begin{center}
\begin{tikzpicture}
  \node (A) at (-1,1) {$A$};
  \node (B) at (1,1)  {$B$};
  \node (C) at (0,-1)  {$C$};
  \foreach \from/\to in {A/C,B/C}
    \draw[->,thick] (\from) -- (\to);
\end{tikzpicture}
\end{center}
of Example \ref{ex.bn} is the symmetric monoidal category with objects collections of the letters $A$, $B$, and $C$, and morphisms generated by counit and comultiplication maps on each of $A$, $B$, and $C$, as well as causal mechanisms $[A]: \varnothing \to A$, $[B]: \varnothing \to B$, and $[C|AB]: AB \to C$. We depict these causal mechanisms respectively as
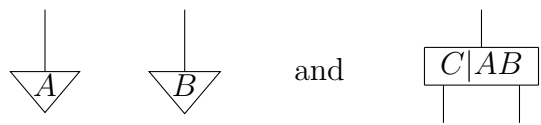
\[
\begin{aligned}
\begin{tikzpicture}
	\begin{pgfonlayer}{nodelayer}
		\node [style=tri] (0) at (0, -0.25) {$A$};
		\node [style=none] (1) at (0, 0.75) {};
		\node [style=none] (2) at (0, -0.75) {};
	\end{pgfonlayer}
	\begin{pgfonlayer}{edgelayer}
		\draw (1.center) to (0);
	\end{pgfonlayer}
\end{tikzpicture}
\end{aligned}
\qquad
\begin{aligned}
\begin{tikzpicture}
	\begin{pgfonlayer}{nodelayer}
		\node [style=tri] (0) at (0, -0.25) {$B$};
		\node [style=none] (1) at (0, 0.75) {};
		\node [style=none] (2) at (0, -0.75) {};
	\end{pgfonlayer}
	\begin{pgfonlayer}{edgelayer}
		\draw (1.center) to (0);
	\end{pgfonlayer}
\end{tikzpicture}
\end{aligned}
\qquad
\mbox{ and } \qquad
\begin{aligned}
\begin{tikzpicture}
	\begin{pgfonlayer}{nodelayer}
		\node [style=none] (0) at (0, 0.75) {};
		\node [style=none] (1) at (-0.5, -0.75) {};
		\node [style=none] (2) at (0.5, -0.75) {};
		\node [style=none] (3) at (0, 0.25) {};
		\node [style=none] (4) at (0.5, -0.25) {};
		\node [style=none] (5) at (-0.5, -0.25) {};
		\node [style=none] (6) at (-0.75, 0.25) {};
		\node [style=none] (7) at (0.75, 0.25) {};
		\node [style=none] (8) at (0.75, -0.25) {};
		\node [style=none] (9) at (-0.75, -0.25) {};
		\node [style=none] (10) at (0, 0) {$C|AB$};
	\end{pgfonlayer}
	\begin{pgfonlayer}{edgelayer}
		\draw (4.center) to (2.center);
		\draw (0.center) to (3.center);
		\draw (5.center) to (1.center);
		\draw (8.center) to (9.center);
		\draw (9.center) to (6.center);
		\draw (6.center) to (7.center);
		\draw (7.center) to (8.center);
	\end{pgfonlayer}
\end{tikzpicture}
\end{aligned}
\]
\end{ex}

We now list a few facts to give a basic understanding of the morphisms in these categories. These morphisms represent predictions of the consequences of the domain variable on the codomain variable. As causal structures are acyclic---giving rise to a `causal direction', or (noncanonical) ordering on the variables---, causal theories similarly have such a direction, and this puts limits on the structure. Indeed, one consequence is that a morphism can only go from an effect to a cause if it factors through the monoidal unit; this represents `forgetting' the outcome of the effect, and reasoning about the outcomes of the cause from other, background, information. We say a map is \emph{inferential} if it does not factor through the monoidal unit.

\begin{prop}
Let $\c_G$ be a causal theory, and let $v,v'$ be atomic variables in $\c_G$. If there exists an inferential map $v \to v'$, then $v$ is an ancestor of $v'$ in $G$.
\end{prop}
\begin{proof}
We reason via string diagrams to prove the contrapositive. 

Observe that a generating map is inferential if and only if, in its string diagram representation, the domain is topologically connected to the codomain, and that this property is preserved by the counitality relation the comonoid maps must obey. Thus it is also true in general: a morphism in $\c_G$ is inferential if and only if, in all string diagram representations, the domain is topologically connected to the codomain.


Note also that for all generating maps with string diagrams in which the domain and codomain are connected, the domain and codomain are nonempty and each element of the domain is either equal to or an ancestor of each element of the codomain. This property is also preserved by the counitality relation. Thus, if $v$ is not an ancestor of $v'$, in all string diagram representations of a map $v \to v'$ the domain is not topologically connected to the codomain. Taking any such string diagram and continuously deforming it by moving all parts of the component connected to the domain below all parts of the component connected to the codomain, we thus see that the map may be rewritten as one that factors through the monoidal unit.
\end{proof}

In fact the converse also holds: if $v$ is an ancestor of $v'$, then there always exists an inferential map $v \to v'$. Indeed, if $w,w'$ are objects in $\c_G$ containing each atomic variable no more than once---that is, when $w,w' \subseteq V$---, we can construct a map $w \to w'$ in the following manner.

\begin{enumerate}
\item Take the smallest subgraph $G_{w \to w'}$ of $G$ containing the vertices of $w$ and $w'$, and all paths in $V$ that terminate at an element of $w'$ and do not pass through $w$. \footnote{Precisely, this means we take the subgraph with set of arrows $A_{w \to w'} \subseteq A$ consisting of all $a \in A$ for which there exist $a_1, \dots,a_n$ with 
\begin{enumerate}[(i)]
\item $s(a_1) = t(a)$,
\item $s(a_{i+1}) = t(a_{i})$ for $i = 1, \dots, n-1$,
\item $t(a_n) \in w'$, 
\item $s(a_i) \notin w$ for $i = 1, \dots, n$,
\end{enumerate}
and vertices 
\[
V_{w \to w'} = w \cup w' \cup \{v \in V \mid v = s(a) \mbox{ or } t(a) \mbox{ for some } a \in A_{w \to w'}\}.
\]
Note that for each $v \in V_{w \to w'} \setminus w$, the set of parents of $v$ in this subgraph is equal to the set of parents of $v$ in the whole graph $G$.
}

\item For each vertex $v \in V_{w \to w'}$ let $k_v$ be the number of arrows of $A_{w \to w'}$ with source $v$. Then:
\begin{enumerate}[(i)]
\item for each $v \in w$ take the string diagram representing the composition of $k_v-1$ comultiplications on $v$ or, when $k_v =0$, the counit of $v$.

\item For each $v \in w'$ take the string diagram for $[v|pa(v)]$ composed with a sequence of $k_v$ comultiplications on $v$.

\item For each $v \in V_{w \to w'} \setminus (w \cup w')$ take the string diagram for $[v|pa(v)]$, composed with either a sequence of $k_v-1$ comultiplications on $v$ or, if $k_v=0$, composed with the counit on $v$.
\end{enumerate}

\item From this collection of string diagrams, create a single diagram by connecting an output of one string diagram to an input of another if the set $A_{w \to w'}$ contains an arrow from the indexing vertex of the first diagram to the indexing vertex of the second.
\end{enumerate}

Due to the symmetry of the monoidal category $\c_G$ and the associativity of the comultiplication maps, this process uniquely defines a string diagram representing a morphism from $w$ to $w'$. Moreover, this map is inferential whenever there exists some $v \in w$ that is an ancestor of some $v' \in w'$. These maps are in a certain sense the uniquely most efficient ways of predicting probabilities on $w'$ using information about $w$, and will play a special role in what follows. For short, we will call these maps \emph{causal conditionals} and write them maps $[w'||w]$, or simply $[w']$ when $w = \varnothing$. In this last case, we will also call the map $[w']$ the \emph{prior on $[w']$}.

\begin{ex}
This construction is a little abstruse on reading, but the main idea is simple and an example should make it much clearer. Let $G$ be the causal structure
\begin{center}
\begin{tikzpicture}
  \node (A) at (0.5,2) {$A$};
  \node  (B) at (-1,1)  {$B$};
  \node (C) at (0,0)  {$C$};
  \node (D) at (1,-1)  {$D$};
  \node (E) at (-0.5,-2)  {$E$};
  \node (F) at (1,-3)  {$F$};  
  \foreach \from/\to in {A/B,A/C,B/C,C/D,C/E,D/E,E/F}
    \draw[->,thick] (\from) -- (\to);
\end{tikzpicture}
\end{center}
and suppose that we wish to compute the causal conditional $[DE||B]$. Step 1 gives the subgraph $G_{B \to DE}$
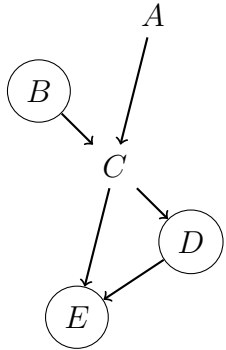
\begin{center}
\begin{tikzpicture}
  \node (A) at (0.5,2) {$A$};
  \node [style=circ] (B) at (-1,1)  {$B$};
  \node (C) at (0,0)  {$C$};
  \node [style=circ] (D) at (1,-1)  {$D$};
  \node [style=circ] (E) at (-0.5,-2)  {$E$};
  \foreach \from/\to in {A/C,B/C,C/D,C/E,D/E}
    \draw[->,thick] (\from) -- (\to);
\end{tikzpicture}
\end{center}
consisting of all paths to $D$ or $E$ not passing through $B$. 

Step 2 then states that the causal conditional $[DE||B]$ comprises the maps
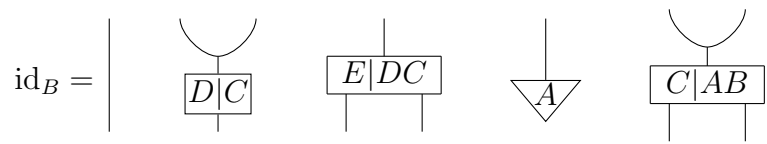
\[
\id_B = 
\begin{aligned}
\begin{tikzpicture}
	\begin{pgfonlayer}{nodelayer}
		\node [style=none] (1) at (0, 0.75) {};
		\node [style=none] (2) at (0, -0.75) {};
	\end{pgfonlayer}
	\begin{pgfonlayer}{edgelayer}
		\draw (1.center) to (2.center);
	\end{pgfonlayer}
\end{tikzpicture}
\end{aligned}
\qquad
\begin{aligned}
\begin{tikzpicture}
	\begin{pgfonlayer}{nodelayer}
		\node [style=none] (0) at (0, -0.75) {};
		\node [style=none] (2) at (0, 0.25) {};
		\node [style=sq] (3) at (0, -0.25) {$D|C$};
		\node [style=none] (4) at (0.5, 0.75) {};
		\node [style=none] (6) at (-0.5, 0.75) {};
	\end{pgfonlayer}
	\begin{pgfonlayer}{edgelayer}
		\draw [in=0, out=-90, looseness=0.75] (4.center) to (2.center);
		\draw [in=180, out=-90, looseness=0.75] (6.center) to (2.center);
		\draw (2.center) to (3);
		\draw (3) to (0.center);
	\end{pgfonlayer}
\end{tikzpicture}
\end{aligned}
\qquad
\begin{aligned}
\begin{tikzpicture}
	\begin{pgfonlayer}{nodelayer}
		\node [style=none] (0) at (0, 1.25) {};
		\node [style=none] (1) at (-0.5, -0.25) {};
		\node [style=none] (2) at (0.5, -0.25) {};
		\node [style=none] (3) at (0, 0.75) {};
		\node [style=none] (4) at (0.5, 0.25) {};
		\node [style=none] (5) at (-0.5, 0.25) {};
		\node [style=none] (6) at (-0.5, -0.5) {};
		\node [style=none] (7) at (0.5, -0.5) {};
		\node [style=none] (8) at (0, 1.5) {};
		\node [style=none] (9) at (-0.75, 0.75) {};
		\node [style=none] (10) at (0.75, 0.75) {};
		\node [style=none] (11) at (0.75, 0.25) {};
		\node [style=none] (12) at (-0.75, 0.25) {};
		\node [style=none] (13) at (0, 0.5) {$E|DC$};
	\end{pgfonlayer}
	\begin{pgfonlayer}{edgelayer}
		\draw (4.center) to (2.center);
		\draw (0.center) to (3.center);
		\draw (5.center) to (1.center);
		\draw (11.center) to (12.center);
		\draw (12.center) to (9.center);
		\draw (9.center) to (10.center);
		\draw (10.center) to (11.center);
	\end{pgfonlayer}
\end{tikzpicture}
\end{aligned}
\qquad
\begin{aligned}
\begin{tikzpicture}
	\begin{pgfonlayer}{nodelayer}
		\node [style=tri] (0) at (0, -0.25) {$A$};
		\node [style=none] (1) at (0, -1) {};
		\node [style=none] (2) at (0, 0.75) {};
		\node [style=none] (3) at (0, 1) {};
	\end{pgfonlayer}
	\begin{pgfonlayer}{edgelayer}
		\draw (2.center) to (0);
	\end{pgfonlayer}
\end{tikzpicture}
\end{aligned}
\qquad
\begin{aligned}
\begin{tikzpicture}
	\begin{pgfonlayer}{nodelayer}
		\node [style=none] (0) at (-0.5, -0.75) {};
		\node [style=none] (1) at (0.5, -0.75) {};
		\node [style=none] (2) at (0, 0.25) {};
		\node [style=none] (3) at (0.5, -0.25) {};
		\node [style=none] (4) at (-0.5, -0.25) {};
		\node [style=none] (5) at (-0.75, 0.25) {};
		\node [style=none] (6) at (0.75, 0.25) {};
		\node [style=none] (7) at (0.75, -0.25) {};
		\node [style=none] (8) at (-0.75, -0.25) {};
		\node [style=none] (9) at (0, 0) {$C|AB$};
		\node [style=none] (10) at (0, 0.5) {};
		\node [style=none] (11) at (0.5, 1) {};
		\node [style=none] (12) at (-0.5, 1) {};
		\node [style=none] (13) at (-0.5, 1) {};
	\end{pgfonlayer}
	\begin{pgfonlayer}{edgelayer}
		\draw (3.center) to (1.center);
		\draw (4.center) to (0.center);
		\draw (7.center) to (8.center);
		\draw (8.center) to (5.center);
		\draw (5.center) to (6.center);
		\draw (6.center) to (7.center);
		\draw [in=0, out=-90, looseness=0.75] (11.center) to (10.center);
		\draw [in=180, out=-90, looseness=0.75] (13.center) to (10.center);
		\draw (10.center) to (2.center);
	\end{pgfonlayer}
\end{tikzpicture}
\end{aligned}
\]
and we then compose these mimicking the topology of the graph $G_{B \to DE}$ to give the map
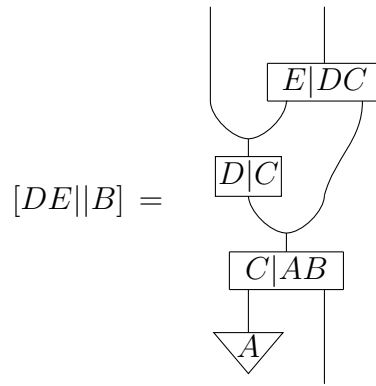
\[
[DE||B]\, = \quad
\begin{aligned}
\begin{tikzpicture}
	\begin{pgfonlayer}{nodelayer}
		\node [style=none] (0) at (0.5, -3) {};
		\node [style=tri] (1) at (-0.5, -2.5) {$A$};
		\node [style=none] (2) at (0, -1.25) {};
		\node [style=none] (3) at (-0.5, -1.75) {};
		\node [style=none] (4) at (0.5, -1.75) {};
		\node [style=none] (5) at (-0.75, -1.25) {};
		\node [style=none] (6) at (0.75, -1.25) {};
		\node [style=none] (7) at (0.75, -1.75) {};
		\node [style=none] (8) at (-0.75, -1.75) {};
		\node [style=none] (9) at (0, -1.5) {$C|AB$};
		\node [style=none] (10) at (0, -1) {};
		\node [style=none] (11) at (0.5, -0.5) {};
		\node [style=none] (12) at (-0.5, -0.5) {};
		\node [style=none] (13) at (-0.5, -0.5) {};
		\node [style=none] (14) at (-1, 0.75) {};
		\node [style=none] (15) at (0, 0.75) {};
		\node [style=none] (16) at (-0.5, 0.25) {};
		\node [style=none] (17) at (-1, 0.75) {};
		\node [style=sq] (18) at (-0.5, -0.25) {$D|C$};
		\node [style=none] (19) at (1, 0.75) {};
		\node [style=none] (20) at (-0.25, 1.25) {};
		\node [style=none] (21) at (0.5, 1) {$E|DC$};
		\node [style=none] (22) at (1.25, 1.25) {};
		\node [style=none] (23) at (0, 0.75) {};
		\node [style=none] (24) at (0.5, 2) {};
		\node [style=none] (25) at (1.25, 0.75) {};
		\node [style=none] (26) at (-0.25, 0.75) {};
		\node [style=none] (27) at (0.5, 1.25) {};
		\node [style=none] (28) at (-1, 2) {};
	\end{pgfonlayer}
	\begin{pgfonlayer}{edgelayer}
		\draw (3.center) to (1);
		\draw (4.center) to (0.center);
		\draw (7.center) to (8.center);
		\draw (8.center) to (5.center);
		\draw (5.center) to (6.center);
		\draw (6.center) to (7.center);
		\draw [in=0, out=-90, looseness=0.75] (11.center) to (10.center);
		\draw [in=180, out=-90, looseness=0.75] (13.center) to (10.center);
		\draw (10.center) to (2.center);
		\draw [in=0, out=-90, looseness=0.75] (15.center) to (16.center);
		\draw [in=180, out=-90, looseness=0.75] (17.center) to (16.center);
		\draw (16.center) to (18);
		\draw (18) to (13.center);
		\draw (25.center) to (26.center);
		\draw (26.center) to (20.center);
		\draw (20.center) to (22.center);
		\draw (22.center) to (25.center);
		\draw (24.center) to (27.center);
		\draw [in=75, out=-90] (19.center) to (11.center);
		\draw (23.center) to (15.center);
		\draw (28.center) to (17.center);
	\end{pgfonlayer}
\end{tikzpicture}
\end{aligned}
\] 
\end{ex}

\section{Interpretations of causal theories}

Causal theories express abstractly avenues of causal reasoning, but this serves no purpose in describing specific causal relationships until we attach meanings, or an interpretation, to the objects and morphisms of the theory. The strength of separating out the syntax of reasoning is that these interpretations may now come from any symmetric monoidal category. Stated formally, let $\c$ be a causal theory, and let $\mathcal D$ be any symmetric monoidal category. Then a \emph{model of $\c$ in $\mathcal D$}, or just a \emph{causal model}, is a strong monoidal functor $M: \c \to \mathcal D$.

We explore the basic properties of causal models in a few categories. To demonstrate the basic ideas, we first take a brief look at models in $\set$ and $\rel$; models in these categories will be useful for describing deterministic and possibilistic causal relationships respectively. While $\meas$ is another obvious candidate setting for examining causal models, we merely note that causal models here behave somewhat similarly to $\set$ and move on to $\stoch$, the main category of interest. Here causal models generalise Bayesian networks. As Bayesian networks are known to provide a useful tool for the discussion of causality, this lends support to the idea that the richer structure of causal models in $\stoch$ do too.

\subsection*{Models in $\set$}

Due to its familiarity, we begin our discussion of causal models with an examination of the forms they take in $\set$. In both $\set$ and $\rel$ the objects are sets. For the purposes of causal models, it is useful to view these sets as variables, with the elements the possible outcomes of the variable. With this interpretation, we can understand $\set$ as a subcategory of $\meas$ in which every measurable space is discrete, making it possible to measure any subset of the outcomes of each variable. Morphisms in $\set$---that is, set functions---then assign a single outcome of the codomain variable to each outcome of the domain variable, and so can be said to describe deterministic causal relationships.

Given a causal theory $\c$, a model of $\c$ in $\set$ by definition consists of a strong monoidal functor $M: \c \to \set$. To specify such a functor up to isomorphism, it is enough to specify the image of each atomic variable and each generating map, subject to the constraints that the generating maps chosen are well-typed with respect to the chosen images of the atomic variables, and that the images of the comultiplication and counit obey the laws of a commutative comonoid. Indeed, once these are specified, the values of the functor on the remaining objects and morphisms of $\c$ are, up to isomorphism, determined by the definition of a strong monoidal functor. Note also that as long as the aforementioned constraints are fulfilled we have a well-defined strong monoidal functor.

We first observe, as we will also in the case of $\stoch$, that each object of $\set$ has a unique comonoid structure, and this comonoid is commutative. To wit, for each set $X$, there is a unique map $X \to \{\ast\}$, taking the product of this map and the identity map $X \to X$ gives the projection maps $X \times X \to X$, and the only function $X \to X \times X$ that composes to the identity with the projection map on each factor is the diagonal map $x \mapsto (x,x)$. Moreover, choosing the diagonal map as a comultiplication indeed gives a commutative comonoid with this map. It is a consequence of this that we need not worry about the comonoid maps; choosing a set for each variable also chooses the comonoid maps for us. 

On the other hand, as the causal mechanisms need not obey any equations, so having defined a map on the objects, any choice of functions from the product set of all the direct causes of each variable to the variable itself then gives a model of the causal theory in $\set$. Each such function returns the outcome of its codomain variable given a configuration of the outcomes of its causes. In this sense a model of a causal theory in $\set$ specifies how causes affect their consequences in a deterministic way.

As maps from the monoidal unit in $\set$ are just a pointings of the target set, the priors $M[w]$ of a causal model are just a choice of an outcome for each of the atomic variables in $w$. In the case of an atomic variable with no causes, the prior $M[v]: \{\ast\} \to Mv$ is simply the causal mechanism $M[v|pa(v)]$, and just picks an element of the set $Mv$. One might interpret this as the `default' state of the variable $v$, and subsequently interpret the prior $M[V]$ on the set of all variables $V$ as the default state of all variables in the system.

We shall see this as a general feature of models of causal theories; the priors specify what can be known about the system in some default state, while more generally the morphisms describe the causal relationships between variables even when not in this state.

\subsection*{Models in $\rel$}

In the category $\rel$, we interpret a relation $r: X \to Y$ to mean that that if $X$ has the outcome $x \in X$, then $Y$ may only take the outcomes $y$ related to $x$ via $r$. This is a possibilistic notion of causality, in which the outcomes of the causes do not determine a single outcome of the effect variable as in $\set$, but only put some constraint on the possible outcomes of the effect variable. 

A curious property of the category $\rel$ is that any relation $X \to Y$ may also be viewed as a relation $Y \to X$ in a natural way---that is, $\rel$ is equipped with a contravariant endofunctor that squares to the identity, or a dagger functor. This means that we have a way of reversing the direction any morphism we choose, and in this sense $\rel$ itself is acausal. This makes causal models all the more useful when working in $\rel$, as they provide a way of privileging certain relations with a causal direction. 

For any object in $\rel$, we may view the functions forming the unique comonoid on this set in $\set$ as relations, and hence have a commutative comonoid in $\rel$. An interesting collection of causal models $M: \c \to \rel$ in $\rel$ are those in which all objects are given these comonoid structures. Note that a map $s: \{\ast\} \to X$ from the monoidal unit to an object $X$ in $\rel$ is simply a subset $S$ of $X$ and, assuming $X$ has the comonoid structure in which the comonoid maps are functions, for any relation $r: X \to Y$ the composite
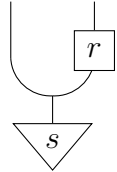
\[
\begin{tikzpicture}
	\begin{pgfonlayer}{nodelayer}
		\node [style=sq] (0) at (0.55, 0.6) {$r$};
		\node [style=none] (1) at (0, 0) {};
		\node [style=none] (2) at (0.55, 0.5) {};
		\node [style=none] (3) at (-0.55, 1.25) {};
		\node [style=tri] (4) at (0, -0.55) {\phantom{'}$s$\phantom{'}};
		\node [style=none] (5) at (-0.55, 0.5) {};
		\node [style=none] (6) at (0.55, 1.25) {};
	\end{pgfonlayer}
	\begin{pgfonlayer}{edgelayer}
		\draw (1.center) to (4);
		\draw [bend left=45] (1.center) to (5.center);
		\draw [bend right=45] (1.center) to (2.center);
		\draw (0) to (2.center);
		\draw (3.center) to (5.center);
		\draw (6.center) to (0);
	\end{pgfonlayer}
\end{tikzpicture}
\]
is then equal to the set
\[
\{(x,y) \in X \times Y \mid x \in S, y\sim_r x\}.
\]
Thus when the comonoid maps of the model are those of $\rel$ inherited from $\set$, it is easy to see that the priors $[w]$ in $\rel$ are given by the subset of the product set of the atomic variables in $w$ consisting of all joint outcomes that are possible given the constraints of the causal mechanisms.

In this setting, however, commutative comonoids are more general; for example, any collection of abelian groups forms a commutative comonoid on the union of the sets of elements of these groups \cite{Pav}. We leave examination of causal structures for these other comonoid structures, and their interpretations, for later work.

\subsection*{Models in $\stoch$}

We begin our discussion of causal models in $\stoch$ by showing that for $\stoch$ too we need not worry about selecting comonoid maps; the deterministic comonoid structure $\stoch$ inherits from $\set$ is the only comonoid structure on each object.

\begin{lem}
Each object of $\stoch$ has a unique comonoid structure. Moreover, this comonoid structure is commutative.
\end{lem}
\begin{proof}
Fix an object $(X, \S)$ in $\stoch$. We first show the existence of a comonoid structure on $X$ by showing that the stochastic maps
\[
\comult: (X, \S) \longrightarrow (X \times X, \S \ot \S);
\]
defined by
\ba
\comult: X \times (\S \ot \S) &\longrightarrow [0,1]; \\
(x,A) &\longmapsto \begin{cases} 1 & \textrm{if } (x,x) \in A, \\ 0 & \textrm{if } (x,x) \notin A, \end{cases}
\ea
and
\[
\counit: (X,\S) \longrightarrow (\ast, \{\varnothing, \ast\})
\]
defined by
\ba
\counit: X \times \{\varnothing, \ast\} &\longrightarrow [0,1]; \\
(x,\ast) &\longmapsto 1, \\
(x,\varnothing) &\longmapsto 0,
\ea
form the comultiplication and counit for a comonoid structure respectively. 

Indeed, observe that both these maps are deterministic stochastic maps, with $\comult$ specified by the measurable function $X \to X \times X; x \mapsto (x,x)$ and $\counit\,$ specified by $X \to \ast; x \mapsto \ast$. From here it is straightforward to verify that these functions obey coassociativity and counitality as functions in $\meas$, and hence these identities are true in $\stoch$.

We next prove uniqueness. As the monoidal unit $(\ast,\{\varnothing, \ast\})$ is terminal, there is a unique stochastic map from $(X,\S)$ to the terminal object, and so $\counit\,$ is the only possible choice for the counit of a comonoid on $(X,\S)$. Suppose that $\d: (X,\S) \to (X \times X, \S \ot \S)$ is a stochastic map such that $(\d, \counit)$ forms a comonoid. We will show in fact that for all $x \in X$ and $A \in \S \ot \S$ we have
\[
\d(x,A) =  \begin{cases} 1 & \textrm{if } (x,x) \in A; \\ 0 & \textrm{if } (x,x) \notin A, \end{cases}
\]
and so $\d = \comult$. Note that as $\d_x = \d(x,-)$ is a probability measure on $(X \times X, \S \ot \S)$ it is enough to show that $\d(x,A) =1$ whenever $(x,x) \in A$.

To begin, note that counitality on the right implies that
\[
\begin{aligned}
\begin{tikzpicture}
	\begin{pgfonlayer}{nodelayer}
		\node [style=none] (0) at (-0.5, 1.5) {};
		\node [style=dot] (1) at (0.5, 1.25) {};
		\node [style=none] (2) at (0, -0.25) {};
		\node [style=none] (3) at (-0.5, 0.75) {};
		\node [style=none] (4) at (0.5, 0.75) {};
		\node [style=none] (5) at (0, 0.25) {};
		\node [style=none] (6) at (-0.75, 0.75) {};
		\node [style=none] (7) at (0.75, 0.75) {};
		\node [style=none] (8) at (0.75, 0.25) {};
		\node [style=none] (9) at (-0.75, 0.25) {};
		\node [style=none] (10) at (0, 0.5) {$\d$};
	\end{pgfonlayer}
	\begin{pgfonlayer}{edgelayer}
		\draw (0.center) to (3.center);
		\draw (5.center) to (2.center);
		\draw (1) to (4.center);
		\draw (8.center) to (9.center);
		\draw (9.center) to (6.center);
		\draw (6.center) to (7.center);
		\draw (7.center) to (8.center);
	\end{pgfonlayer}
\end{tikzpicture}
\end{aligned}
\quad 
= 
\quad 
\begin{aligned}
\begin{tikzpicture}
	\begin{pgfonlayer}{nodelayer}
		\node [style=none] (0) at (0, .75) {};
		\node [style=none] (1) at (0, -1) {};
	\end{pgfonlayer}
	\begin{pgfonlayer}{edgelayer}
		\draw (0.center) to (1.center);
	\end{pgfonlayer}
\end{tikzpicture}
\end{aligned}
\]
so for all $x \in X$ and $B \in \S$ we have
\ba
\d_x(B \times X)  = \d(x,B \times X) &= \int_{X\times X} \chi_{B \times X} \, d\d_x \\
&= \int_{X\times X} \idline \,\counit (-,-,B) \,d\d_x \\
&= \begin{aligned}
\begin{tikzpicture}[scale=.55]
	\begin{pgfonlayer}{nodelayer}
		\node [style=none] (0) at (-0.5, 1.5) {};
		\node [style=dot] (1) at (0.5, 1.25) {};
		\node [style=none] (2) at (0, -0.25) {};
		\node [style=none] (3) at (-0.5, 0.75) {};
		\node [style=none] (4) at (0.5, 0.75) {};
		\node [style=none] (5) at (0, 0.25) {};
		\node [style=none] (6) at (-0.75, 0.75) {};
		\node [style=none] (7) at (0.75, 0.75) {};
		\node [style=none] (8) at (0.75, 0.25) {};
		\node [style=none] (9) at (-0.75, 0.25) {};
		\node [style=none] (10) at (0, 0.5) {$_\d$};
	\end{pgfonlayer}
	\begin{pgfonlayer}{edgelayer}
		\draw (0.center) to (3.center);
		\draw (5.center) to (2.center);
		\draw (1) to (4.center);
		\draw (8.center) to (9.center);
		\draw (9.center) to (6.center);
		\draw (6.center) to (7.center);
		\draw (7.center) to (8.center);
	\end{pgfonlayer}
\end{tikzpicture}
\end{aligned}(x,B) \\
&= \id_{(X,\S)} (x,B) \\
&= \begin{cases} 1 & \textrm{if } x \in B; \\ 0 & \textrm{if } x \notin B. \end{cases}
\ea
Similarly, counitality on the left implies that for all $x \in X$ and $B \in \S$ we have
\[
\d_x(X \times B) =  \begin{cases} 1 & \textrm{if } x \in B; \\ 0 & \textrm{if } x \notin B. \end{cases}
\]
We shall use these facts in the following.

Fix $x \in X$ and let now $A \in \S \ot \S$ be such that $(x,x) \in A$. Recalling our characterisation of product $\s$-algebras in Example \ref{ex:prodsga}, we may assume $A$ is of the form $A = \cup_{i \in I} (C_i \times D_i)$, where $I$ is a countable set and $C_i, D_i \in \S$ for all $i \in I$. There thus exists $0 \in I$ such that $(x,x) \in C_0 \times D_0$, and hence $x \in C_0$ and $x \in D_0$. Since we have shown above that this implies that $\d_x(C_0^c\times X) = \d_x(X \times D_0^c) = 0$, we then have
\ba
\d_x(A) &= \d_x(\cup_{i\in I} (C_i \times D_i)) \\
&\ge \d_x(C_0 \times D_0)  \\
&= \d_x((C_0 \times X) \cap (X \times D_0)) \\
&= 1- \d_x((C_0^c \times X) \cup (X \times D_0^c)) \\
&\ge 1- (\d_x(C_0^c \times X) + \d_x(X \times D_0^c)) \\
&=1 ,
\ea
so $\d(x,A) =1$ as required. 

It remains to check that this comonoid is commutative. Recalling that the swap $\swa: (X \times X, \S \ot \S) \longrightarrow (X \times X, \S \ot \S)$ on $(X,\S) \ot (X,\S)$ is the deterministic stochastic map given by $X \times X \to X \times X; (x,y) \mapsto (y,x)$, it is immediately clear that the comultiplication $\comult$ is commutative. 
\end{proof}

Arguing as for models of $\set$, given a causal theory $\c$, strong monoidal functors $P: \c \to \stoch$ are thus specified by arbitrary choices of measurable space for each atomic variable, and a subsequent arbitrary choices of causal mechanisms of the required domain and codomain. 

These interpretations of the causal mechanisms give rise to a joint probability measure compatible with the causal structure underlying the causal theory. Indeed, this can be seen as the key difference between a model of a causal theory in $\stoch$ and a Bayesian network: models of causal theories privilege factorisations, while Bayesian networks only care about the joint probability measure.

\begin{thm} \label{th.main}
Let $G$ be a directed acyclic graph with vertex set $V$ and let \linebreak \mbox{$P: \c_G \to \stoch$} be a model of the causal theory $\c_G$ in $\stoch$. Then the causal \linebreak structure $G$ and the probability measure defined by the prior $P[V]$ are compatible.
\end{thm}
\begin{proof}
Recall that $[V]$ is the prior on the collection consisting of one copy of each of the atomic variables $V$. For each $v \in V$ we have a measurable space $Pv$, and as \mbox{$P[V]: \varnothing \to PV$} is a point of $\stoch$ it defines a joint probability measure on the product measurable space $PV$. We must show that $P[V]$ has the required factorisation.

To this end, choose some ancestral ordering of $V$, writing $V$ now as $\{v_1, \dots, v_n\}$ with the elements numbered according to this ordering. By construction, the string diagram of the prior $[V]$ consists of one copy of each causal mechanism $[v_i|\pa(v_i)]$ and $k_i$ copies of each comultiplication $\comult_{v_i}$, where $k_i$ is the number of children of the vertex $v_i$. As each $v_i$ appears exactly once as the codomain of the causal mechanisms $[v_i|\pa(v_i)]$, the coassociativity of each comonoid and the rules of the graphical calculus for symmetric monoidal categories, show that any way of connecting these elements to form a morphism $\varnothing \to V$ produces the same morphism. In particular, we may build $[V]$ as the composite of the morphisms $[V]_i: v_1\dots v_{i-1} \to v_1\dots v_i$ defined by
\[
[V]_i = \quad 
\begin{aligned}
\begin{tikzpicture}[scale=1.2]
	\begin{pgfonlayer}{nodelayer}
		\node [style=none] (0) at (-2, 1.75) {};
		\node [style=none] (1) at (-2, -0.75) {};
		\node [style=none] (2) at (-1, -0.75) {};
		\node [style=none] (3) at (-1, 1.75) {};
		\node [style=none] (4) at (0.25, -0.75) {};
		\node [style=none] (5) at (1.25, -0.75) {};
		\node [style=none] (6) at (0.25, 0) {};
		\node [style=none] (7) at (1.25, 0) {};
		\node [style=none] (8) at (-0.5, 1) {};
		\node [style=none] (9) at (0.5, 1) {};
		\node [style=none] (10) at (1, 1) {};
		\node [style=none] (11) at (2, 1) {};
		\node [style=none] (12) at (-0.5, 1.75) {};
		\node [style=none] (13) at (0.5, 1.75) {};
		\node [style=none] (14) at (2.25, 1) {};
		\node [style=none] (15) at (1.5, 1.5) {};
		\node [style=none] (16) at (0.75, 1) {};
		\node [style=none] (17) at (0.75, 1.5) {};
		\node [style=none] (18) at (2.25, 1.5) {};
		\node [style=none] (19) at (1.5, 1.75) {};
		\node [style=none] (20) at (-1.5, -0.25) {$\dots$};
		\node [style=none] (21) at (0.75, -0.5) {$\dots$};
		\node [style=none] (22) at (0, 1.5) {$\dots$};
		\node [style=none] (23) at (1.5, 0.75) {$\dots$};
		\node [style=none] (24) at (1.5, 1.25) {$v_i|\pa(v_i)$};
		\node [style=none] (25) at (-1.5, -1) {$(v_j \notin \pa(v_i))$};
		\node [style=none] (26) at (0.75, -1) {$(v_j \in \pa(v_i))$};
		\node [style=none] (27) at (0, 2) {$v_1v_2\dots v_i$};
		\node [style=none] (28) at (0, -1.5) {$v_1v_2\dots v_{i-1}$};
		\node [style=none] (29) at (0, 0.75) {$\dots$};
		\node [style=none] (30) at (-1.5, 1.25) {$\dots$};
		\node [style=none] (31) at (0, 2.25) {};
		\node [style=none] (32) at (0,-1.75) {};
	\end{pgfonlayer}
	\begin{pgfonlayer}{edgelayer}
		\draw (0.center) to (1.center);
		\draw (3.center) to (2.center);
		\draw (6.center) to (4.center);
		\draw (7.center) to (5.center);
		\draw [in=-90, out=165] (7.center) to (9.center);
		\draw [in=-90, out=165] (6.center) to (8.center);
		\draw (8.center) to (12.center);
		\draw (9.center) to (13.center);
		\draw [in=-90, out=15] (6.center) to (10.center);
		\draw [in=-90, out=15] (7.center) to (11.center);
		\draw (14.center) to (16.center);
		\draw (16.center) to (17.center);
		\draw (17.center) to (18.center);
		\draw (18.center) to (14.center);
		\draw (15.center) to (19.center);
	\end{pgfonlayer}
\end{tikzpicture}
\end{aligned}
\]
In words, the morphism $[V]_i$ is the morphism $v_1 \dots v_{i-1} \to v_1 \dots v_i$ constructed by applying a comultiplication to each of the parents of $v_i$, and then applying the causal morphism $[v_i|\pa(v_i)]$. Note that as we have ordered the set $V$ with an ancestral ordering, all parents of $v_i$ do lie in the set of predecessors of $v_i$.

Observe now that given any stochastic map $k: (X, \S_X) \to (Y, \S_Y)$, if $\comult$ is the unique comultiplication on $X$, then the composite
\[
\begin{aligned}
\begin{tikzpicture}
	\begin{pgfonlayer}{nodelayer}
		\node [style=none] (0) at (0, -0.25) {};
		\node [style=none] (1) at (0, -0.75) {};
		\node [style=none] (2) at (0.5, 0.25) {};
		\node [style=none] (3) at (-0.5, 0.25) {};
		\node [style=sq] (4) at (0.5, 0.25) {$k$};
		\node [style=none] (5) at (-0.5, 0.75) {};
		\node [style=none] (6) at (0.5, 0.75) {};
	\end{pgfonlayer}
	\begin{pgfonlayer}{edgelayer}
		\draw [in=0, out=-90, looseness=0.75] (2.center) to (0.center);
		\draw [in=180, out=-90, looseness=0.75] (3.center) to (0.center);
		\draw (0.center) to (1.center);
		\draw (2.center) to (4);
		\draw (4) to (6.center);
		\draw (3.center) to (5.center);
	\end{pgfonlayer}
\end{tikzpicture}
\end{aligned}
: (X, \S_X) \longrightarrow (X \times Y, \S_X \ot \S_Y)
\]
is given by
\[
\big((\id_X \ot k) \circ \comult \big)(x,A\times B) = \int_{X \times X} \chi_A(-)k(-,B) \,d\comult_x = \chi_A(x)k(x,B)
\]
for all $x \in X$, $A \in \S_X$, $B \in \S_Y$. Furthermore, if $\mu: \ast \to (X,\S_X)$ is a measure on $X$, then
\[
\begin{aligned}
\begin{tikzpicture}
	\begin{pgfonlayer}{nodelayer}
		\node [style=none] (0) at (0, -0.25) {};
		\node [style=tri] (1) at (0, -0.75) {$\mu$};
		\node [style=none] (2) at (0.5, 0.25) {};
		\node [style=none] (3) at (-0.5, 0.25) {};
		\node [style=sq] (4) at (0.5, 0.25) {$k$};
		\node [style=none] (5) at (-0.5, 0.75) {};
		\node [style=none] (6) at (0.5, 0.75) {};
	\end{pgfonlayer}
	\begin{pgfonlayer}{edgelayer}
		\draw [in=0, out=-90, looseness=0.75] (2.center) to (0.center);
		\draw [in=180, out=-90, looseness=0.75] (3.center) to (0.center);
		\draw (0.center) to (1);
		\draw (2.center) to (4);
		\draw (4) to (6.center);
		\draw (3.center) to (5.center);
	\end{pgfonlayer}
\end{tikzpicture}
\end{aligned}
: (\ast, \{\varnothing, \ast\}) \longrightarrow (X \times Y, \S_X \ot \S_Y)
\]
is given by
\[
\big((\id_X \ot k) \circ \comult \circ \mu\big)(A\times B) = \int_{A} k(x,B) \, d\mu
\]
for all $A \in \S_X$, $B \in \S_Y$.

Thus, taking the image under $P$, each of the maps $[V]_i$ gives
\[
P[V]_i(x_1, \dots, x_{i-1},A_1 \times \dots \times A_i) = \prod_{j=1}^{i-1} \chi_{A_j}(x_j) P[v_i|\pa(v_i)](x_1, \dots,x_{i-1}, A_i).
\]
for all $x_j \in Pv_j$, $A_j \in \S_{Pv_j}$, and composing them gives
\ba
P[V](A_1 \times \dots \times A_n) &= P[V]_n \circ \dots \circ P[V]_2 \circ P[V]_1(A_1 \times A_2 \times \dots \times A_n) \\
&=  \int_{A_1} P[V]_n \circ \dots \circ P[V]_2(-,A_2 \times \dots \times A_n) dP[v_1|\pa(v_1)] \\
&= \qquad \qquad \vdots \\
&=  \int_{A_1} \int_{A_2} \dots \int_{A_n} \,dP[v_n|\pa(v_n)] \dots dP[v_2|\pa(v_2)]dP[v_1|\pa(v_1)]
\ea
for all $A_j \in \S_{Pv_j}$. This is a factorisation of $P[V]$ of the required type.
\end{proof}

We call the pair $(G,P[V])$ the \emph{Bayesian network induced by $P$}. Thus we see that, given a causal structure and any stochastic causal model of its theory, the induced joint distribution on the atomic variables of the theory forms a Bayesian network with the causal structure. On the other hand, if we have a Bayesian network on this causal structure, we may construct a stochastic causal model inducing this distribution, but only by picking some factorisation of our joint distribution. To iterate, this is the key distinction between Bayesian networks and stochastic causal models: a Bayesian network on a causal structure requires only that there \emph{exist} a factorisation for the distribution respecting the causal structure, while a stochastic causal model \emph{explicitly chooses} a factorisation. 

An advantage of working within a causal theory, rather than just with the induced Bayesian network, is that the additional structure allows neat representations of operations that one might want to do to a Bayesian network. The remainder of this dissertation comprises a brief exploration of this. We conclude this section by noting that the priors of the causal theory represent the marginals of the induced Bayesian network. 

\begin{thm}
Given a model of $P: \c \to \stoch$ of a causal theory $\c_G$ and a set $w$ of atomic variables of $\c_G$, the prior $P[w]$ is equal to the marginal on the product measurable space of the variables in $w$ of the induced Bayesian network.
\end{thm}
\begin{proof}
We first note a more general fact: given a joint probability measure on $X \times Y$ expressed as a point in $\stoch$, marginalisation over $Y$ can be expressed as the composite of this point with the product of the identity $\idline$ on $X$ and counit $\counit$ on $Y$. Indeed, if $\mu: \ast \to X \times Y$ is a probability measure, then 
\[
(\idline \counit \circ \mu)(A) = \int_{X \times Y} \idline \counit (x,y,A) \, d\mu = \int_{X \times Y} \chi_A(x)\,d\mu = \mu_X(A).
\]
Thus the marginals of $P[V]$ may be expressed by composing $P[V]$ with counits on the factors marginalised over. We wish to show that these are the priors $P[w]$ of $\c$. Reasoning inductively, to show this it is enough to show that the composite of a prior with the product of a counit on one of its factors and identity maps on the remaining factors is again a prior.

Let $w$ be a set of atomic variables of $\c$ and let $v \in w$. We will show that the composite of $P[w]$ with the product of the counit on $v$ and identity on $w\setminus\{v\}$ is equal to the prior $P[w \setminus \{v\}]$. We split into two cases: when $v$ has a consequence in $G_{\varnothing \to w}$, and when $v$ has no consequences in $G_{\varnothing \to w}$. For the first case, observe that $G_{\varnothing \to w} = G_{\varnothing \to w\setminus\{v\}}$ and $k_v \ge 1$. Thus the priors $[w]$ and $[w \setminus\{v\}]$ are the same but for the fact we compose with one extra comultiplication after the causal mechanism $[v|\pa(v)]$ in the case of the prior $[w]$. Thus the composite of $[w]$ with a counit on $v$ is equal to $[w \setminus \{v\}]$ by the counitality law on $v$.

To deal with the second case we must work in $\stoch$ and make use of the fact that the monoidal unit in $\stoch$ is terminal. Indeed, as the monoidal unit in $\stoch$ is terminal, in $\stoch$ we have the equality of morphisms
\[
\begin{aligned}
\begin{tikzpicture}
	\begin{pgfonlayer}{nodelayer}
		\node [style=dot] (0) at (0, 1.25) {};
		\node [style=none] (1) at (-0.5, -0.25) {};
		\node [style=none] (2) at (0.5, -0.25) {};
		\node [style=none] (3) at (-0.5, 0.75) {};
		\node [style=none] (4) at (0, 0.75) {};
		\node [style=none] (5) at (0.5, 0.25) {};
		\node [style=none] (6) at (-0.5, 0.25) {};
		\node [style=none] (7) at (0, -0.5) {$\pa(v)$};
		\node [style=none] (8) at (0, 0) {$\dots$};
		\node [style=none] (9) at (-0.75, 0.75) {};
		\node [style=none] (10) at (0.75, 0.75) {};
		\node [style=none] (11) at (0.75, 0.25) {};
		\node [style=none] (12) at (-0.75, 0.25) {};
		\node [style=none] (13) at (0, 0.5) {$v|\pa(v)$};
	\end{pgfonlayer}
	\begin{pgfonlayer}{edgelayer}
		\draw (5.center) to (2.center);
		\draw (0) to (4.center);
		\draw (6.center) to (1.center);
		\draw (11.center) to (12.center);
		\draw (12.center) to (9.center);
		\draw (9.center) to (10.center);
		\draw (10.center) to (11.center);
	\end{pgfonlayer}
\end{tikzpicture}
\end{aligned}
\quad = \quad
\begin{aligned}
\begin{tikzpicture}
	\begin{pgfonlayer}{nodelayer}
		\node [style=none] (0) at (0, 1.25) {};
		\node [style=none] (1) at (-0.5, -0.25) {};
		\node [style=none] (2) at (0.5, -0.25) {};
		\node [style=dot] (3) at (0.5, 0.5) {};
		\node [style=dot] (4) at (-0.5, 0.5) {};
		\node [style=none] (5) at (0, -0.5) {$\pa(v)$};
		\node [style=none] (6) at (0, 0) {$\dots$};
	\end{pgfonlayer}
	\begin{pgfonlayer}{edgelayer}
		\draw (3) to (2.center);
		\draw (4) to (1.center);
	\end{pgfonlayer}
\end{tikzpicture}
\end{aligned}
\]
As $v$ has no consequences in $G_{\varnothing \to w}$, the causal mechanism $[v|\pa(v)]$ is not followed with any comultiplications in the construction of $[w]$. Thus, after composing $P[w]$ with a counit on $v$ we may invoke the above identity, and then invoke the counitality law for each of the parents of $v$. This means that $P[w]$ is equal to the morphism constructed without the causal mechanism  $[v| \pa(v)]$, and with one fewer comultiplications on each of the parents of $v$. But this is precisely the morphism $P[w\setminus\{v\}]$. This proves the theorem.
\end{proof}

\section{Application: visualising Simpson's paradox}

One of the strengths of causal theories is that their graphical calculi provide a guide to which computations should be made if one wants to respect a causal structure, and in doing so also clarify what these computations mean. An illustration can be found in an exploration of confounding variables and Simpson's paradox. This section owes much to Pearl \cite[Chapter 6]{P}, extending the basics of that discussion with our new graphical notation.

Simpson's paradox refers to the perhaps counterintuitive fact that it is possible for to have data such that, for all outcomes of a confounding variable, a fixed outcome of the independent variable makes another fixed outcome of the dependent variable \emph{more} likely, and yet also that upon aggregation over the confounding variable the same fixed outcome of the independent variable makes the same fixed outcome of the dependent variable \emph{less} likely. This is perhaps best understood through an example.

Consider the following, somewhat simplified, scenario: let us imagine that we wish to test the efficacy of a proposed new treatment for a certain heart condition. In our clinical experiment, we take two groups of patients each suffering from the heart condition and, after treating the patients in the proposed way, record whether they recover. In addition, as we know that having a healthy blood pressure is also an important factor in recovery, we also take records of whether the blood pressure of the patient is within healthy bounds or otherwise at the conclusion of the treatment programme. This gives three binary variables: an independent variable $T = \{t,\neg t\}$, a dependent variable $R = \{r,\neg r\}$, and a third, possibly confounding variable $B = \{b,\neg b\}$, where will think of the variables as representing the truth or otherwise of the following propositions:

\begin{itemize}

\item[T:] the patient receives \emph{treatment} for heart condition.

\item[R:] the patient has \emph{recovered} at the conclusion of treatment.

\item[B:] the patient has healthy \emph{blood pressure} at post-treatment checkup.

\end{itemize}
\noindent Suppose then that our experiment yields the data of Figure \ref{fig.sp}.

\begin{figure}[h] 
\begin{center}
\begin{tabular}{cc||cc}
&\multicolumn{3}{c}{\phantom{MM} T} \\
&&$t$&$\neg t$ \\ \cline{2-4}
\multirow{2}{*}{$R$}&$r$ & 39 & 42 \\
&$\neg r$ & 61 & 58 \\
\end{tabular}
\hspace{1.5cm}
\begin{tabular}{cc||cc|cc}
&\multicolumn{5}{c}{\phantom{M} TB} \\
&&$t,b$&$\neg t,b$&$t,\neg b$&$\neg t,\neg b$ \\ \cline{2-6}
\multirow{2}{*}{$R$} & $r$ & 30 & 40 & 9 & 2 \\
&$\neg r$ & 10 & 40 & 51 & 18 \\
\end{tabular}
\end{center}
\caption{\small{The table on the left displays the experiment data for the treatment and recovery of patients, while the table on the right displays the experiment when further subdivided according to the blood pressure of the patients measured post-treatment.}}
\label{fig.sp}
\end{figure}

In these data we see the so-called paradox: for both patients with healthy and unhealthy blood pressure, treatment seems to significantly improve the chance of recovery, with the recovery rates increasing from 50\% to 80\% and from 10\% to 15\% respectively when patients are treated. On the other hand, when the studies are taken a whole, it seems treatment has no significant effect on the recovery rate, which drops slightly from 42\% to 39\%. Given this result, it is not clear whether the experiment indicates that treatment improves or even impairs chance of recovery. Should we or should we not then recommend the treatment?

The answer depends on the causal relationships between our variables. Suppose that the treatment acts in part via affecting blood pressure. Then the causal structure of the variables is given by the graph
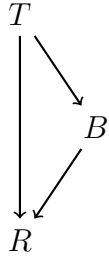
\begin{center}
\begin{tikzpicture}
  \node (I) at (0,1.5) {$T$};
  \node (D) at (0,-1.5)  {$R$};
  \node (C) at (1,0)  {$B$};
  \foreach \from/\to in {I/D,I/C,C/D}
    \draw[->,thick] (\from) -- (\to);
\end{tikzpicture}
\end{center}
In this case we should make our decision with respect to the aggregated data: else when we condition on the post-treatment blood pressures we eliminate information about how the treatment is affecting blood pressure, and so eliminate information about an important causal pathway between treatment and recovery. We therefore should not recommend treatment---although when we control for blood pressure the treatment seems to improve chances of recovery, the treatment also makes it less likely that a healthy blood pressure will be reached, offsetting any gain.

On the other hand, suppose that the treatment works in a way that has no effect on blood pressure. Then from the fact that the blood pressure and treatment variables are not independent we may deduce that the blood pressure variable biased selection for the treatment trial, and so the causal structure representing these variables is
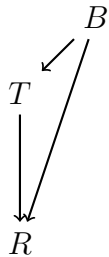
\begin{center}
\begin{tikzpicture}
  \node (I) at (0,0.5) {$T$};
  \node (D) at (0,-1.5)  {$R$};
  \node (C) at (1,1.5)  {$B$};
  \foreach \from/\to in {I/D,C/I,C/D}
    \draw[->,thick] (\from) -- (\to);
\end{tikzpicture}
\end{center}
Here we \emph{should} pay attention to the data when divided according to blood pressure, as by doing this we control for the consequences of this variable. We then see that no matter whether a patient has factors leading to healthy or unheathly blood pressure, the treatment raises their chance of recovery by a significant proportion.

These ideas are codified in the corresponding causal theories and their maps, with the causal effect of treatment on recovery expressed via the causal conditional $[R||T]$. For the first structure, let the corresponding causal theory be $\c_1$, and the data give the following interpretations of the causal mechanisms $P: \c_1 \to \finstoch$:
\[
\begin{aligned}
\begin{tikzpicture}
	\begin{pgfonlayer}{nodelayer}
		\node [style=tri] (0) at (0, 0) {$T$};
		\node [style=none] (1) at (0, 0.60) {};
		\node [style=none] (2) at (0, -0.60) {};
	\end{pgfonlayer}
	\begin{pgfonlayer}{edgelayer}
		\draw (1.center) to (0);
	\end{pgfonlayer}
\end{tikzpicture}
\end{aligned}
\:
= \begin{pmatrix} 0.5 \\ 0.5 \end{pmatrix} \qquad 
\begin{aligned}
\begin{tikzpicture}
	\begin{pgfonlayer}{nodelayer}
		\node [style=sq] (0) at (0, 0) {$B|T$};
		\node [style=none] (1) at (0, 0.60) {};
		\node [style=none] (2) at (0, -0.60) {};
	\end{pgfonlayer}
	\begin{pgfonlayer}{edgelayer}
		\draw (1.center) to (0);
		\draw (0) to (2.center);
	\end{pgfonlayer}
\end{tikzpicture}
\end{aligned} 
\:
= \begin{pmatrix} 0.4 & 0.8 \\ 0.6 & 0.2 \end{pmatrix} 
\]
\[ 
\begin{aligned}
\begin{tikzpicture}
	\begin{pgfonlayer}{nodelayer}
		\node [style=none] (0) at (0.75, 0.25) {};
		\node [style=none] (1) at (-0.5, -0.25) {};
		\node [style=none] (2) at (0.5, -0.60) {};
		\node [style=none] (3) at (0, 0) {$R|TB$};
		\node [style=none] (4) at (0, 0.60) {};
		\node [style=none] (5) at (0.75, -0.25) {};
		\node [style=none] (6) at (-0.75, -0.25) {};
		\node [style=none] (7) at (0.5, -0.25) {};
		\node [style=none] (8) at (0, 0.25) {};
		\node [style=none] (9) at (-0.75, 0.25) {};
		\node [style=none] (10) at (-0.5, -0.60) {};
	\end{pgfonlayer}
	\begin{pgfonlayer}{edgelayer}
		\draw (7.center) to (2.center);
		\draw (4.center) to (8.center);
		\draw (5.center) to (6.center);
		\draw (6.center) to (9.center);
		\draw (9.center) to (0.center);
		\draw (0.center) to (5.center);
		\draw (1.center) to (10.center);
	\end{pgfonlayer}
\end{tikzpicture}
\end{aligned}
\:
= \begin{pmatrix} 0.75 & 0.15 & 0.5 & 0.1 \\ 0.25 & 0.85 & 0.5 & 0.9 \end{pmatrix}
\]
Here we have written the maps as their representations in $\mathbf{SMat}$ with respect to the basis ordering given in our definition of the variables. The causal conditional $P[R||T]$ is then 
\[
P[R||T] =
\:
\begin{aligned}
\begin{tikzpicture}
	\begin{pgfonlayer}{nodelayer}
		\node [style=none] (0) at (0.75, 1.25) {};
		\node [style=none] (1) at (-0.5, 0.75) {};
		\node [style=sq] (2) at (0.5, 0) {$B|T$};
		\node [style=none] (3) at (0, 1) {$R|TB$};
		\node [style=none] (4) at (0, 1.75) {};
		\node [style=none] (5) at (0.75, 0.75) {};
		\node [style=none] (6) at (-0.75, 0.75) {};
		\node [style=none] (7) at (0.5, 0.75) {};
		\node [style=none] (8) at (0, 1.25) {};
		\node [style=none] (9) at (-0.75, 1.25) {};
		\node [style=none] (10) at (0, -1.5) {};
		\node [style=none] (11) at (-0.5, -0.5) {};
		\node [style=none] (12) at (-0.5, -0.5) {};
		\node [style=none] (13) at (0, -1) {};
		\node [style=none] (14) at (0.5, -0.5) {};
		\node [style=none] (15) at (0, -1.75) {$T$};
		\node [style=none] (16) at (0, 2) {$R$};
	\end{pgfonlayer}
	\begin{pgfonlayer}{edgelayer}
		\draw (7.center) to (2);
		\draw (4.center) to (8.center);
		\draw (5.center) to (6.center);
		\draw (6.center) to (9.center);
		\draw (9.center) to (0.center);
		\draw (0.center) to (5.center);
		\draw [in=0, out=-90, looseness=0.75] (14.center) to (13.center);
		\draw [in=180, out=-90, looseness=0.75] (12.center) to (13.center);
		\draw (13.center) to (10.center);
		\draw (1.center) to (11.center);
		\draw (2) to (14.center);
	\end{pgfonlayer}
\end{tikzpicture}
\end{aligned}
\:
= \begin{pmatrix} 0.39 & 0.42 \\ 0.61 & 0.58 \end{pmatrix} 
\]
The elements of the first row of this matrix represent the probability of recovery given treatment and no treatment respectively, and so this agrees with our assertion that in this case we should view the treatment as ineffective, and perhaps marginally harmful.

On the other hand, writing the corresponding causal theory to the second causal structure as $\c_2$, in this case the data gives the stochastic model $Q: \c_2 \to \finstoch$ defined by the maps
\[
\begin{aligned}
\begin{tikzpicture}
	\begin{pgfonlayer}{nodelayer}
		\node [style=tri] (0) at (0, 0) {$T$};
		\node [style=none] (1) at (0, 0.60) {};
		\node [style=none] (2) at (0, -0.60) {};
	\end{pgfonlayer}
	\begin{pgfonlayer}{edgelayer}
		\draw (1.center) to (0);
	\end{pgfonlayer}
\end{tikzpicture}
\end{aligned}
\:
= \begin{pmatrix} 0.5 \\ 0.5 \end{pmatrix} \qquad 
\begin{aligned}
\begin{tikzpicture}
	\begin{pgfonlayer}{nodelayer}
		\node [style=tri] (0) at (0, 0) {$B$};
		\node [style=none] (1) at (0, 0.60) {};
		\node [style=none] (2) at (0, -0.60) {};
	\end{pgfonlayer}
	\begin{pgfonlayer}{edgelayer}
		\draw (1.center) to (0);
	\end{pgfonlayer}
\end{tikzpicture}
\end{aligned}
\:
= \begin{pmatrix} 0.6 \\ 0.4 \end{pmatrix} 
\]
\[ 
\begin{aligned}
\begin{tikzpicture}
	\begin{pgfonlayer}{nodelayer}
		\node [style=none] (0) at (0.75, 0.25) {};
		\node [style=none] (1) at (-0.5, -0.25) {};
		\node [style=none] (2) at (0.5, -0.60) {};
		\node [style=none] (3) at (0, 0) {$R|TB$};
		\node [style=none] (4) at (0, 0.60) {};
		\node [style=none] (5) at (0.75, -0.25) {};
		\node [style=none] (6) at (-0.75, -0.25) {};
		\node [style=none] (7) at (0.5, -0.25) {};
		\node [style=none] (8) at (0, 0.25) {};
		\node [style=none] (9) at (-0.75, 0.25) {};
		\node [style=none] (10) at (-0.5, -0.60) {};
	\end{pgfonlayer}
	\begin{pgfonlayer}{edgelayer}
		\draw (7.center) to (2.center);
		\draw (4.center) to (8.center);
		\draw (5.center) to (6.center);
		\draw (6.center) to (9.center);
		\draw (9.center) to (0.center);
		\draw (0.center) to (5.center);
		\draw (1.center) to (10.center);
	\end{pgfonlayer}
\end{tikzpicture}
\end{aligned}
\:
= \begin{pmatrix} 0.75 & 0.15 & 0.5 & 0.1 \\ 0.25 & 0.85 & 0.5 & 0.9 \end{pmatrix}
\]
We then may compute the causal conditional $[R||T]$ to be
\[
Q[R||T] = 
\: 
\begin{aligned}
\begin{tikzpicture}
	\begin{pgfonlayer}{nodelayer}
		\node [style=none] (0) at (0.75, 1.25) {};
		\node [style=none] (1) at (-0.5, 0.75) {};
		\node [style=tri] (2) at (0.5, -.1) {$B$};
		\node [style=none] (3) at (0, 1) {$R|TB$};
		\node [style=none] (4) at (0, 1.75) {};
		\node [style=none] (5) at (0.75, 0.75) {};
		\node [style=none] (6) at (-0.75, 0.75) {};
		\node [style=none] (7) at (0.5, 0.75) {};
		\node [style=none] (8) at (0, 1.25) {};
		\node [style=none] (9) at (-0.75, 1.25) {};
		\node [style=none] (10) at (-0.5, 0.25) {};
		\node [style=none] (11) at (0, -1.25) {};
		\node [style=none] (12) at (0, -1.5) {$T$};
		\node [style=none] (13) at (0, 2) {$R$};
	\end{pgfonlayer}
	\begin{pgfonlayer}{edgelayer}
		\draw (7.center) to (2);
		\draw (4.center) to (8.center);
		\draw (5.center) to (6.center);
		\draw (6.center) to (9.center);
		\draw (9.center) to (0.center);
		\draw (0.center) to (5.center);
		\draw [in=90, out=-90, looseness=1.50] (10.center) to (11.center);
		\draw (10.center) to (1.center);
	\end{pgfonlayer}
\end{tikzpicture}
\end{aligned}
\:
= \begin{pmatrix} 0.51 & 0.34 \\ 0.49 & 0.66 \end{pmatrix} 
\]
This again agrees with the above assertion that with this causal structure the treatment is effective, as here the probability of recovery with treatment is $0.51$, compared with a probability of recovery without treatment of $0.34$. In this case the map $Q[R||T]$ is the \emph{only} inferential map from $T$ to $R$; it thus may be seen as the only way to deduce information about $R$ from information about $T$ consistent with their causal relationship. Thus within the framework given by the causal theory, there is no possible way to come to the wrong conclusion about the efficacy of the treatment.

In the first case, however, there is one other map; we may infer information about recovery given treatment via
\[
\begin{aligned}
\begin{tikzpicture}
	\begin{pgfonlayer}{nodelayer}
		\node [style=none] (0) at (0.75, 1.25) {};
		\node [style=none] (1) at (-0.5, 0.75) {};
		\node [style=sq] (2) at (0.5, 0.15) {$B|T$};
		\node [style=none] (3) at (0, 1) {$R|TB$};
		\node [style=none] (4) at (0, 1.75) {};
		\node [style=none] (5) at (0.75, 0.75) {};
		\node [style=none] (6) at (-0.75, 0.75) {};
		\node [style=none] (7) at (0.5, 0.75) {};
		\node [style=none] (8) at (0, 1.25) {};
		\node [style=none] (9) at (-0.75, 1.25) {};
		\node [style=none] (10) at (-0.5, 0.25) {};
		\node [style=none] (11) at (0, -1.25) {};
		\node [style=none] (12) at (0, -1.5) {$T$};
		\node [style=none] (13) at (0, 2) {$R$};
		\node [style=tri] (14) at (0.5, -0.55) {$T$};
	\end{pgfonlayer}
	\begin{pgfonlayer}{edgelayer}
		\draw (7.center) to (2);
		\draw (4.center) to (8.center);
		\draw (5.center) to (6.center);
		\draw (6.center) to (9.center);
		\draw (9.center) to (0.center);
		\draw (0.center) to (5.center);
		\draw [in=90, out=-90, looseness=1.50] (10.center) to (11.center);
		\draw (10.center) to (1.center);
		\draw (2) to (14);
	\end{pgfonlayer}
\end{tikzpicture}
\end{aligned}
\:
= \begin{pmatrix} 0.51 & 0.34 \\ 0.49 & 0.66 \end{pmatrix} 
\]
As suggested by the form of the string diagram, this may be interpreted as the chance of recovery if the effect of the treatment on blood pressure is nullified, but nonetheless assuming that the proportion of patients presenting healthy blood pressures at the conclusion of the treatment was typical. In particular, this indicates that if it was inevitable that a group of patients would end up with healthy blood pressure levels in the proportion specified by $[B|T]\circ[T]$, then the treatment would be effective for this group.

Note that the string diagrams themselves encode the flow of causal influence in their depictions of the conditionals. In doing so they make the source of confusion patently clear: we may judge the effect of treatment on recovery in two different ways, one in which use information about how treatment affected blood pressure, and one in which we forget this link and assume the variables are unrelated. 

Finally, observe that causal structures are thus very relevant when interpreting data, and awareness of them can allow one to extract information that could not otherwise be extracted. Indeed, although under the second causal structure the fact that the blood pressure variable biased our selection procedure for treatment---making it more likely that we treated those with unhealthy blood pressure---can be seen as ill-considered experiment design, we see that nonetheless an understanding of the causal structure allowed us to recover the correct conclusion from the data. This becomes critically useful in cases when we do not have the abilities to correct such biases methodologically, such as when data is taken from observational studies rather than controlled experiments.

\chapter{The Structure of Stochastic Causal Models}

Our aim through this dissertation has been to develop tools to discuss causality, and in particular causal relationships between random variables. Our claim is now these are well described by \emph{stochastic causal models}: models $P: \c \to \cgst$ of a causal theory $\c$ in $\cgst$. Indeed, we have seen these are slight generalisations of Bayesian networks in which the factorisation of the joint distribution is made explicit. One advantage of moving to this setting is that we now have a natural notion of map between causal models: a monoidal natural transformation between their functors. We begin this chapter by exploring these, before using the knowledge we gain to look at the existence or otherwise of some basic universal constructions in the category of stochastic causal models.

\section{Morphisms of stochastic causal models}

Fix a causal theory $\c$. Although we have so far had no problems discussing models of causal theories in $\stoch$, we shall define the \emph{stochastic causal models} of $\c$ to be the objects of the category $\cgst^\c_{SSM}$ of strong symmetric monoidal functors $\c \to \cgst$. This more restrictive definition allows for a more well-behaved notion of maps between stochastic causal models. Indeed, we take the notion of morphism in $\cgst^\c_{SSM}$---a monoidal natural transformation between functors---to be the notion of map between stochastic causal models. As we will see in this section, these are much like deterministic stochastic maps. Our aim will be to define the terms in, and then prove, the following theorem:

\begin{thm} \label{th.cms}
Morphisms of stochastic causal models factor into a coarse graining followed by an embedding.
\end{thm} 

To this end, let $P,Q: \c \to \cgst$ be stochastic causal models, and let $\a: P \Rightarrow Q$ be a monoidal natural transformation. By definition, this means we have a collection of stochastic maps $\a_w: Pw \to Qw$ such that for all variables $w, w' \in \c$ and all morphisms $r:Pw \to Pw'$ the following diagrams commute:
\[
\begin{aligned}
\xymatrixcolsep{3pc}
\xymatrixrowsep{3pc}
\xymatrix{
Pw \ar[r]^{Pr} \ar[d]_{\a_w} & Qw \ar[d]^{Qr}\\
Pw' \ar[r]_{\a_{w'}} & Qw'
}
\end{aligned}
\qquad \qquad
\begin{aligned}
\xymatrix{
P\varnothing \ar[rr]^{\a_\varnothing}&& Q\varnothing \\
& \ast \ar[ul]^{P_\ast} \ar[ur]_{Q_\ast}
} 
\end{aligned}
\]
\[
\begin{aligned}
\xymatrixcolsep{3pc}
\xymatrixrowsep{3pc}
\xymatrix{
Pw \ot Pw' \ar[r]^{\a_w \ot \a_{w'}} \ar[d]_{P_{\ot,w,w'}} & Qw \ot Qw' \ar[d]^{Q_{\ot,w,w'}}\\
Pww' \ar[r]_{\a_{ww'}} & Qww'
}
\end{aligned}
\]
We can, however, write this definition a bit more efficiently. 

As $P\varnothing$ and $Q\varnothing$ are isomorphic to the monoidal unit, and as the monoidal unit $\ast$ of $\cgst$ is terminal, the above triangle gives no constraints on the morphisms. The lower square specifies the relationships between the maps $\a_w$ and $\a_{w'}$ and the map $\a_{ww'}$ on the product variable. Due to this, it suffices to define the natural transformation $\a$ only on the atomic variables $v$ of $\c$, and let the commutativity of the square specify the maps $\a_w$ on the remaining variables. It thus remains to ensure that our maps $\a_v$ on the atomic variables $v$ satisfy the defining square of a natural transformation. 

We first consider the constraints given by the comonoid maps. The counit maps provide no constraint: since $Q\varnothing$ is terminal, the diagram
\[
\xymatrixcolsep{3pc}
\xymatrixrowsep{3pc}
\xymatrix{
Pv \ar[r]^{\a_v} \ar[d]_{\counit} & Qv \ar[d]^{\counit}\\
P\varnothing \ar[r]^\sim_{\a_\varnothing} & Q\varnothing
}
\]
always commutes. On the other hand, the comultiplication maps heavily constrain the $\a_v$: they require that
\[
\xymatrixcolsep{3pc}
\xymatrixrowsep{3pc}
\xymatrix{
Pv \ar[r]^{\a_v} \ar[d]_{\comult} & Qv \ar[d]^{\comult}\\
Pv \ot Pv \ar[r]_{\a_v \ot \a_v} & Qv \ot Qv
}
\]
or in string diagrams:
\[
\begin{aligned}
\begin{tikzpicture}[scale=.7]
	\begin{pgfonlayer}{nodelayer}
		\node [style=none] (0) at (0, 0) {};
		\node [style=none] (1) at (0, -1) {};
		\node [style=sq] (2) at (1, 1) {$\a_v$};
		\node [style=sq] (3) at (-1, 1) {$\a_v$};
		\node [style=none] (4) at (-1, 1.75) {};
		\node [style=none] (5) at (1, 1.75) {};
	\end{pgfonlayer}
	\begin{pgfonlayer}{edgelayer}
		\draw (0.center) to (1.center);
		\draw [in=0, out=-90, looseness=0.75] (2) to (0.center);
		\draw [in=180, out=-90, looseness=0.75] (3) to (0.center);
		\draw (4.center) to (3);
		\draw (5.center) to (2);
	\end{pgfonlayer}
\end{tikzpicture}
\end{aligned}
\quad
=
\quad
\begin{aligned}
\begin{tikzpicture}[scale=.7]
	\begin{pgfonlayer}{nodelayer}
		\node [style=none] (0) at (0, 0.75) {};
		\node [style=sq] (1) at (0, 0) {$\a_v$};
		\node [style=none] (2) at (1, 1.75) {};
		\node [style=none] (3) at (-1, 1.75) {};
		\node [style=none] (4) at (0, -1) {};
	\end{pgfonlayer}
	\begin{pgfonlayer}{edgelayer}
		\draw (0.center) to (1);
		\draw [bend left=45] (0.center) to (3.center);
		\draw [bend right=45] (0.center) to (2.center);
		\draw (4.center) to (1);
	\end{pgfonlayer}
\end{tikzpicture}
\end{aligned}
\]
The following lemma shows that this is true if and only if each $\a_v$ must be deterministic.

\begin{lem}
A stochastic map is a comonoid homomorphism if and only if it is deterministic.
\end{lem}
\begin{proof}
Let $k: X \to Y$ be a stochastic map. As the monoidal unit is terminal in $\cgst$, all stochastic maps preserve the counit. We thus want to show that 
\[
\begin{aligned}
\begin{tikzpicture}[scale=.7]
	\begin{pgfonlayer}{nodelayer}
		\node [style=none] (0) at (0, 0) {};
		\node [style=none] (1) at (0, -1) {};
		\node [style=sq] (2) at (1, 1) {$k$};
		\node [style=sq] (3) at (-1, 1) {$k$};
		\node [style=none] (4) at (-1, 1.75) {};
		\node [style=none] (5) at (1, 1.75) {};
	\end{pgfonlayer}
	\begin{pgfonlayer}{edgelayer}
		\draw (0.center) to (1.center);
		\draw [in=0, out=-90, looseness=0.75] (2) to (0.center);
		\draw [in=180, out=-90, looseness=0.75] (3) to (0.center);
		\draw (4.center) to (3);
		\draw (5.center) to (2);
	\end{pgfonlayer}
\end{tikzpicture}
\end{aligned}
\quad
=
\quad
\begin{aligned}
\begin{tikzpicture}[scale=.7]
	\begin{pgfonlayer}{nodelayer}
		\node [style=none] (0) at (0, 0.75) {};
		\node [style=sq] (1) at (0, 0) {$k$};
		\node [style=none] (2) at (1, 1.75) {};
		\node [style=none] (3) at (-1, 1.75) {};
		\node [style=none] (4) at (0, -1) {};
	\end{pgfonlayer}
	\begin{pgfonlayer}{edgelayer}
		\draw (0.center) to (1);
		\draw [bend left=45] (0.center) to (3.center);
		\draw [bend right=45] (0.center) to (2.center);
		\draw (4.center) to (1);
	\end{pgfonlayer}
\end{tikzpicture}
\end{aligned}
\]
if and only if $k$ is deterministic. 

Now, given $x \in X$ and $B \in \S_Y$, the left hand side of the above equality takes value
\ba
(k \ot k)\circ \comult(x,B \times B) &= \int k \ot k(-,-,B \times B) \, d\comult_x \\
&= \int k(-,B) k(-,B) d\comult_x \\
&= k(x,B)^2,
\ea
while the right hand side equals
\ba
\comult \circ k(x,B \times B) &= \int \comult(-,B \times B) \,dk_x \\
&= \int \chi_{B} \,dk_x \\
&= k(x, B).
\ea
Thus if $k$ is a comonoid homomorphism, then $k(x,B)^2 = k(x,B)$, and hence $k(x,B) = 0$ or $1$. This shows that $k$ is deterministic. Conversely, if $k$ is deterministic, then $k(x,B)^2 = k(x,B)$ for all $x \in X$, $B \in \S_Y$, so $k$ is a comonoid homomorphism.
\end{proof}

Summing up, a morphism $\a: P \Rightarrow Q$ of stochastic causal models is specified by a collection $\{\a_v\}_{v \in V_\c}$ of deterministic stochastic maps $\a_v: Pv \to Qv$ such that for all atomic variables $v \in V_\c$ the squares
\[
\xymatrixcolsep{3pc}
\xymatrixrowsep{3pc}
\xymatrix{
P(pa(v)) \ar[r]^{\a_{pa(v)}} \ar[d]_{P[v|pa(v)]} & Q(pa(v)) \ar[d]^{Q[v|pa(v)]}\\
Pv \ar[r]_{\a_{v}} & Qv
}
\]
commute. 

We say that a morphism $\a$ of stochastic models of $\c$ is an \emph{embedding} if for all objects $v$ of $\c$ the deterministic stochastic map $\a_v$ is an embedding. Similarly, we say that a morphism $\a$ of stochastic models of $\c$ is a \emph{coarse graining} if for all objects $v$ of $\c$ the deterministic stochastic map $\a_v$ is a coarse graining. Theorem \ref{th.cms} now follows from Proposition \ref{pr.dsms}, with the causal model it factors through having the induced structure. 

We caution that despite the similar terminology to deterministic stochastic maps, the situation here differs as stochastic causal models consist of much more data than measurable spaces, and so the compatibility requirements a morphism must obey here are much stricter. For example, while in $\cgst$ it is always possible to find a deterministic map between any two objects, this is rarely possible in $\cgst^\c_{SSM}$. 

Let $P,Q: \c \to \cgst$ be stochastic causal models, and let $\a: P \Rightarrow Q$ be a morphism between them. Then for any prior $[w]$ of $\c$, the diagram
\[
\xymatrix{
Pw  \ar[rr]^{\a_{w}} && Qw \\
\\
& \ast \ar[uul]^{P[w]} \ar[uur]_{Q[w]}\\
}
\]
 commutes. This says that the pushforward measure of any prior $P[w]$ along the deterministic stochastic map $\a_w$ must agree with $Q[w]$. No such map exists, for example, when $Pw$, $Qw$ are binary discrete measurable spaces and $P[w]$, $Q[w]$ have matrix representations
\[
P[w] = \begin{pmatrix} p \\ 1-p \end{pmatrix}, \quad Q[w] = \begin{pmatrix} q \\ 1-q \end{pmatrix}
\]
with $p,q \in [0,1]$ and $q \ne 0$, $p$, $1-p$, or $1$.

As diagrams involving all morphisms of $\c$, and not just the priors, are required to commute, still more constraints apply. Although there are exceptions for finely-tuned parameters, it is generically true that if one wishes to find a coarse graining $\gamma: P \Rightarrow Q$ between two causal models, then two outcomes of a measurable space $Pw$ can be identified by the map $\gamma_w$ only when they define the same measure on the codomain for all maps $P\varphi: Pw \to Pw'$, where $\varphi: w \to w'$ is any morphism of $\c$ with domain $w$. The intuition here is that coarse grainings allow us to group outcomes and treat them as a single outcome. But for this to be possible, the outcomes must behave similarly enough to treat as one. Since outcomes now have consequences, we have much higher ability to see differences between them, and hence coarse grainings are far more restrictive for causal models than for measurable spaces.

As embeddings do not identify distinguishable outcomes of the domain, we need not worry about such complications in understanding restrictions on their construction. Nonetheless, if $\epsilon: P \Rightarrow Q$ is an embedding of stochastic causal models, then since the push-forward measure of any prior $P[w]$ along the deterministic stochastic map $\epsilon_w$ must agree with $Q[w]$, any measurable set of $Qw$ not intersecting the image of $\cgst$-embedding must have $Q[w]$-measure zero. Furthermore, for the naturality squares of the morphisms $\varphi: w \to w'$ of $\c$ to commute, each map $Q\varphi$ must behave as $P\varphi$ on the image of $\epsilon_w$. This means that an embedding $\epsilon: P \Rightarrow Q$ forces the priors of $P$ and $Q$ to be the `same' up to sets of measure zero.

\section{Basic constructions in categories of stochastic causal models}

In this final section we continue our characterisation of categories of stochastic causal models by exploring a few universal constructions. In particular, we show that these categories have a terminal object, but no initial object, and in general no products or coproducts either. Again fix a causal theory $\c$. 

\begin{prop}
The functor $T: \c \to \cgst$ sending all objects of $\c$ to the monoidal unit $\ast$ of $\cgst$ and all morphisms of $\c$ to the identity map on $\ast$ is a terminal object in the category $\cgst^\c_{SSM}$ of stochastic causal models of $\c$.
\end{prop}
\begin{proof}
Note first that, since the monoidal product of $\ast$ with itself is again $\ast$, the constant functor $T: \c \to \cgst$ is a well-defined stochastic causal model.

Let $P: \c \to \cgst$ be a stochastic causal model of $\c$. We construct a monoidal natural transformation $\a: P \Rightarrow T$. Then for each $w \in \c$, define $\a_w: Pw \to Tw = \ast$ to be the unique stochastic map $Pw \to \ast$. This exists as $\ast$ is terminal in $\cgst$. Furthermore, from the fact that $\ast$ is terminal in $\cgst$ it is immediate that for each morphism of $\c$ the required naturality square commutes. As these maps to the terminal object are each deterministic, we thus have a well-defined morphism of stochastic causal models. 

By construction it is clear that this is the unique morphism of causal models $P \to T$. This proves the proposition.
\end{proof}

The functor $T$ is an example of what we will call a trivial model. Given a measure space $(X,\S,\mu)$, we define the \emph{trivial model on $(X,\S,\mu)$} to be the functor $T_\mu: \c \to \cgst$ sending each atomic variable $v$ of $\c$ to $X$, and each causal mechanism $[v|pa(v)]$ to the map $T_\mu(pa(v)) \longrightarrow \ast \stackrel{\mu}{\longrightarrow} Tv = X$ assigning to each element of $T_\mu(pa(v))$ the measure $\mu$. This represents the situation in which all the atomic variables are the same random variable with the same prior, and have no causal influence on each other. We shall use these to show the non-existence of an initial object, products, and coproducts.

\begin{prop}
The category $\cgst^\c_{SSM}$ of stochastic causal models of $\c$ has no initial object.
\end{prop}
\begin{proof}
We prove by contradiction. Suppose that $I: \c \to \cgst$ is an initial object of $\cgst^\c_{SSM}$.

Let $(B,\mathcal P(B),\nu)$ be the discrete measure space with two outcomes $\{b_1, b_2\}$ such that the probability of each outcome is one half, and let $T_\nu$ be the trivial model of this space. Note that as $\nu$ has full support, the only measure space $(X,\S,\mu)$ for which there exists a monic deterministic stochastic map $k: X \to B$ such that 
\[
\xymatrix{
X  \ar[rr]^{k} && B \\
\\
& \ast \ar[uul]^{\mu} \ar[uur]_{\nu}
}
\]
commutes is $(B,\mathcal P(B),\nu)$ itself. In this case $k$ must also be an epimorphism in $\cgst$; $k$ is either the identity map, or the map $s:B \to B$ induced by the function sending $b_1$ to $b_2$ and $b_2$ to $b_1$. Thus any map of stochastic causal models $\a: P \to T_\nu$ with codomain $T_\nu$ must be defined objectwise by coarse grainings of $\cgst$, and hence itself be an epimorphism in $\cgst^\c_{SSM}$.

In particular, the unique morphism of stochastic causal models $\theta_{T_\nu}: I \Rightarrow T_\nu$ must be an epimorphism. Since, by uniqueness, the diagram
\[
\xymatrix{
T_\nu  \ar@{=>}[rr]^{\a} && T_\nu \\
\\
& I \ar@{=>}[uul]^{\theta_{T_\nu}} \ar@{=>}[uur]_{\theta_{T_\nu}}\\
}
\]
must commute for any morphism of stochastic models $\a: T_\nu \Rightarrow T_\nu$, this implies that the only such morphism is the identity map. But it is readily observed that defining $\a_v = s: B \to B$ for each atomic variable $v$ of $\c$ gives a monoidal natural transformation $\a: T_\nu \Rightarrow T_\nu$ not equal to the identity. We thus have a contradiction, and so $\cgst^\c_{SSM}$ has no initial object, as claimed.
\end{proof}

\begin{ex}[Two objects which have no product]
We again work with $T_\nu$, where $(B,\mathcal P(B),\nu)$ be the discrete measure space with two outcomes $\{b_1, b_2\}$ such that the probability of each outcome is one half. We will see that the product of $T_\nu$ with itself does not exist. Suppose to the contrary that a product $X$ does exist, with projections $\pi_1, \pi_2: X \Rightarrow T_\nu$. We assume without loss of generality that each $Xv$ is empirical; recall that this means that each point of the set $Xv$ is measurable. 

Given the identity monoidal natural transformation $\id: T_\nu \Rightarrow T_\nu$, there exists a unique monoidal natural transformation $\theta: T_\nu \Rightarrow X$ such that
\[
\xymatrix{
&& T_\nu \\
T_\nu \ar@{=>}[rru]^{\id} \ar@{=>}[rrd]_{\id} \ar@{=>}[rr]^{\theta} && X \ar@{=>}[u]_{\pi_1} \ar@{=>}[d]^{\pi_2} \\
&&  T_\nu
}
\]
commutes. This shows that for each atomic variable $v \in \c$, $Xv$ has an outcome $x_1$ of measure $\frac12$ such that $\theta_v$ is induced by a function mapping $b_1$ to $x_1$, and $\pi_{1v}$ and $\pi_{2v}$ are induced by functions mapping $x_1$ to $b_1$. Similarly, we also have $x_2 \in Xv$ of measure $\frac12$ such that $\theta_v$ is induced by a function mapping $b_2$ to $x_2$, and $\pi_{1v}$ and $\pi_{2v}$ are induced by functions mapping $x_2$ to $b_2$. Note that each $Xv$ then has no other outcomes of positive measure.

Let now $\a: T_\nu \Rightarrow T_\nu$ be the monoidal natural transformation of the previous proof defined by $\a_v = s: B \to B$, where $s$ is induced by the function sending $b_1$ to $b_2$ and $b_2$ to $b_1$. Then there exists a unique monoidal natural transformation $\theta': T_\nu \Rightarrow X$ such that
\[
\xymatrix{
&& T_\nu \\
T_\nu \ar@{=>}[rru]^{\a} \ar@{=>}[rrd]_{\id} \ar@{=>}[rr]^{\theta'} && X \ar@{=>}[u]_{\pi_1} \ar@{=>}[d]^{\pi_2} \\
&&  T_\nu
}
\]
commutes. Now as $\nu(b_1) = \nu(b_2) = \frac12$, for each atomic $v \in \c$ the stochastic map $\theta'_v: B \to Xv$ must then be induced either by the function mapping $b_1$ to $x_1$ and $b_2$ to $x_2$, or by the function mapping $b_1$ to $x_2$ and $b_2$ to $x_1$. Both cases give a contradiction. In the first case, the composite $(\pi_1 \circ \theta')_v$ is then equal to the identity map on $B$, contradicting the definition of $\a$. In the second case, the composite $(\pi_2 \circ \theta')_v$ is equal to $s$, and hence not equal to $\id_B$ as required. 

Thus no product stochastic causal model $T_\nu \times T_\nu$ exists. 
\end{ex}

\begin{ex}[Two objects which have no coproduct]
Let $T$ be the terminal object of $\cgst^\c_{SSM}$, and let $T_\lambda$ be the trivial model on the Lebesgue measure $([0,1],\mathcal B_{[0,1]}, \lambda)$ of the unit interval. We show that these two stochastic causal models have no coproduct in $\cgst^\c_{SSM}$. To this end, suppose that a coproduct $X$ does exist, with injections $i_1: T_\lambda \Rightarrow X$ and $i_2: T_\ast \Rightarrow X$. We again assume without loss of generality that each $Xv$ is empirical.

To show the difficulties in constructing a coproduct, we use the test object $T_\mu$ defined as the trivial model on the measure space $(B,\mathcal P(B),\mu)$ with $B=\{b_1,b_2\}$, $\mu(b_1) =1$ and $\mu(b_2) = 0$. Note that there is a unique map $\b: T \Rightarrow T_\mu$; this is induced on each atomic $v \in \c$ by the function sending the unique point $\ast$ of $Tv$ to $b_1 \in B = T_\mu v$. This is the only such map as, since $\ast$ is a point of measure 1, its image must be a point measure on a point of measure 1. Note also this implies that for each $v \in \c$ the set $Xv$ consists of a point $x_1$ such that a measurable subset of $Xv$ has $X[v]$-measure 1 if $x_1 \in Xv$, and measure 0 otherwise.

Consider now maps $\a: T_\lambda \Rightarrow T_\mu$. These are defined by, for each atomic $v \in \c$, a choice of a Lebesgue measure 0 subset of $[0,1]$. We then may let $\a_v: [0,1] \to B$ be induced by the function mapping each element of this measure zero subset to $b_2$, and then remaining elements to $b_1$. In particular, for each $p \in [0,1]$, let $\a_p: T_\lambda \Rightarrow T_\mu$ be the monoidal natural transformation such that for all atomic $v \in \c$ the map $(\a_p)_v: [0,1] \to B$ is induced by the function mapping $p \to b_2$ and each element of $[0,1] \setminus\{p\}$ to $b_1$. By the universal property of the coproduct, for each such map there exists a unique map $\theta: X \Rightarrow T_\mu$ such that
\[
\xymatrix{
T_\lambda \ar@{=>}[d]_{i_1} \ar@{=>}[rrd]^{\a_p} \\
X \ar@{=>}[rr]^{\theta} && T_\mu \\
T \ar@{=>}[u]^{i_2} \ar@{=>}[rru]_{\b}
}
\]
commutes. This implies that for each atomic $v$ the function inducing the deterministic stochastic map $i_{1v}: [0,1] \to Xv$ does not map $p$ to $x_1$. But this implies that the push-forward measure of $\lambda$ along $i_{1v}$ is the zero measure, contradicting the commutativity of the diagram
\[
\xymatrix{
[0,1]  \ar[rr]^{i_{1v}} && Xv \\
\\
& \ast \ar[uul]^{\lambda} \ar[uur]_{X[v]}\\
}
\]
This shows that $T$ and $T_\lambda$ do not have a coproduct in $\cgst^\c_{SSM}$.
\end{ex}

\phantomsection
\addcontentsline{toc}{chapter}{Further Directions}
\chapter*{Further Directions}

In arriving at this point we have seen that causal theories provide a framework for reasoning about causal relationships, with the morphisms of these categories representing methods of inference, the topology of the string diagrams for these morphisms providing an intuitive visualisation of information flow, and the stochastic models of causal theories slight generalisations of Bayesian networks. 

There are many directions in which this study could be continued. One obvious avenue for further exploration is to continue the work of the previous chapter in the characterisation of categories of stochastic causal models. This should, at the very least, provide additional insight into relationships between Bayesian networks. Although we have seen that products and coproducts do not exist in the category of stochastic causal models, and it is likely similar arguments show other types of limits and colimits do not exist, one suggestion is to examine ideas of families and moduli of stochastic causal models. For this, call Bayesian networks equivalent if there are measure-preserving measurable functions between their joint probability distributions that compose to the identity almost everywhere, and call two stochastic causal models equivalent if their induced Bayesian networks are equivalent. It may then be possible to put some geometric structure on the set of stochastic causal models, and subsequently define a moduli problem. This will perhaps generalise work on the algebraic geometry of Bayesian networks, such as that in \cite{GSS}. One could also explore the relationships between the categories of stochastic causal models of distinct causal theories. Here one might define a functor between such categories if there exists a map of directed graphs between their underlying causal structures.

A weakness of causal theories is that their morphisms only describe predictive inference; reasoning that infers information about causes from their consequences. In general we are interested in other modes of inference too, and extension of the framework to allow discussion of these would make it much more powerful. In the probabilistic case, it can be shown that all conditionals of a joint distribution can be written as morphisms if one can also write Bayesian inverses of the causal conditionals. Given variables $w, w'$, these may be characterised as maps $k: w' \to w$ such that
\[
\begin{aligned}
\begin{tikzpicture}
	\begin{pgfonlayer}{nodelayer}
		\node [style=sq] (0) at (0.75, 0.5) {$w'||w$};
		\node [style=none] (1) at (0, -0.25) {};
		\node [style=none] (2) at (0.75, 0.5) {};
		\node [style=none] (3) at (-0.75, 1.25) {};
		\node [style=tri] (4) at (0, -0.95) {\phantom{'}$w$\phantom{'}};
		\node [style=none] (5) at (-0.75, 0.5) {};
		\node [style=none] (6) at (0.75, 1.25) {};
		\node [style=none] (8) at (0.75, 1.51) {$w'$};
		\node [style=none] (9) at (-0.75, 1.5) {$w$};		
	\end{pgfonlayer}
	\begin{pgfonlayer}{edgelayer}
		\draw (1.center) to (4);
		\draw (0) to (2.center);
		\draw (3.center) to (5.center);
		\draw (6.center) to (0);
		\draw [in=-90, out=180, looseness=0.75] (1.center) to (5.center);
		\draw [in=-90, out=0] (1.center) to (2.center);
	\end{pgfonlayer}
\end{tikzpicture}
\end{aligned}
=
\begin{aligned}
\begin{tikzpicture}
	\begin{pgfonlayer}{nodelayer}
		\node [style=none] (0) at (0.75, 1.25) {};
		\node [style=none] (1) at (0, 0) {};
		\node [style=none] (2) at (0.75, 0.75) {};
		\node [style=sq] (3) at (-0.75, 0.75) {$k$};
		\node [style=tri] (4) at (0, -0.95) {\phantom{'}$w$\phantom{'}};
		\node [style=none] (5) at (-0.75, 0.75) {};
		\node [style=none] (6) at (-0.75, 1.25) {};
		\node [style=sq] (7) at (0, -0.35) {$w'||w$};
		\node [style=none] (8) at (0.75, 1.51) {$w'$};
		\node [style=none] (9) at (-0.75, 1.5) {$w$};		
	\end{pgfonlayer}
	\begin{pgfonlayer}{edgelayer}
		\draw (0.center) to (2.center);
		\draw (3) to (5.center);
		\draw [in=-90, out=180, looseness=0.75] (1.center) to (5.center);
		\draw [in=-90, out=0] (1.center) to (2.center);
		\draw (6.center) to (3);
		\draw (4) to (7);
		\draw (7) to (1.center);
	\end{pgfonlayer}
\end{tikzpicture}
\end{aligned}
\]
Methods for constructing such maps often run into issues of uniqueness on outcomes of measure zero in the prior. While in Coecke and Spekkens \cite{CS} give a method for realising the Bayesian inverse of a finite stochastic map as transposition with respect to a compactness structure in $\mathbf{Mat}$ when the prior is of full support, and Abramsky, Blute, and Panangaden \cite{ABP} give a category, similar to $\stoch$, in which Bayesian inversion may be viewed as a dagger-functor, work remains to be done to merge these ideas with those presented here.

Another topic deserving investigation is suggested by the fact that, although a joint probability distribution is compatible with a causal structure if it satisfies the required set of conditional independence relations, not every possible combination of conditional independence relations of a set of random variables can be represented by a causal structure. Indeed, the number of combinations of conditional independence relations grows exponentially in number of atomic variables, while the number of causal structures grows only quadratically. It is possible that the richer structure of categories may allow us to define causal theories more general than those arising from causal structures, such that models in some category are those that satisfy precisely a given set of conditional independencies, and no more.

Finally, Sections 4.2 and 4.3 suggest their own further lines of investigation. While we have focussed on models in $\stoch$ and its subcategories, it would also be worthwhile to understand more thoroughly models in $\rel$, and models in the category $\mathbf{Hilb}$ of Hilbert spaces and linear maps may be interesting from the perspective of quantum theory. It would also be interesting to find further examples of applications of the graphical languages for causal theories. One option is to look at representations of algorithms used on Bayesian networks, such as Gibbs sampling in Bayesian networks \cite{Hr}.

\phantomsection
\addcontentsline{toc}{chapter}{Bibliography}
\bibliography{refs}        
\bibliographystyle{plain}  

\end{document}